\documentclass{amsart}
\usepackage[all]{xy}
\usepackage{xspace}
\usepackage{amsmath}
\usepackage{amstext}
\usepackage{amsfonts}
\usepackage[mathscr]{euscript}
\usepackage{amscd}
\usepackage{latexsym}
\usepackage{amssymb}
\usepackage{tikz}
\usepackage{tabulary}
\theoremstyle{plain}
\swapnumbers
    \newtheorem{thm}{Theorem}[section]
    \newtheorem{prop}[thm]{Proposition}
    \newtheorem{lemma}[thm]{Lemma}
    \newtheorem{corollary}[thm]{Corollary}
    
    \newtheorem*{thma}{Theorem A}
    \newtheorem*{thmb}{Theorem B}
    
    \newtheorem*{propa}{Proposition}
    \newtheorem{subsec}[thm]{}
\theoremstyle{definition}
    
    \newtheorem{defn}[thm]{Definition}
    \newtheorem{example}[thm]{Example}

    \newtheorem{notation}[thm]{Notation}
    
    \newtheorem{remark}[thm]{Remark}
    \newtheorem{ack}[thm]{Acknowledgements}
\theoremstyle{remark}

\newenvironment{myeq}[1][]
{\stepcounter{thm}\begin{equation}\tag{\thethm}{#1}}
{\end{equation}}
\newcommand{\supsect}[2]
{\vspace*{-5mm}\quad\\\begin{center}\textbf{{#1}}\vsm.~~~~\textbf{{#2}}\end{center}}

\newcommand{\mysdg}[2][]{\myeq[#1]\xymatrix@R=10pt@C=15pt{#2}}
\newcommand{\mydiagram}[2][]
{\stepcounter{thm}\begin{equation}
     \tag{\thethm}{#1}\vcenter{\xymatrix{#2}}\end{equation}}
\newcommand{\mysdiag}[2][]
{\stepcounter{thm}\begin{equation}
     \tag{\thethm}{#1}\vcenter{\xymatrix@R=20pt@C=15pt{#2}}\end{equation}}
\newcommand{\mydiagrm}[2][]
{\stepcounter{thm}\begin{equation}
    \tag{\thethm}{#1}\vcenter{\entrymodifiers={+++[o]}\xymatrix@R=25pt@C=25pt{#2}}\end{equation}}
\newenvironment{mysubsection}[2][]
{\begin{subsec}\begin{upshape}\begin{bfseries}{#2.}
\end{bfseries}{#1}}
{\end{upshape}\end{subsec}}
\newenvironment{mysubsect}[2][]
{\begin{subsec}\begin{upshape}\begin{bfseries}{#2\vsn.}
\end{bfseries}{#1}}
{\end{upshape}\end{subsec}}
\newcommand{\sect}{\setcounter{thm}{0}\section}
\newcommand{\wh}{\ -- \ }
\newcommand{\wwh}{-- \ }
\newcommand{\w}[2][ ]{\ \ensuremath{#2}{#1}\ }
\newcommand{\ww}[1]{\ \ensuremath{#1}}
\newcommand{\www}[2][ ]{\ensuremath{#2}{#1}\ }
\newcommand{\wwb}[1]{\ \ensuremath{(#1)}-}
\newcommand{\wb}[2][ ]{\ (\ensuremath{#2}){#1}\ }
\newcommand{\wref}[2][ ]{\ (\ref{#2}){#1}\ }
\newcommand{\wwref}[2][ ]{\ (\ref{#2}){#1}}
\newcommand{\hsp}{\hspace*{7 mm}}
\newcommand{\hs}{\hspace*{4 mm}}
\newcommand{\hsm}{\hspace*{2 mm}}
\newcommand{\vsn}{\vspace{1 mm}}
\newcommand{\vs}{\vspace{7 mm}}
\newcommand{\vsm}{\vspace{3 mm}}
\newcommand{\col}{\colon\thinspace}
\newcommand{\ssr}{\!\!\!\!\!\!\!\!\!\!\!\!\!\!}
\newcommand{\adj}[2]{\hspace*{3mm}\substack{{#1}\\ {\raisebox{-.5ex}{$\longrightarrow$} \hspace*{-7mm}\raisebox{.5ex}{$\longleftarrow$}}\\ {#2}}\hspace*{3mm}}
\newcommand{\hra}{\hookrightarrow}
\newcommand{\lto}{\longrightarrow}
\newcommand{\lra}[1]{\langle{#1}\rangle}
\newcommand{\rws}[1]{\overset{#1}{\rightarrow}}
\newcommand{\supar}[1]{\overset{#1}{-\!\!-\!\!\!\rightarrow}}
\newcommand{\lws}[1]{\overset{#1}{\leftarrow}}
\newcommand{\suparle}[1]{\overset{#1}{\leftarrow\!\!\!-\!\!-}}

\newcommand{\Rw}{\Rightarrow}
\newcommand{\up}[1]{^{(#1)}}
\newcommand{\lo}[1]{_{(#1)}}
\newcommand{\xrw}{\supar}
\newcommand{\xlw}{\suparle}
\newcommand{\tiund}[1]{{\times}_{#1}}
\newcommand{\ovl}[1]{\overline{#1}}
\newcommand{\bsim}{/\!\!\sim}

\newcommand{\pro}[3]{#1\tiund{#2}\overset{#3}{\cdots}\tiund{#2}#1}
\newcommand{\pros}[3]{#1\tiund{#2}\overset{#3}{\cdots}{#1}}
\newcommand{\tens}[2]{#1\,\tiund{#2}\,#1}

\newcommand{\doublerightarrow}{%
  \begin{tikzpicture}
  \node (a) at (0,0){\!\!\!};
  \node (b) at (1,0){\!\!\!};
  \path[->]
  ([yshift=3pt]a.east)edge node[below]{}([yshift=3pt]b.west)
  ([yshift=-3pt]a.east)edge node[above]{}([yshift=-3pt]b.west);
  \end{tikzpicture}
}
\newcommand{\toto}{\,{\doublerightarrow}\,}

\newcommand{\hxi}{\hat{\xi}}
\newcommand{\hphi}{\hat{\phi}}
\newcommand{\hpsi}{\hat{\psi}}
\newcommand{\var}{\varepsilon}
\newcommand{\wvar}{\widehat{\var}}
\newcommand{\ovar}{\ovl{\var}}
\newcommand{\seg}[1]{\mu\sb{#1}}
\newcommand{\iseg}[1]{\widehat{\mu}\sb{#1}}
\newcommand{\pis}{\pi_{\ast}}
\newcommand{\hp}[1]{\hat{\pi}\sb{#1}}
\newcommand{\hpv}[1]{\hat{\pi}^{v}_{#1}}
\newcommand{\hpi}{\hp{1}}
\newcommand{\hpu}[1]{\hpi\up{#1}}
\newcommand{\tX}{\widetilde{X}}
\newcommand{\tY}{\widetilde{Y}}
\newcommand{\bPz}[1]{\Pi\sb{0}\up{#1}}
\newcommand{\oPz}{\ovl{\Pi}\sb{0}}
\newcommand{\opz}[1]{\ovl{\pi}\sb{0}\up{#1}}
\newcommand{\ata}{A\tiund{B}A}
\newcommand{\Aut}{\operatorname{Aut}}

\newcommand{\cons}{\operatorname{c}\!}
\newcommand{\cat}{\operatorname{cat}}

\newcommand{\csk}[1]{\operatorname{csk}_{#1}}
\newcommand{\Disc}{\operatorname{Disc}}
\newcommand{\diz}{\Disc_{0}}

\newcommand{\Dec}{\operatorname{Dec}}
\newcommand{\Decp}{\Dec'}
\newcommand{\Dex}{\Dec\,X}

\newcommand{\Diag}{\operatorname{Diag}}
\newcommand{\Dlo}[1]{\Diag\lo{#1}}
\newcommand{\oDlo}[1]{\ovl{\Diag}\lo{#1}}
\newcommand{\Dup}[1]{\Diag\up{#1}}

\newcommand{\Glo}[1]{\gamma\lo{#1}}

\newcommand{\hd}{\operatorname{hd}}
\newcommand{\ho}{\operatorname{ho}}
\newcommand{\Hom}{\operatorname{Hom}}
\newcommand{\Id}{\operatorname{Id}}

\newcommand{\Obj}{\operatorname{Obj}\,}
\newcommand{\op}{\sp{\operatorname{op}}}
\newcommand{\ordin}{\operatorname{or}}
\newcommand{\ord}{\operatorname{Or}}
\newcommand{\ost}{\ord}
\newcommand{\osl}[1]{\ost\lo{#1}}
\newcommand{\ovt}[2]{\ovl{\ord}\sp{({#1})}\lo{#2}}
\newcommand{\pr}{\operatorname{pr}}
\newcommand{\ps}{\operatorname{ps}}
\newcommand{\sk}[1]{\operatorname{sk}_{#1}}
\newcommand{\red}{\sp{\operatorname{red}}}
\newcommand{\wg}{\operatorname{wg}}
\newcommand{\Twg}[1]{T\sp{\wg}\lo{#1}}
\newcommand{\Tps}[1]{T\sp{\ps}\lo{#1}}
\newcommand{\Tm}{\operatorname{Tm}}
\newcommand{\TTm}[1]{T\sp{\Tm}\lo{#1}}
\newcommand{\map}{\operatorname{map}}

\newcommand{\hy}[2]{{#1}({#2})}

\newcommand{\hyp}[2]{${#1}$-${#2}$}
\newcommand{\Cat}{\mbox{\sf Cat}}

\newcommand{\Catc}{\Cat(\cC)}
\newcommand{\Gp}{\mbox{\sf Gp}}
\newcommand{\Gpd}{\mbox{\sf Gpd}}
\newcommand{\Ghd}[1]{\Gpd\sb{\hd}\sp{#1}}
\newcommand{\Gpt}[1]{\Gpd\sb{\wg}\sp{#1}}
\newcommand{\Gptk}[1]{\Gpd\sb{\wg}\up{#1}}

\newcommand{\Po}[1]{P\sp{#1}}
\newcommand{\PsGpd}{\mbox{\sf PsGpd}}
\newcommand{\PsG}[1]{\PsGpd\sb{\wg}\sp{#1}}
\newcommand{\PhG}[1]{\PsGpd\sb{\hd}\sp{#1}}
\newcommand{\PsGk}[1]{\PsGpd\sb{\wg}\up{#1}}
\newcommand{\Set}{\mbox{\sf Set}}
\newcommand{\Tam}[1]{{\mbox{\sf Tam}}\sp{#1}}
\newcommand{\tTam}[1]{{\underline{\mbox{\sf Tam}}}\sp{#1}}
\newcommand{\Top}{\mbox{\sf Top}}
\newcommand{\Topa}{\Top\sb{\ast}}

\newcommand{\Track}[1]{{\mbox{\sf Track}}\sb{#1}}

\newcommand{\PS}[1]{\mbox{\sf P}\sp{#1}\sS}
\newcommand{\PT}[1]{{\mbox{\sf P}}\sp{#1}\Top}
\newcommand{\Pud}[2]{{\mbox{\sf P}}\sp{#1}\sb{#2}\Topa}
\newcommand{\VC}{\cV$-$\Cat}

\newcommand{\bDelta}{\Delta}
\newcommand{\Dop}{\bDelta\op}
\newcommand{\sC}[1]{[\Dop,\,{#1}]}
\newcommand{\sS}{\sC{\Set}}
\newcommand{\Dox}[1]{\bDelta\!^{{#1}\op}}
\newcommand{\sCx}[2]{[\Dox{#1},\,{#2}]}
\newcommand{\Sx}[1]{\sCx{#1}{\Set}}
\newcommand{\Dnop}{\Dox{n}}
\newcommand{\Snx}{\Sx{n}}

\newcommand{\va}{a}

\newcommand{\vb}{b}

\newcommand{\bk}{[k]}
\newcommand{\bm}{[m]}
\newcommand{\bn}{[n]}
\newcommand{\bze}{[0]}
\newcommand{\bon}{[1]}
\newcommand{\Gd}{G\sb{\bullet}}
\newcommand{\Pd}{P\sb{\bullet}}
\newcommand{\Wd}{W\sb{\bullet}}
\newcommand{\Wu}{W\sp{\bullet}}

\newcommand{\Xd}{X\sb{\bullet}}
\newcommand{\Xdd}{X\sb{\bullet\bullet}}

\newcommand{\Yd}{Y\sb{\bullet}}

\newcommand{\Zdd}[1]{Z\sb{\bullet\bullet{#1}}}
\newcommand{\FFp}{{\mathbb F}\sb{p}}
\newcommand{\bP}{{\mathbb P}}
\newcommand{\bN}{{\mathbb N}}

\newcommand{\cC}{\EuScript C}
\newcommand{\cD}{\EuScript D}

\newcommand{\Llo}[1]{L\sb{#1}}
\newcommand{\Ld}{\Llo{\bullet}}
\newcommand{\cN}{N}
\newcommand{\dN}{d\!N}
\newcommand{\diN}{B}
\newcommand{\odiN}[1]{\ovl{B}\up{#1}}
\newcommand{\TamR}{B}
\newcommand{\FTm}[1]{F\sb{\Tm}\sp{#1}}

\newcommand{\Nup}[1]{N\up{#1}}
\newcommand{\Nlo}[1]{N\lo{#1}}
\newcommand{\oNl}[1]{\ovl{N}\lo{#1}}
\newcommand{\cP}[1]{\mathcal{P}[{#1}]}
\newcommand{\Pup}[1]{P\up{#1}}
\newcommand{\Plo}[1]{P\lo{#1}}
\newcommand{\oPl}[1]{\ovl{P}\lo{#1}}
\newcommand{\Qlo}[1]{Q\lo{#1}}
\newcommand{\hQ}{\widehat{Q}}
\newcommand{\hQlo}[1]{\hQ\lo{#1}}
\newcommand{\oQl}[2]{\ovl{Q}\up{#1}\lo{#2}}
\newcommand{\cS}{\mathcal{S}}
\newcommand{\cV}{\EuScript V}
\newcommand{\cW}{\mathcal{W}}
\newcommand{\Wlo}[1]{\cW\lo{#1}}
\newcommand{\gup}[1]{\gamma\up{#1}}
\newcommand{\hgup}[1]{\ovl{\gamma}\up{#1}}

\title {Segal-type algebraic models of $n$-types}
\author{David Blanc}
\address{Department of Mathematics\\ University of Haifa\\ 31905 Haifa\\ Israel}
\email{blanc@math.haifa.ac.il}
\author{Simona Paoli}
\address{Department of Mathematics\\ University of Leicester\\ Leicester
  LE1 7RH, UK}
\email{sp424@le.ac.uk}
\date{\today}
\subjclass{55S45; 18G50, 18B40}
\keywords{$n$-type, $n$-fold groupoid, weakly globular, algebraic model}
\begin{document}
\begin{abstract}
For each \w[,]{n\geq 1} we introduce two new Segal-type models
of $n$-types of topological spaces: \emph{weakly globular $n$-fold
groupoids}, and a lax version of these.
We show that any $n$-type can be represented up to homotopy
by such models via an explicit \emph{algebraic fundamental $n$-fold
groupoid} functor. We compare these models to Tamsamani's
weak $n$-groupoids, and extract from them a model for
\wwb{k-1}connected $n$-types.
\end{abstract}

\maketitle

%
%
\section{Introduction and summary}
\label{cover}

Many homotopy invariants of a topological space $T$, such as its homotopy, homology,
and cohomology groups, are graded by dimension, so that we do not need to know
all of $T$ to determine \w[,]{\pi\sb{n}T} \w[,]{H\sb{n}T} or \w[,]{H\sp{n}(T;G)} but
only a skeleton or Postnikov section of $T$. Thus, for many purposes
a good first approximation to homotopy theory is the study of $n$-\emph{types}:
that is, spaces $T$ whose homotopy groups \w{\pi\sb{k}(T,t\sb{0})}
vanish for \w[.]{k>n}

One advantage of such approximations is that they have algebraic models:
the classical example is the homotopy category of connected $1$-types, which
is equivalent to the category of groups. More generally, all $1$-types
are modelled by groupoids via the fundamental groupoid functor
\w[.]{\hpi\col \Top\to\Gpd}

The arrows of \w{\hpi T} are homotopy classes of paths, so higher order
approximations should encode higher homotopies (see \cite{GrothP}), and thus
involve higher categorical structures.

Many such structures have been shown to model the homotopy category
\w{\ho\PT{n}} of $n$-types of topological spaces:
in the path-connected case, these include the \ww{\cat^{n}}-groups of
\cite{LodaS}, the crossed $n$-cubes of \cite{ESteH} and \cite{TPorN},
the $n$-hyper-crossed complexes of \cite{CCegG}, and the weakly globular
\ww{\cat^{n}}-groups of \cite{PaolW}.
Special models exist for \w[,]{n=2,3} starting with the crossed modules of
\cite{MWhitT}, and including the homotopy double groupoids of
\cite{BHKPortH}, the homotopy bigroupoids of \cite{HKKieHB},
the strict $2$-groupoids of \cite{MSvenA}, the double groupoids of
\cite{CHRemeD}, the double groupoids with connections of
\cite{BSpenD}, the Gray groupoids of \cite{LeroT,BergD,JTierA}, and
the quadratic modules of \cite{BauCH}.
In the general case, such models include Batanin's higher groupoids
(see \cite{BataMG,CisiB}), the $n$-hypergroupoids of \cite{GlennR},
and Tamsamani's weak $n$-groupoids (see \cite{TamsN,SimpH})\vsm.

In this paper we discuss three algebraically defined categories of
Segal-type objects which can be used to model all $n$-types of
topological spaces. All three are full subcategories of the category
\w{\sCx{n-1}{\Gpd}} of $(n-1)$-fold simplicial objects in groupoids.

\begin{mysubsection}{The three models}
\label{nthreem}
The first is the known category \w{\Tam{n}} of \emph{Tamsamani weak $n$-groupoids}.
The second is a new category \w{\Gpt{n}} of \emph{weakly globular $n$-fold
groupoids}. This is a full subcategory of the category \w{\Gpd\sp{n}} of
$n$-fold groupoids (iteratively defined as groupoids internal to
\w[).]{\Gpd\sp{n-1}} The third is another new category \w{\PsG{n}} of
\emph{weakly globular pseudo $n$-fold groupoids}.

To grasp the idea behind these notions, it is useful to consider another
higher categorical structure which embeds in all three of the above,
the category \w{n\mbox{-}\Gpd} of strict $n$-groupoids (iteratively
defined as groupoids enriched in \w[).]{(n-1)\mbox{-}\Gpd}

There are full and faithful inclusions:
\mydiagrm[\label{overvdiag1}]{
     & \PsG{n} &\\
    \Tam{n} \ar@{^{(}->}[ur] && \Gpt{n} \ar@{_{(}->}[ul]\\
& n\mbox{-}\Gpd \ar@{^{(}->}[ur] \ar@{_{(}->}[ul] &
    }

The category \w{n\mbox{-}\Gpd} admits a multi-simplicial description
as the full subcategory of those $(n-1)$-fold simplicial objects
\w{X\in\sCx{n-1}{\Gpd}} satisfying:

\begin{itemize}
\item [(i)] \www{X\sb{0}\up{1}\in\sCx{n-2}{\Gpd}} and
\w{X\up{1,\dotsc,r+1}\sb{1\underset{r}{\cdots}10}\in[\Dox{n-r-2},\,\Gpd]}
are discrete \wh that is, constant multi-simplicial sets \wh for all
\w[.]{1\leq r\leq n-2} Here we use the notation of \S \ref{nsimp}(b).
\item [(ii)] The Segal maps (see \S \ref{dsegal} below) in all directions
are isomorphisms.
\end{itemize}

In addition, we require that after applying \w{\pi_0\col \Gpd\to\Set} in each
simplicial dimension we obtain a strict $(n-1)$-groupoid.

The sets in (i) corresponds to the set of $r$-cells \wb{1\leq r\leq n-2}
in the strict $n$-groupoid. By (ii), their composition is associative
and unital.

Condition (i) is also called the \emph{globularity condition}, since it
determines the globular shape of the cells in a strict $n$-groupoid. For
instance, when \w{n=2} we can picture the $2$-cells as globes:
$$
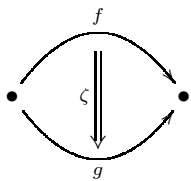
\xymatrix@R=13pt{
& \ar @{=>}[dd]_{\zeta} & \\
\bullet  \ar@/^{2.0pc}/[rr]^{f} \ar@/_{2.0pc}/[rr]_{g} && \bullet\\
&&
}
$$

Although strict $n$-groupoids have applications in homotopy theory,
especially in their equivalent form of crossed $n$-complexes
(cf.\ \cite{BHSiveN}), they cannot model all $n$-types of topological
spaces (see \cite[\S 5]{SimpHT} for a counterexample in dimension $3$).
Therefore, we must relax the strict structure in order to recover all
$n$-types.

We consider three approaches to this:

\begin{itemize}
\item [(a)] In the first approach, we preserve condition (i) and relax
(ii), by allowing the Segal maps to be suitable iteratively-defined
equivalences. The composition of cells is then no longer
strictly associative and unital. This leads to the category \w{\Tam{n}} of
Tamsamani weak $n$-groupoids (Definition \ref{dtamn}).
\item [(b)] In this paper we offer a second approach, in which condition (ii)
is preserved, while (i) is replaced by weak globularity, so that the
multi-simplicial objects in (i) are no longer required to be discrete, but
only ``homotopically discrete'' (in a way that allows iteration). This
leads to the category \w{\Gpt{n}} of weakly globular $n$-fold groupoids
(Definition \ref{dntng}).
\item [(c)] We also describe a third approach, in which both (i) and (ii)
are relaxed. This yields the category \w{\PsG{n}} of weakly globular
pseudo $n$-fold groupoids (Definition \ref{dwgpng}).
\end{itemize}

Moreover, we have a \emph{realization} functor
\w{\diN\col \PsG{n}\to\Top} \wwh the composite:
\medskip\noindent
\begin{myeq}\label{eqrealize}
\PsG{n}~\hra~\sCx{n-1}{\Gpd}~\xrw{N}~\sCx{n}{\Set}~\xrw{\Dlo{n}}~\sC{\Set}~
\xrw{\|-\|}~\Top~,
\medskip
\end{myeq}

\noindent where \w{\Dlo{n}} is the $n$-fold diagonal.
The same is therefore true of the two subcategories
\w{\Tam{n}} and \w[.]{\Gpt{n}} In all three categories, maps which induce
weak homotopy equivalences on realizations are called
\emph{geometric weak equivalences}\vsm .

The precise definitions of these three categories appear as cited above; here
we will only highlight some key features common to all three\vsm:

\noindent 1)~The construction of each category is by induction, starting in
all three cases from the category of groupoids for \w[.]{n=1}
A weakly globular pseudo $n$-fold groupoid is in particular
a simplicial object $X$ in \w[,]{\PsG{n-1}} (and similarly for the other
two categories)\vsm.

\noindent 2)~Moreover, \w{X\sb{0}} is a homotopically discrete
weakly globular pseudo \wwb{n-1}fold groupoid (Definition \ref{dphdng})
and similarly for \w[,]{\Gpt{n}} while in the case of a Tamsamani weak
$n$-groupoid \w[,]{X\in\sC{\Tam{n-1}}} \w{X\sb{0}} is actually
discrete\vsm .

\noindent 3)~Since \w{\PsG{n}} is a subcategory of \w[,]{\sCx{n-1}{\Gpd}} we can
apply the functor \w{\pi\sb{0}} to each groupoid of any weakly globular
pseudo $n$-fold groupoid $X$ to obtain \w[.]{\opz{n}X\in\sCx{n-1}{\Set}}
In each of our three categories the functor \w{\opz{n}} lifts to functors
\begin{equation*}
\begin{split}
\bPz{n}\col \PsG{n}~&\to~\PsG{n-1}~,\hsp \bPz{n}\col \Gpt{n}~\to~\Gpt{n-1},\\
\text{and} & \hsp \bPz{n}\col \Tam{n}~\to~\Tam{n-1}~.
\end{split}
\end{equation*}

These serve as algebraic \wwb{n-1}Postnikov section functors, so we have a
natural Postnikov tower:
$$
\PsG{n}~\supar{\Pi_0\up{n}}~\PsG{n-1}~\supar{\Pi_0\up{n-1}}~\cdots~\Gpd~
\supar{\pi_0}~\Set~.
$$
\noindent and similarly for the other two categories\vsm .

\noindent 4)~In all three categories, let
\w{\gamma\col X\sb{0}\to X\sb{0}\sp{d}} denote the weak equivalence from
the homotopically discrete object \w{X\sb{0}} to its discretization
(so $\gamma$ is the identity for \w[).]{X\in\Tam{n}}
For each \w{k\geq 2} the composite of the maps
\begin{myeq}\label{eqindsegal}
X\sb{k}~\xrw{\seg{k}}~\pro{X_1}{X_0}{k}~\xrw{\gamma\sb{\ast}}~
\pro{X_1}{X\sb{0}\sp{d}}{k}
\end{myeq}
\noindent (cf. \S \ref{dsegal}) is called the $k$-th
\emph{induced Segal map}. We require that these maps be geometric
weak equivalences.

When \w[,]{X\in\Tam{n}} the second map in \wref{eqindsegal}
is the identity, while when \w[,]{X\in\Gpt{n}} the first map is an
isomorphism.
\end{mysubsection}

\begin{mysubsection}{Main results}
\label{sbasicf}
The process of discretizing the homotopically discrete sub-objects in
\w{\Gpt{n}} and \w{\PsG{n}} gives rise to discretization functors \w{D\sb{n}}
making the following diagram commute
\mydiagrm[\label{overvdiag2}]{
     & \PsG{n} \ar_{D_n}[dl]&\\
    \Tam{n}  && \Gpt{n}. \ar^{D_n}[ll] \ar@{_{(}->}[ul]
  \vsm  }

All three categories \w[,]{\PsG{n}} \w[,]{\Gpt{n}} and \w{\Tam{n}} share
some useful features\vsm:

First, the realization functor \w{\diN\col \PsG{n}\to\Top} actually lands in
the category \w{\PT{n}} of $n$-types, so the same is true of \w{\Gpt{n}} and
\w[\vsm .]{\Tam{n}}

Furthermore, all three models have \emph{algebraic homotopy groups}
\w{\omega_{k}(X,x)} (cf.\ \S \ref{sahgp}), which allow one to extract
\w{\pi\sb{k}\diN X} directly from the model $X$. In addition, we have
higher-dimensional analogues of the categorical notion of
an equivalence of groupoids. Together, these two notions allow to define
\emph{algebraic weak equivalences} in each of the categories, and show that
these are the same as the geometric weak equivalences
(see Corollary \ref{chtpygp}, [T, \S \ref{cwgpng}], \S \ref{sahgpps},
and \ref{rhtpyps}). Thus each of these models is entirely algebraic\vsm .

Our main results are as follows\vsm:

\begin{thma}
For each \w[:]{n\geq 1}
\begin{enumerate}
\renewcommand{\labelenumi}{(\alph{enumi})~}
\item The functor \w{\hQlo{n}} induces a faithful embedding
$$
\ho\PT{n}~\hra~\ho\Gpt{n}~,
$$
\noindent so for each \w{T\in\PT{n}} there is an isomorphism in \w{\ho\PT{n}}
between \w{T} and \w[.]{\diN\hQlo{n}T}
\item There is a functor
\w{\bPz{n}\col \Gpt{n}\to\Gpt{n-1}} with a natural isomorphism
\w[,]{\bPz{n}\hQlo{n}\cong\hQlo{n-1}} so we can extract the
model for the \wwb{n-1}st Postnikov section \w{\Po{n-1}T} from
\w{\hQlo{n}T} algebraically.
\item There are \emph{algebraic homotopy group} functors
\w{\omega_{k}\col \Gpt{n}\to\Gp} such that
\w{\pi_{k}(\diN G;x_{0})~\cong~\omega_{k}(G;x_{0})} \wb[.]{0\leq k\leq n}
\end{enumerate}
\end{thma}
\noindent [See Theorem \ref{teqcat}, Proposition \ref{pnequiv}, and
Theorem \ref{thdnt}]\vsm .

\begin{thmb}
The functors \w{\hQlo{n}} and \w{\diN} induce equivalences of categories
$$
\ho\PT{n}~\approx~\ho\PsG{n}~.
$$
\end{thmb}
\noindent [See Theorem \ref{tpswgntype}]\vsm .

Furthermore, every object of \w{\PsG{n}} is weakly equivalent through a
zig-zag to an object of \w{\Gpt{n}} as well as to an object of
\w{\Tam{n}} (see Remark \ref{rzigzag}). Thus we can regard \w{\Tam{n}}
and \w{\Gpt{n}} as two different types of partial strictifications of the
category \w{\PsG{n}} which preserve the homotopy type. The passage from
\w{\PsG{n}} to \w{\Tam{n}} strictifies the globularity condition, while the
passage from \w{\PsG{n}} to \w{\Gpt{n}} strictifies the Segal maps.

The fundamental $n$-fold groupoid functor \w{\hQlo{n}\col \Top\to\Gpt{n}}
provides an explicit form for the algebraic model of an $n$-type. This
is a desirable feature of an algebraic model, especially in view of
applications.

 In Subsection \ref{cappl}.A we discuss an application to
the modelling of \wwb{k-1}connected $n$-types. For this purpose we
identify suitable subcategories \w{\PsG{(n,k)}} and \w{\Gpt{(n,k)}}
of \w{\PsG{n}} and \w{\Gpt{n}} respectively, which are algebraic models of
\wwb{k-1}connected $n$-types, and we also establish a connection with
iterated loop spaces (Proposition \ref{pkloop}).
\end{mysubsection}

\vspace{5mm}
\begin{mysubsect}{Organization of the paper}
\label{nmainst}

In Section \ref{cfng} we describe the construction of the
\emph{fundamental $n$-fold grou\-poid} functor \w[:]{\hQlo{n}} to obtain
a multi-simplicial algebraic model from a space, we first take a
fibrant simplicial set model using the singular functor
\w[.]{\cS\col \Top\to\sS} We can associate to any simplicial set
$X$ an ``$n$-fold resolution'' \w[,]{\osl{n} X} which is an object of
\w{\Snx} representing the same homotopy type (Lemma \ref{lores}).
We then obtain an $n$-fold groupoid by applying the left adjoint
\w{\Plo{n}\col \Snx\to \Gpd\sp{n}} to the $n$-fold nerve
\w[.]{N\lo{n}\col \Gpd\sp{n}\to\Snx} Thus \w{\hQlo{n}} is the composite:
$$
\Top~\supar{\cS}~\sS~\supar{\osl{n}}~\Snx~\supar{\Plo{n}}~\Gpd\sp{n}
$$
\noindent (cf.\ Definition \ref{dladjoint}).

For a general $n$-fold simplicial set $Y$, \w{\Plo{n}Y} does not
have a simple and explicitly computable expression. However, we show
that the fibrancy of \w{\cS T} induces a property of
\w[,]{\osl{n}\cS T} which we call \wwb{n,2}\emph{fibrancy}
(see Definition \ref{dfibrant} and Proposition \ref{pntfibor}).
We then show that to apply \w{\Plo{n}} to an \wwb{n,2}fibrant
$n$-simplicial set, we need only apply the usual fundamental groupoid
functor in each of the \w{n-1} simplicial directions.
We thus have
$$
\hQlo{n}T~=~\hpu{1}\hpu{2}\ldots\hpu{n}\osl{n} \cS T
$$
\noindent (cf.\ Theorem \ref{tadjoint}).

In Section \ref{cntng}, we describe certain features of
those $n$-fold groupoids which are in the image of the functor
\w{\hQlo{n}} (and thus will be used to represent $n$-types).
These are encoded in the notion of \emph{weakly globular $n$-fold
groupoids}. As explained in \S \ref{nthreem}, we first identify a
suitable subcategory of \emph{homotopically discrete} objects
(Subsection \ref{cntng}.A), which are needed for the weak globularity
condition in the definition of weakly globular $n$-fold groupoid
(Subsection \ref{cntng}.B).

In Section \ref{cnt} we show that  the $n$-Postnikov section
\w{\Po{n}T} and \w{\diN\hQlo{n}T} have the same homotopy type
(Proposition \ref{pnequiv}), so \w{\Gpt{n}} represents all $n$-types.
In Subsection \ref{cnt}.A we show that the realization of a weakly
globular $n$-fold groupoid is an $n$-type (an alternative proof using
a comparison with Tamsamani's model is given in Section \ref{ctmng}).
Subsection \ref{cnt}.B provides a new iterative description of the
fundamental $n$-fold groupoid functor \w[.]{\Qlo{n}}
This is used in Proposition \ref{pnequiv} of Subsection \ref{cnt}.C,
where we show that the functor \w{\hQlo{n}} lands in the category
\w[.]{\Gpt{n}} This leads to one of our main results, Theorem \ref{teqcat},
saying that $\diN$ and \w{\hQlo{n}} induce functors
\begin{myeq}\label{overveq1}
\ho\PT{n}~\adj{\diN}{\hQlo{n}}~\ho\Gpt{n}~,
\end{myeq}
\noindent with \w[.]{\diN\circ\hQlo{n}\cong\Id}

In Section \ref{ctmng} we provide an equivalent definition of Tamsamani's
weak $n$-groupoids (see Subsection \ref{ctmng}.A), and in
Subsection \ref{ctmng}.B we construct a \emph{discretization} functor
$$
D_n\col \Gpt{n}\to \Tam{n}~,
$$
\noindent which replaces a weakly globular $n$-fold groupoid
\w{G\in\Gpt{n}} by a Tamsamani weak $n$-groupoid \w{D\sb{n}G} of
the same homotopy type (Theorem \ref{ttamsamani})\vsm.

In Section \ref{cwgpng} we consider the wider context of
\emph{weakly globular pseudo $n$-fold groupoids}. These are defined in
Subsection \ref{cwgpng}.A, and compared to Tamsamani's model in Section
\ref{cwgpng}.B, where we again construct a discretization functor
$$
D_n\col \PsG{n}~\to~\Tam{n}~,
$$
\noindent and in Theorem \ref{tpsg} we show that for any
\w[,]{X\in\PsG{n}} there is a zig-zag of weak equivalences in
\w{\PsG{n}} between $X$ and \w[.]{D\sb{n}X}
Our main Theorem \ref{tpswgntype} then follows.

In Section \ref{cappl} we describe an application, and indicate some
directions for future work:   in Subsection \ref{cappl}.A, we show how to
extract from our results an algebraic model for \wwb{k-1}connected
$n$-types (Proposition \ref{pkconn}), and thus for iterated loop spaces.
In Subsection \ref{cappl}.B we define \emph{$n$-track categories}
(one of the original motivations for our work), with possible
future applications.

An appendix proves some technical facts about ${\osl{n}}$
needed in Section \ref{cfng}.
\end{mysubsect}
\newpage
%
%
%
\section*{Index of terminology and notation}
\begin{center}
\vspace{-5mm}
\renewcommand{\arraystretch}{1.35}
\begin{tabulary}{14.0cm}{l L l}
${\diN G}$ &  realization of an $n$-fold (pseudo) groupoid
$G$
& \S \ref{dwecs}\\
$\cons X$ &  discrete groupoid on a set $X$
& \S \ref{ddiscgpd}\\
${\ovl{c}\up{n}}$, ${c\up{n}}$  &  discrete groupoid functor applied to
an \wwb{n-1}fold simplicial groupoid or an \wwb{n-1}fold groupoid
& \S \ref{dbpz}, \S \ref{rbpz}\\
${D\sb{n}}$ &  discretization functors for \w{\Gpt{n}} and \w{\PsG{n}}
& \S \ref{dtndn}, \S \ref{ddnpsg}\\
${\diz}$ &  $0$-discretization functor on weakly globular
$n$-fold groupoids
& \S \ref{ddisc}\\
${\Dec}$, ${\Dec'}$ &  d\'{e}calage functors on simplicial sets
& \S \ref{sdec}\\
${\Dlo{n}}$  &  $n$-fold diagonal functor
& \S \ref{nsimp}\\
${\sCx{n}{\cC}}$  &  category of $n$-fold simplicial objects in $\cC$
& \S \ref{snssn}\\
${\FTm{n}}$  & Tamsamani's Poincar\'{e} $n$-groupoid functor
& \S \ref{ttamseq}\\
${\Gpd}$  &  category of groupoids
& \S \ref{cfng}.B\\
${n\mbox{-}\Gpd}$  &  category of  strict $n$-groupoids
& \S \ref{nthreem}\\
${\Gpd(\cV)}$  &  category of internal groupoids in $\cV$
& \S \ref{nfundg}\\
${\Gpd\sp{n}}$  &  category of $n$-fold groupoids
& \S \ref{nfundg}\\
${\Gpt{n}}$  &  category of weakly globular $n$-fold groupoids.
& \S \ref{dntng}\\
${\Gptk{n,k}}$  &  category of \wwb{n,k}weakly globular pseudo
$n$-fold groupoids
& \S \ref{dcont}\\
${\Ghd{n}}$  &  category of homotopically discrete $n$-fold groupoids
& \S \ref{dhdng}\\
${\Llo{k}X}$  &  simplicial ``bar-path construction''
& \S \ref{dllo}\\
${\seg{k}}$  &  $k$-th Segal map
& \S \ref{eqsegalmap}\\
${\iseg{k}}$  &  $k$-th induced Segal map
& \S \ref{eqindsegal}\\
${\Nup{i}}$  &  nerve functor of an $n$-fold groupoid in the
$i$-th direction
& \S \ref{addnote}\\
${\Nlo{n}}$  &  multinerve functor on $n$-fold groupoids
& \S \ref{dcspace}\\
${\osl{n}}$  &  $n$-fold ordinal sum of a simplicial set
& \S \ref{sorsum}\\
${\Pup{i}}$  &  left adjoint to \w{\Nup{i}}
& \S \ref{padjoint}\\
${\Plo{n}}$  &  left adjoint to \w{\Nlo{n}}
& \S \ref{dladjoint}\\
${\PhG{n}}$  &  category of homotopically discrete pseudo $n$-fold
groupoids
& \S \ref{dphdng}\\
${\PsGk{n,k}}$  &  category of \wwb{n,k}weakly globular pseudo
$n$-fold groupoids
& \S \ref{dcont}\\
${\PT{n}}$  &  full subcategory of $n$-Postnikov sections in \w{\Top}
& \S \ref{dpostnikov}\\
${\hpi}$  &  fundamental groupoid of a topological space
& \S \ref{dfundg}\\
${\bPz{n}}$  &  algebraic \wwb{n-1}st Postnikov section functor
& \S \ref{lhdng}, \S \ref{dntng}, \S \ref{dtamn}, \S \ref{dphdng}\\
$\Qlo{n},\, \hQlo{n}$  &  fundamental $n$-fold groupoid functors
& \S \ref{dladjoint}\\
${\Twg{n}}$  &  fundamental groupoid functor for \w{\Gpt{n}}
& \S \ref{ntwg}\\
${\TTm{n}}$  &  Tamsamani fundamental groupoid functor
& \S \ref{sbptam}\\
\end{tabulary}
\end{center}

\begin{center}
\renewcommand{\arraystretch}{1.35}
\begin{tabulary}{14.0cm}{l L L}
${\Tps{n}}$  &  fundamental groupoid functor for \w{\PsG{n}}
& \S \ref{nfundgppsg}\\
${\Tam{n}}, \,{\tTam{n}}$  & two equivalent formulations of the
category of Tamsamani weak $n$-groupoids
& \S \ref{dtamn}, \S \ref{sbptam}\\
${\Top}$  &  category of topological spaces.\\
${\Wlo{n,k}}$  &  $k$-fold object of arrows of an $n$-fold groupoid
&\S \ref{diao}\\
${\omega_{k}(G;x\sb{0})}$  &  $k$-th algebraic homotopy group
& \S \ref{sahgp}, \vs\S \ref{sahgpps}\\
\end{tabulary}
\end{center}

See also the list of special notations for $n$-fold simplicial objects
in \S \ref{nsimp}, in particular for the notation $\ovl{F}$ for any functor $F$.
\begin{ack}
This research was supported by the first author's Israel Science
Foundation Grant No.~47377, and the second author's  Marie Curie
International Reintegration Grant No.~256341 and London Mathematical Society
Grant No.~41134. The second author would also like to thank the
Department of Mathematics at the University of Haifa for its hospitality
during visits in March-April, 2012, and June and August, 2013.
We thank Dorette Pronk for many helpful comments. Finally, we would like to
thank the referee for his or her detailed and pertinent remarks.
\end{ack}

%
%
%
\sect{The fundamental $n$-fold groupoid of a space}
\label{cfng}

As noted above, the fundamental groupoid \w{\hpi T} of a (not necessarily
connected) space $T$ is an algebraic model for its $1$-type. We now
show how the notion of the fundamental $2$-typical double groupoid defined in
\cite[\S 2.21]{BPaolT} generalizes to all $n$. We consider the standard model
structure on \w[,]{\Top} so that \w{\ho\Top} means its localization with respect
to the class of weak homotopy equivalences.

\supsect{\protect{\ref{cfng}}.A}{Simplicial constructions}

Given a topological space $T$, we construct its fundamental $n$-fold groupoid
from a fibrant simplicial set model for $T$, such as its singular
set \w[.]{X:=\cS T\in \sS} We therefore first recall some notation and
constructions for simplicial sets.

\begin{defn}\label{dsimp}
For any category $\cC$, \w{\sC{\cC}} is the category of simplicial objects in
$\cC$, where $\bDelta$  denotes the category of finite ordered sets:
\w[,]{\bze,\bon} and so on. As usual, we write \w{X\sb{n}} for \w[.]{X(\bn)}
If $\cC$ is concrete, the \emph{$n$-skeleton} \w{\sk{n}X\in\sC{\cC}} of any
\w{X\in\sC{\cC}} is generated under the degeneracy maps by \w[.]{X_{0},\dotsc, X_{n}}
The \emph{$n$-coskeleton} functor \w{\csk{n}\col \sC{\cC}\to\sC{\cC}} is left adjoint
to \w[.]{\sk{n}} We say that $X$ is $n$-\emph{coskeletal} if the
natural map \w{X\to\csk{n}X} is an isomorphism.
\end{defn}

\begin{remark}\label{rinvolution}
There is an order-reversing involution \w[,]{I\col \Delta\to\Delta} which induces a
functor \w{I^{\ast}\col \sC{\cC}\to\sC{\cC}} (sending \w{d_{i}\col X_{n}\to X_{n-1}} to
\w[).]{d_{n-i}} This functor \w{I^{\ast}} is not generally an isomorphism, but
for a Kan complex \w{X\in\sS} we have a natural isomorphism of fundamental
groupoids \w{(\hpi I^{\ast}X)\op\cong\hpi X} (cf.\ \cite[I.8]{GJarS}).
\end{remark}

\begin{defn}\label{dsegal}
Let \w{X\in\sC{\cC}} be a simplicial object in any category $\cC$ with
pullbacks. For each \w[,]{1\leq j\leq k} let \w{\nu_j\col X_k\to X_1} be induced
by the map  \w{[1]\to[k]} in $\Delta$ sending $0$ to \w{j-1} and $1$ to $j$. Then the
following diagram commutes:
\mysdiag[\label{eqsegalmap}]{
&&&& X\sb{k} \ar[llld]_{\nu\sb{1}} \ar[ld]_{\nu\sb{2}} \ar[rrd]^{\nu\sb{k}} &&&& \\
& X\sb{1} \ar[ld]_{d\sb{1}} \ar[rd]^{d\sb{0}} &&
X\sb{1} \ar[ld]_{d\sb{1}} \ar[rd]^{d\sb{0}} && \dotsc &
X\sb{1} \ar[ld]_{d\sb{1}} \ar[rd]^{d\sb{0}} & \\
X\sb{0} && X\sb{0} && X\sb{0} &\dotsc X\sb{0} && X\sb{0}
}

If we let \w{\pro{X_1}{X_0}{k}} denote the limit of the lower part of
Diagram \wref[,]{eqsegalmap} the $k$-th \emph{Segal map} for $X$
is the unique map
$$
\seg{k}\col X\sb{k}~\to~\pro{X\sb{1}}{X\sb{0}}{k}
$$
\noindent such that \w[,]{\pr_j\,\seg{k}=\nu\sb{j}} where
\w{\pr\sb{j}} is the $j$-th projection (see \cite{SegCC}).

Note that $X$ is the nerve of an internal category in $\cC$ if and only if
all the Segal maps are isomorphisms.
\end{defn}

\begin{mysubsection}{$\mathbf{n}$-fold simplicial objects}
\label{snssn}
An \emph{$n$-fold simplicial object} in $\cC$ is a functor
\w[,]{\Dnop\to\cC} and we denote the category of such by
\w[.]{\sCx{n}{\cC}} Thus \w{X\in\sCx{n}{\cC}} consists of
objects \w{X_{i_{1}i_{2}\dotsc i_{n}}} in $\cC$ for each $n$-fold
multi-index \w[,]{i_{1},i_{2},\dotsc,i_{n}\in\bN} along with face and
degeneracy maps in each of the $n$ directions, satisfying the usual
simplicial identities. We assume a fixed ordering of these directions as
first, second, and so on.
\end{mysubsection}

\begin{mysubsection}{Notation and conventions}
\label{nsimp}
\begin{enumerate}
\renewcommand{\labelenumi}{(\alph{enumi})~}
\item We can identify \w{\sCx{n}{\cC}} with \w{\sC{\sCx{n-1}{\cC}}} in $n$
different ways:  thus, given an $n$-fold simplicial object \w[,]{X\in\sCx{n}{\cC}}
for each \w{1\leq i\leq n} we write \w{X\up{i}\in\sC{\sCx{n-1}{\cC}}} to indicate
that the primary simplicial direction is the $i$-th one of the original $X$.
\item More generally, if we choose $k$ of the $n$ directions
\w[,]{1\leq j_{1}<j_{2}<\dotsc<j_{k}\leq n}
we obtain a $k$-fold simplicial object
\w{X\up{j_{1},j_{2},\dotsc,j_{k}}} in \w[.]{\sCx{n-k}{\cC}} Thus
$$
X\up{j_{1},j_{2},\dotsc,j_{k}}\in\sCx{k}{\sCx{n-k}{\cC}}
$$
\noindent is a diagram of objects
\w{X\up{j_{1},j_{2},\dotsc,j_{k}}\sb{i_{1}\dotsc i_{k}}} in
\w[.]{\sCx{n-k}{\cC}} For example,
\w[,]{X\up{1,\dotsc,k}\sb{i_{1}\dotsc i_{k}}=X([i_{1}],\dotsc[i_{k}],-)}
in the notation of \S \ref{dsimp}.

Equivalently, for each object
\w[,]{\va\in\bDelta^{n-k}}
\w{X\up{j_{1},j_{2},\dotsc,j_{k}}(\va)\in\sCx{k}{\cC}} is a $k$-fold
simplicial object in $\cC$, natural in $\va$.
\item In particular,
$$
X\up{\widehat{\imath}}=X\up{1,\dotsc,i-1,i+1,\dotsc,n}\in\sCx{n-1}{\sC{\cC}}
$$
\noindent is an \wwb{n-1}fold simplicial object in \w{\sC{\cC}}
(in the $i$-th direction).
\item Given \w{X\in\sCx{n}{\cC}} and a functor \w[,]{F\col \sC{\cC}\to\cC}
we denote by \w{F\up{k}X\in\sCx{n-1}{\cC}} the object obtained by applying $F$
``objectwise'' to \w{X\up{\widehat{k}}}\!\! (thought of as a
\ww{\Dox{n-1}}-indexed diagram in \w{\sC{\cC}}\!\!\!).

Thus for every \w[,]{i\sb{1},\dotsc,i\sb{n-1}\in\bN} we have
$$
(F\up{k}X)\sb{i\sb{1},\dotsc,i\sb{n-1}}~=~
FX\up{1,\dotsc,k-1,k+1,\dotsc,n}\sb{i\sb{1},\dotsc,i\sb{n-1}}~.
$$
\item The composite \w{F\up{1}F\up{2}\dotsc F\up{n-1}F\up{n}}
will be denoted by \w[.]{F\lo{n}\col \sCx{n}{\cC}\to\cC}
\item In particular, the $n$-fold diagonal functor
\w{\Dlo{n}\col \sCx{n}{\cC}\to\sC{\cC}} is given by
\w[.]{(\Dlo{n}X)_{m}:=X_{m,m,\dotsc,m}} (In this case, the order does not matter.)
\item For any functor \w[,]{F\col \cC\to\cD} the prolongation of $F$ to simplicial
objects is denoted by \w[.]{\ovl{F}\col \sC{\cC}\to\sC{\cD}}
\item In particular, for a functor \w[,]{G\col \sCx{n-1}{\cC}\to\cD} the result
of applying $G$ to an $n$-fold simplicial object \w{X\in\sCx{n}{\cC}} in each
simplicial dimension in the $k$-th direction will be denoted by
\w[.]{\ovl{G}\up{k}X\in\sC{\cD}}
Thus for every \w{j\in\bN} we have
$$
(\ovl{G}\up{k}X)\sb{j}~=~GX\up{k}\sb{j}~.
$$
\end{enumerate}
\end{mysubsection}

\begin{mysubsection}{D\'{e}calage}
\label{sdec}
Recall from \cite[\S 2.6]{DuskSM} the comonad \w{\Dec\col \sS\to\sS} on
simplicial sets, where \w[,]{(\Dex)_{n}=X_{n+1}} forgetting the \emph{last} face
and degeneracy operators in each dimension (see also \cite{IlluC2}). The
counit \w{\var\col \Dex\to X} is given by \w{d_{n}\col X_{n+1}\to X_{n}} in simplicial
dimension $n$. It has a section \w[,]{\sigma\col X\to\Dex}
given by \w[.]{s_{n}\col X_{n}\to X_{n+1}}

There is also a version forgetting the \emph{first} face and degeneracy operators,
which we denote by \w[.]{\Decp\col \sS\to\sS} In the notation of
\S \ref{rinvolution}, \w[.]{\Decp X:=I^{\ast}\Dec I^{\ast}X}

The comonad \w{\Dec} yields a simplicial resolution \w{\Yd\in\sC{\sS}} for any
\w[,]{X\in\sS} with
$$
Y_{k-1}~:=~\Dec^{k}X~:=~\underset{k}{\underbrace{\Dec(\Dec\ldots\Dec X\ldots)}}
\hspace*{7mm}\text{in~~}\sS~,
$$
\noindent and the counit $\var$ for \w{\Dec} induces a map of
bisimplicial sets \w[,]{\var\col Y\to\cons\up{2}X} where
\w{\cons\up{2}X} is the constant simplicial object on $X$ in \w{\sC{\sS}}
(thinking of the outer simplicial direction of \w{\sC{\sS}} as second).
The bisimplicial set \w{\Yd} is depicted in Figure \ref{for}, viewed
as a horizontal simplicial object over \w{\sS} (degeneracy maps and $\var$
are not shown).

The corresponding resolution using \w{\Decp} is also depicted in
Figure \ref{for}, viewed as a vertical simplicial object over \w[.]{\sS}

%
%
\begin{figure}[ht]
\begin{center}
$$
\entrymodifiers={+++[o]}
\xymatrix@R=25pt{
\cdots \ssr &
X_{5} \ar@<2ex>[rrr]^{d_{5}} \ar[rrr]^{d_{4}}\ar@<-2ex>[rrr]^{d_{3}}
     \ar@<2.5ex>[d]^{d_{2}} \ar[d]^{d_{1}}\ar@<-2.5ex>[d]^{d_{0}} &&&
X_{4} \ar@<1ex>[rrr]^{d_{4}} \ar@<-1ex>[rrr]^{d_{3}}
     \ar@<2.5ex>[d]^{d_{2}} \ar[d]^{d_{1}}\ar@<-2.5ex>[d]^{d_{0}} &&&
X_{3}  \ar@<2.5ex>[d]^{d_{2}} \ar[d]^{d_{1}}\ar@<-2.5ex>[d]^{d_{0}} \\
\cdots \ssr  &
X_{4} \ar@<2ex>[rrr]^{d_{4}} \ar[rrr]^{d_{3}}\ar@<-2ex>[rrr]^{d_{2}}
     \ar@<0.5ex>[d]^{d_{1}} \ar@<-0.5ex>[d]_{d_{0}} &&&
X_{3} \ar@<1ex>[rrr]^{d_{3}} \ar@<-1ex>[rrr]^{d_{2}}
     \ar@<0.5ex>[d]^{d_{1}} \ar@<-0.5ex>[d]_{d_{0}} &&&
X_{2}       \ar@<0.5ex>[d]^{d_{1}} \ar@<-0.5ex>[d]_{d_{0}} \\
\cdots \ssr &
X_{3} \ar@<2ex>[rrr]^{d_{3}} \ar[rrr]^{d_{2}}\ar@<-2ex>[rrr]^{d_{1}} &&&
X_{2} \ar@<1ex>[rrr]^{d_{2}} \ar@<-1ex>[rrr]^{d_{1}} &&& X_{1}
}
$$
\end{center}
\caption{Corner of \ $\osl{2} X$}
\label{for}
\end{figure}
\end{mysubsection}

\begin{remark}\label{rdex}
Note that if $X$ is a fibrant simplicial set, then so is
\w[,]{\Dec X} and the augmentation \w{\var\col \Dec X\to X} is a fibration
(with section \w[).]{\sigma\col X\to\Dex} Similarly for \w[.]{\Decp}
\end{remark}

\begin{mysubsection}{Ordinal sum}
\label{sorsum}
In order to produce an $n$-fold simplicial set out of a Kan complex
\w[,]{X\in\sS} with the same homotopy type (that is, an $n$-fold resolution
of $X$), we shall use the functor
\w[,]{\osl{n}:=\ordin\sb{n}\sp{\ast}\col \sS\to\Snx} induced by the ordinal sum
\w{\ordin\sb{n}\col \Delta^{n}\to\Delta} (cf.\ \cite[\S 2]{EPorJ}). Thus
\begin{myeq}\label{eqorn}
(\osl{n} X)_{p\sb{1}\dotsc p\sb{n}}~:=~X_{n-1+p\sb{1}+\dotsc p\sb{n}}
\end{myeq}

If we define \w{\ovt{i}{n-1}\col \Sx{2}\to\Snx} for a bisimplicial set
$X$ by applying \w{\osl{n-1}} to $X$ in each simplicial dimension in
the $i$-th direction \wb{i=1,2} (cf.\ \S \ref{nsimp}(h)), we have:
\begin{myeq}\label{eqovt}
\osl{n} X~=~\ovt{2}{n-1}\osl{2}X~.
\end{myeq}

See Figure \ref{fort} for a depiction of \w[,]{\osl{3}X} where the
vertical direction is first, the diagonal is second, and the horizontal is
third.

The bisimplicial set \w{\osl{2} X} appears in Figure \ref{for}:  this means that
if we choose the vertical direction to be first and the horizontal to be
second, then

\begin{myeq}\label{eqorder}
(\osl{2} X)^{(1)}_{i}~=~\Dec^{i+1}X~\hspace*{7mm}\text{and}\hspace*{3mm}
(\osl{2} X)^{(2)}_{i}~=~(\Decp)^{i+1}X~.
\end{myeq}
\end{mysubsection}
%
%
\begin{lemma}\label{lores}
For any  simplicial set \w[,]{X\in\sS} there is a natural weak
equivalence \w[.]{\var\lo{n}\col \Dlo{n}\osl{n} X\to X}
\end{lemma}

\begin{proof}
By induction on \w[.]{n\geq 2}

For \w[,]{n=2} as noted in \S \ref{sdec}, the counit \w{\var\col \Dex\to X} induces
a map of bisimplicial sets \w{\wvar\col \osl{2}X\to\cons\up{2}X} which is a weak
equivalence of horizontal simplicial sets \w{(\Decp)^{i+1}X\to\cons X_{i}}
(where \w{\cons X_{i}} is the constant simplicial set on the set \w[),]{X_{i}}
using \wref[.]{eqorder} Thus  by \cite[\S 2.6]{DuskSM} it induces a weak
equivalence
$$
\var\lo{2}\col \Dlo{2}\osl{2}X~\to~\Dlo{2}\cons\up{2}X~=~X~.
$$

In the induction stage we have a weak equivalence
$$
\var\lo{n-1}\col \Dlo{n-1}\osl{n-1} Y\to Y~,
$$
\noindent natural in $Y$.
Using \wref[,]{eqovt} and applying \w{\var\lo{n-1}} to \w{\osl{n}X}
in each simplicial dimension (in direction 2), we obtain a map of bisimplicial
sets:
$$
\oDlo{n-1}\up{2}\osl{n}X~=~\oDlo{n-1}\up{2}\ovt{2}{n-1}\osl{2}X~
\xrw{\ovar\lo{n-1}\up{2}}~\osl{2}X
$$
\noindent which is a weak equivalence in each simplicial dimension in direction
2, by the induction hypothesis. Therefore, after applying \w{\Dlo{2}} we obtain a
weak equivalence of simplicial sets
$$
\Dlo{2}\ovar\lo{n-1}\up{2}\col \Dlo{n}\osl{n}X~\to~\Dlo{2}\osl{2}X~.
$$
\noindent Post-composing with \w{\var\lo{2}\col \Dlo{2}\osl{2} X\to X} yields the
required weak equivalence \w[.]{\var\lo{n}\col \Dlo{n}\osl{n} X\to X}
\end{proof}

\supsect{\protect{\ref{cfng}}.B}{$\mathbf{n}$-Fold groupoids}

Recall that a \emph{groupoid} is a small category $G$ in which all morphisms
are isomorphisms. It can thus be described by a diagram of
sets:
\mydiagram[\label{eqgpoid}]{
G_{1}\times_{G_{0}}G_{1}\ar@<2ex>[rr]^{d_{0}}\ar[rr]^{m}\ar@<-1.5ex>[rr]_{d_{2}}&&
G_{1}\ar@<0.7ex>[rr]^{s} \ar@<-.7ex>[rr]_{t} \ar@/^2pc/[ll]^{s_{1}}
\ar@/_2pc/[ll]_{s_{0}} &&  G_{0} \ar@/_1.3pc/[ll]_{i}
}
\noindent where \w{G_{0}} is the set of objects of $G$ and \w{G_{1}} the
set of arrows. Here $s$ and $t$ are the source and target functions, $i$
associates to an object its identity map, \w{d_{0}} and \w{d_{2}} are
the respective projections, with sections \w{s_{0}} and \w[,]{s_{1}}
and $m$ is the composition \wh all satisfying appropriate identities.
Let \w{\Gpd} denote the category of small groupoids (a full subcategory of
the category \w{\Cat} of small categories).

We can think of \wref{eqgpoid} as the $2$-skeleton of a simplicial
set (with \w[,]{G_{2}:=G_{1}\times_{G_{0}}G_{1}} and
\w[).]{d_{1}=m\col G_{2}\to G_{1}} The \emph{nerve} functor
\w{\cN\col \Gpd\to\sS} (cf.\ \cite{SegCC}) assigns to $G$ the
corresponding $2$-coskeletal simplicial set \w[,]{\cN G} so:
\begin{myeq}\label{eqcoskel}
(\cN G)_{n}~:=~
G_{1}\times_{G_{0}}G_{1}~\overset{n}{\cdots}~G_{1}\times_{G_{0}}G_{1}
\end{myeq}
\noindent for all \w[,]{n\geq 2} with face and degeneracy maps determined
by the associativity and unit laws for the composition $m$.
%

\begin{defn}\label{nfundg}
If $\cV$ is any category with pullbacks, an \emph{internal groupoid}
in $\cV$ is a diagram in $\cV$ of the form \wref[,]{eqgpoid} satisfying
the same axioms (see \cite[I, \S 8.1]{BorcH}). The category of
internal groupoids in $\cV$ is denoted by \w[.]{\Gpd(\cV)} Thus an
(ordinary) groupoid is an internal groupoid in \w[.]{\Set}

When \w{\cV} is locally finitely presentable, the nerve functor
\w{N\col \Gpd(\cV) \to [\Dop, \cV]} has a left adjoint,
the fundamental internal groupoid (see \cite[II, \S 5.5-5.6]{BorcH}).

For each \w[,]{n\geq 1} an \emph{$n$-fold groupoid} is defined inductively to
be an internal groupoid  in the category \w{\cV=\Gpd^{n-1}} of \wwb{n-1}fold
groupoids (where \w[),]{\Gpd^{0}:=\Set} so
$$
\Gpd^{n}~:=~\Gpd(\Gpd^{n-1})~.
$$
\end{defn}

\begin{defn}\label{dfundg}
Let \w{\hpi\col \sS\to\Gpd} denote the \emph{fundamental groupoid} functor.
See \cite[\S I.8]{GJarS} and \wref{nfundg} when \w[.]{\cV = \Set}
When $X$ is fibrant, \w{\hpi X} has the simple form described in
\cite[\S I.8]{GJarS}. If \w{X\in\Snx} is an $n$-fold simplicial set, then for each
\w[,]{1\leq i\leq n} \w{\hpu{i}X} is the \wwb{n-1}fold simplicial
object in \w{\Gpd} obtained by applying the fundamental groupoid
functor \w{\hpi} in the $i$-th direction \wh that is, objectwise to the
\ww{\Dox{n-1}}-indexed diagram \w[.]{X\up{i}}
\end{defn}

\begin{notation}\label{addnote}
As in \S \ref{nsimp}(d), for each \w[,]{1\leq i\leq n} let
\w{\Nup{i}\col \Gpd^{n}\to\sC{\Gpd^{n-1}}} denote the nerve functor in the
$i$-th direction. More generally, for any $k$ of
the $n$ indices \w[,]{1\leq i_{1}<i_{2}<\dotsc<i_{k}\leq n}
\w{\Nup{i_{1},i_{2},\dotsc,i_{k}}\col \Gpd^{n}\to\sCx{k}{\Gpd^{n-k}}} takes an
$n$-fold groupoid $G$ to a $k$-fold simplicial object in \wwb{n-k}fold
groupoids by applying the nerve functor in the indicated $k$ directions. In
particular, \w{\Nup{\widehat{\imath}}} means that we take nerves in
all but the $i$-th direction.
\end{notation}

\begin{defn}\label{dcspace}
The \emph{multinerve}
$$
\Nlo{n}\col \Gpd^{n}\to\Snx
$$
\noindent is defined by applying \w{\Nup{i}} for \w{1\leq i\leq n} to obtain
the $n$-fold simplicial set \w[.]{\Nlo{n}G:=\Nup{1}\Nup{2}\dotsc\Nup{n}G} We say
that an $n$-fold groupoid $G$ is \emph{discrete} if \w{\Nlo{n}G} is
a constant $n$-fold simplicial set. It is readily verified that we have an
adjoint pair \w[:]{\Plo{n}\dashv\Nlo{n}}
\begin{myeq}\label{eqadj}
\Snx~\adj{\Nlo{n}}{\Plo{n}}~\Gpd^{n}
\end{myeq}
where \w{\Plo{n}} is the left adjoint to \w{N\lo{n}} as in \wref{nfundg}
with \w[.]{\cV = \Gpd^{n-1}}
\end{defn}

\begin{defn}\label{dwecs}
The composite of \w{\Nlo{n}} with \w{\Dlo{n}} (cf.\ \S \ref{nsimp}(f))
yields the \emph{diagonal nerve} functor \w[,]{\dN:=\Dlo{n}\Nlo{n}} and its
geometric realization \w{\diN G:=\|\dN G\|\in\Top} is called the
\emph{classifying space} of $G$.

A map of $n$-fold groupoids \w{f\col G\to G'} is called a
\emph{geometric weak equivalence} if it induces a weak equivalence of
simplicial sets \w{\dN f\col \dN G\to\dN G'}  (that is, a homotopy equivalence
of topological spaces on geometric realizations
\w[).]{\diN f\col \diN G\to\diN G'}
\end{defn}

\begin{remark}\label{rgwe}
Since the diagonal of a bisimplicial set is its homotopy colimit, a map
\w{f\col X\to Y} in \w{\sCx{2}{\Set}} which is a weak equivalence
\w{f\sb{k}\col X\sb{k}\to Y\sb{k}} in each simplicial dimension \w{k\geq 0}
is a geometric weak equivalence (cf.\ \cite[IV, Proposition 1.7]{GJarS}).
Thus by induction the same is true for a map
\w{f\col X\to Y} in \w{\sCx{n}{\Set}} which is a geometric weak equivalence
in each simplicial dimension.
\end{remark}

\begin{defn}\label{ddiscgpd}
If \w{G\in\Gpd^{n-1}} is an \wwb{n-1}fold groupoid, then \w{\cons\up{n}G} denotes
the $n$-fold groupoid which, as a groupoid object in \w[,]{\Gpd^{n-1}} is discrete
on $G$. In particular, if $A$ is a set, \w{A^{d}\lo{n}} denotes the discrete
$n$-fold groupoid \w{\cons\up{1}\dotsc\cons\up{n} A} on $A$. For an
$n$-fold groupoid $G$ we let \w{G\sp{d}} denote
the discrete $n$-fold groupoid \w[.]{(\pi\sb{0}\diN G)\sp{d}\lo{n}}
\end{defn}

\begin{notation}\label{nbpz}
If \w{G\in\Gpd^{n}} is an $n$-fold groupoid for \w[,]{n\geq 2} it is a
groupoid object in \wwb{n-1}fold groupoids (cf.\ \S \ref{nfundg}):
that is, it is described by a diagram
\w{G\up{1}_{1}\toto G\up{1}_{0}} in \w[,]{\Gpd^{n-1}} as in \wref[.]{eqgpoid}
Thus it has an \wwb{n-1}fold groupoid of objects denoted by
\w[,]{G\up{1}_{0}} in the notation of \S \ref{nsimp}(a) (which in turn has its
\wwb{n-2}fold groupoid of objects \w{G\up{1,2}_{00}} and \wwb{n-2}fold
groupoid of morphisms \w[).]{G\up{1,2}_{01}}
Similarly, the \wwb{n-1}fold groupoid of morphisms of $G$ (in the first
direction) is denoted by \w[.]{G\up{1}_{1}}

More explicitly, $G$ may be
described by a diagram in \w{\Gpd^{n-2}} of the form:
\mydiagram[\label{eqdgpd}]{
& G_{11}\times_{G_{10}}G_{11}\ar[d]^{c^{1\ast}} \ar@<0.5ex>[r]\ar@<-0.5ex>[r] &
G_{01}\times_{G_{00}}G_{01}\ar[d]^{c^{0\ast}} \\
G_{11}\times_{G_{01}}G_{11}\ar[r]^<<<<<<<{c^{\ast 1}}
\ar@<0.5ex>[d]\ar@<-0.5ex>[d] &
G_{11} \ar@<0.5ex>[d]^{d_{0}^{1\ast}} \ar@<-0.5ex>[d]_{d_{1}^{1\ast}}
\ar@<0.5ex>[r]^{d_{0}^{\ast 1}} \ar@<-0.5ex>[r]_{d_{1}^{\ast 1}}  &
G_{01} \ar@<0.5ex>[d]^{d_{1}^{0\ast}} \ar@<-0.5ex>[d]_{d_{1}^{0\ast}}
\ar@/_1.5pc/[l] \\
G_{10}\times_{G_{00}}G_{10} \ar[r]^<<<<<<<{c^{\ast 0}} &
G_{10} \ar@<0.5ex>[r]^{d_{0}^{\ast 0}} \ar@<-0.5ex>[r]_{d_{1}^{\ast 0}}
\ar@/^2pc/[u] &
G_{00}~. \ar@/^1.5pc/[l] \ar@/_2pc/[u]
}
\noindent Here we omit throughout the upper index \w[,]{(1,2)} which indicates
that we are showing only the first two directions of $G$.

More generally, for each \w{i\geq 2} we let
\begin{myeq}\label{eqextend}
G_{i1}:=\pro{G_{11}}{G_{01}}{i} \hsm \text{and}\hsm
G_{i0}:=\pro{G_{10}}{G_{00}}{i}
\end{myeq}
\noindent as limits of \wwb{n-2}fold groupoids, with
\w{d^{i,\ast}_{0},d^{i,\ast}_{1}\col G_{i1}\to G_{i0}} induced by the
source and target maps.
\end{notation}

\begin{remark}\label{rnfoldg}
Using this convention, an $n$-fold groupoid $G$ may be thought of as a
diagram of sets with objects \w{G\sb{i_{1},\dotsc,i_{n}}} for each
\w[,]{(i_{1},\dotsc,i_{n})\in\bN^{n}=\Obj(\bDelta^{n})} where all the maps
in the diagram are induced by those of \wref{eqgpoid} and the structure maps
for the limits \wref{eqextend} (in each of the $n$ directions).
\end{remark}
The following technical fact about \w{\osl{n}} will be used in Subsection
\ref{cnt}.A below:

\begin{lemma}\label{lnuph}
For any fibrant simplicial set \w{X\in\sS} and \w[,]{n\geq 2} we have:
\begin{myeq}\label{eqnuph}
\ovt{2}{n-1}\Nup{2}\hpu{2} \osl{2}X~=~\Nup{n}\hpu{n}\osl{n}X~.
\end{myeq}
\end{lemma}

\
\begin{proof}
\
%
%
%
\begin{figure}[ht]
\begin{center}
$$
\entrymodifiers={+++[o]}
\xymatrix @R=20pt @C=40pt{
X_{5}\ar@<0.7ex>[rr]^{d_{5}} \ar@<-0.7ex>[rr]_{d_{4}} \ar@<0.7ex>[dd]^{d_{1}}
    \ar@<-0.7ex>[dd]_{d_{0}} \ar@<0.7ex>[rd]^{d_{3}} \ar@<-0.7ex>[rd]_{d_{2}} &&
X_{4} \ar@{-}[d]<0.7ex> \ar@{-}[d]<-0.7ex> \ar@<0.7ex>[rd]^{d_{3}}
     \ar@<-0.7ex>[rd]_{d_{2}}\\
& X_{4} \ar@<0.7ex>[rr]^(0.3){d_{4}} \ar@<-0.7ex>[rr]_(0.3){d_{3}}
     \ar@<0.7ex>[dd]^(0.3){d_{1}} \ar@<-0.7ex>[dd]_(0.3){d_{0}}&
    \text{  }\ar@<0.7ex>[d]^{d_{1}} \ar@<-0.7ex>[d]_{d_{0}}&
X_{3} \ar@<0.7ex>[dd]^{d_{1}} \ar@<-0.7ex>[dd]_{d_{0}}\\
X_{4} \ar@<0.7ex>[rd]^{d_{2}} \ar@<-0.7ex>[rd]_{d_{1}} \ar@{-}[r]<0.7ex>
   \ar@{-}[r]<-0.7ex>  &\text{  }\ar@<0.7ex>[r]^{d_{4}} \ar@<-0.7ex>[r]_{d_{3}}&
X_{3} \ar@<0.7ex>[rd]^{d_{2}} \ar@<-0.7ex>[rd]_{d_{1}} \\
& X_{3} \ar@<0.7ex>[rr]^{d_{3}} \ar@<-0.7ex>[rr]_{d_{2}}  &&
X_{2} }
$$
\end{center}
\caption{Corner of \ $\osl{3} X$}
\label{fort}
\end{figure}
By induction on \w[.]{n\geq 2}

When \w[,]{n=2} \w{\osl{n-1}} is the identity, so both sides of \wref{eqnuph} are
the same.

 For \w[,]{n=3} \w{\hpu{3}\osl{3}X} is obtained from Figure \ref{fort}
by replacing the left-hand square by
$$
\xymatrix{
X_{5}\bsim \ar@<0.7ex>[d]^{d_{1}} \ar@<-0.7ex>[d]_{d_{0}}
\ar@<0.7ex>[rr]^{d_{3}} \ar@<-0.7ex>[rr]_{d_{2}} &&
X_{4}\bsim \ar@<0.7ex>[d]^{d_{1}} \ar@<-0.7ex>[d]_{d_{0}}\\
X_{4}\bsim \ar@<0.7ex>[rr]^{d_{2}} \ar@<-0.7ex>[rr]_{d_{1}}
&& X_{3}\bsim
}
$$
\noindent and from Figure \ref{for} we see this is the same as first applying
\w{\cN\hpi} to \w{\osl{2}X} horizontally in each vertical dimension (which is
\w[),]{\Nup{2}\hpu{2}} and then taking \w{\osl{2}} vertically in each
horizontal dimension (which is \w[).]{\ovt{2}{2}}

For \w[,]{n\geq 4} we see that
\begin{equation*}
\Nup{n}\hpu{n}\osl{n}X~=~\Nup{n}\hpu{n}\ovt{2}{n-1}\osl{2}X~=~
\ovl{\Nup{n-1}\hpu{n-1}\osl{n-1}}\up{2}\,\osl{2}X
\end{equation*}
\noindent using \wref{eqovt} and the convention of \S \ref{nsimp}(d). Applying
the induction hypothesis \wref{eqnuph} for \w[,]{n-1} we see this is equal
to:
$$
\ovl{\ovt{2}{n-2}\Nup{2}\hpu{2}\osl{2}}\up{2}\,\osl{2}X
~=~\ovl{\ovt{2}{n-2}\Nup{2}\hpu{2}}\up{2}\ovt{2}{2}\,\osl{2}X~,
$$
\noindent and using \wref{eqovt} for \w[,]{n=3} we see this is
$$
\ovl{\ovt{2}{n-2}\Nup{2}\hpu{2}}\up{3}\,\osl{3}X~=~
\ovl{\ovt{2}{n-2}}\up{3}\Nup{3}\hpu{3}\,\osl{3}X~,
$$
\noindent where for a $3$-fold simplicial object $Z$, we have
$$
\ovl{\ovt{2}{n-2}\Nup{2}\hpu{2}}\up{3}Z=
\ovl{\ovt{2}{n-2}\Nup{2}\hpu{2}}\up{2}Z
$$
\noindent by our indexing convention \ref{nsimp}(h).

Now applying \wref{eqnuph} for \w[,]{n=3} we see this equals:
$$
\ovl{\ovt{2}{n-2}}\up{2}\ovt{2}{2}\Nup{2}\hpu{2}\,\osl{2}X~=~
\ovt{2}{n-1}\Nup{2}\hpu{2}\,\osl{2}X~,
$$
\noindent using \wref{eqovt} once more.
\end{proof}
%
\supsect{\protect{\ref{cfng}}.C}
{The fundamental $\mathbf{n}$-fold groupoid of a space}

We now introduce the central construction of our paper. Its internal
analogue in the category of groups is the fundamental
\ww{\cat\sp{n}}-group of a space, due to Bullejos, Cegarra, and Duskin
(see \cite[\S 2]{BCDusC}).

\begin{defn}\label{dladjoint}
We define \w{\Qlo{n}\col \sS\to\Gpd^{n}} to
be the composite
\begin{equation*}
\sS~\supar{\osl{n}}~\Snx~\supar{\Plo{n}}~\Gpd^{n}~,
\end{equation*}
\noindent for \w{\Plo{n}} the left adjoint to \w{\Nlo{n}} of \wref[.]{eqadj}
We define \w{\hQlo{n}\col \Top\to\Gpd^{n}} to be the composite
\begin{equation*}
\Top~\supar{\cS}~\sS~\supar{\Qlo{n}}~\Gpd^{n}~,
\end{equation*}
\noindent where \w{\cS\col \Top\to\sS} is the singular set functor
(cf.\ \cite[I, \S 1]{GJarS}, and call \w{\hQlo{n}T} the
\emph{fundamental $n$-fold groupoid} of \w[.]{T\in\Top}
\end{defn}

We shall show that if \w{Y\in\Snx} satisfies certain fibrancy conditions, then
\w{\Plo{n} Y} has a particularly simple form. These require that a certain
$2$-dimensional notion of fibrancy (introduced in \cite[\S 2]{BPaolT}) hold in
every bisimplicial bidirection. They hold for \w{Y=\osl{n} X} when
$X$ is fibrant (in particular, for \w[),]{X=\cS T} leading to a simple
expression for \w{\hQlo{n} T} in Theorem \ref{tadjoint} below.

\begin{defn}\label{dfibrant}
Let \w[.]{n\geq 2} An $n$-fold simplicial set \w{X\in\Snx} is
called \emph{\wwb{n,2}fibrant} if for each \w{1\leq i\neq j\leq n}
and \w[,]{\va\in\bDelta^{n-2}} the bisimplicial set  $Y$
obtained by applying the $2$-coskeleton functor to each
vertical simplicial set \w{X\up{i,j}(\va)_{k\bullet}} \wwh that is,
\w{Y_{k\bullet}:=\csk{2}X\up{i,j}(\va)_{k\bullet}} \wwh is a Kan complex for
\w[,]{k=0,1,2} and the horizontal face map
\w{d_{0}\col Y_{1\bullet}\to Y_{0\bullet}} is a fibration in \w[.]{\sS}
\end{defn}

\begin{defn}\label{dgfibrant}
Let \w{G\in\sCx{m}{\Gpd^{n-m}}} be an $m$-fold simplicial object in
\wwb{n-m}fold groupoids (cf.\ \S \ref{nfundg}). We say that $G$ is
\emph{\wwb{n,2}fibrant} if, after applying the nerve functor in each of the
\w{n-m} groupoid directions, the resulting $n$-fold simplicial set
\w{\Nup{1,2,\dotsc,n-m}G\in\Snx} is \wwb{n,2}fibrant in the sense of Definition
\ref{dfibrant}.
\end{defn}

We recall the following results from \cite{BPaolT} where the left adjoint
\w{\Pup{1}\col \sC{\Gpd}\to\Gpd^{2}} to the nerve
\w[,]{\Nup{1}\col \Gpd^{2}\to\sC{\Gpd}} is as in \wref{nfundg} with \w[.]{\cV=\Gpd}

%
%
\begin{prop}[\protect{\cite[Prop.\ 2.10]{BPaolT}}]\label{ptadjoint}
The left adjoint \w{\Pup{1}\col \sC{\Gpd}\to\Gpd^{2}} to the nerve
\w[,]{\Nup{1}\col \Gpd^{2}\to\sC{\Gpd}} when applied to a \wwb{2,2}fibrant
simplicial groupoid \w[,]{\Gd} is given by \w{\hpu{1}\Gd} (the functor
\w{\hpi} applied in the simplicial direction).
\end{prop}

%
%
\begin{prop}[\protect{\cite[Prop.\ 2.11]{BPaolT}}]\label{ptfibrant}
If \w{X\in\Sx{2}} is a \wwb{2,2}fibrant bisimplicial set, then \w{\hpu{1} X} is a
\wwb{2,2}fibrant simplicial groupoid.
\end{prop}

%
%
\begin{lemma}\label{ltpi}
If \w{\Gd} is a \wwb{2,2}fibrant simplicial groupoid (with simplicial
sets of objects \w{G\sb{\bullet0}}
and morphisms \w[),]{G\sb{\bullet1}} then
\w[.]{\Nup{2}\hpu{1}\Gd=\hpu{1}\Nup{2}\Gd}
\end{lemma}

\begin{proof}
It suffices to show that for each \w[:]{k\geq 2}
\begin{myeq}\label{eqhpisg}
\hpi(\pro{G\sb{\bullet1}}{G\sb{\bullet0}}{k})~\cong~
\pro{\hpi(G\sb{\bullet1})}{\hpi(G\sb{\bullet0})}{k}~.
\end{myeq}
\noindent Since both sides are groupoids, we evidently have equality
on objects, and \wref{eqhpisg} holds on morphisms by
\cite[App.~A, following (8.13)]{BPaolT}.
\end{proof}

\begin{lemma}\label{lfibrant}
If \w{X\in\Snx} is \wwb{n,2}fibrant, then \w{\hpu{k} X} is
\wwb{n,2}fibrant.
\end{lemma}

\begin{proof}
By definition of \wwb{n,2}fibrancy, for each \w{\va\in\bDelta^{n-2}}
and \w[,]{1\leq i\neq j\leq n} the bisimplicial set \w{X\up{i,j}()}
satisfies the hypotheses of Proposition \ref{ptfibrant}. Hence,
applying \w{\hpu{k}} to it yields an \wwb{n,2}fibrant object of
\w[.]{\sCx{n-2}{\sC{\Gpd}}}
\end{proof}

%
%
\begin{prop}\label{padjoint}
For each \w[,]{1\leq i\leq n} \w[,]{\Pup{i}\col \sC{\Gpd^{n-1}}\to\Gpd^{n}}
the left adjoint to
\w{\Nup{i}\col \Gpd^{n}\to\sC{\Gpd^{n-1}}} of \wref[,]{eqadj}
when applied to an \wwb{n,2}fibrant simplicial \wwb{n-1}fold
groupoid $X$, is given by \w[.]{\Pup{i} X=\hpu{i} X}
\end{prop}

\begin{proof}
We think of the simplicial direction of $X$ as being the $i$-th, and
let \w{1\leq j\leq n} be one of the groupoidal directions
(so \w[).]{i\neq j} Applying the \wwb{n-2}fold iterated nerve functor
$$
\Nup{\widehat{\imath}\widehat{\jmath}}\col ~\sC{\Gpd^{n-1}}~\to~
\sC{\sCx{n-2}{\Gpd}}~\cong~\sCx{n-2}{\sC{\Gpd}}
$$
\noindent of \S \ref{addnote} (in all but the $i$ and $j$ directions)
to $X$ yields an \wwb{n-2}fold simplicial object in simplicial groupoids
$\tX$.
Since $X$ is\wwb{n,2}fibrant, for each \w[,]{\va\in\bDelta^{n-2}} the
simplicial groupoid \w{\tX(\va)} (see \S \ref{nsimp}(b)) satisfies the
hypotheses of Proposition \ref{ptadjoint}, where the simplicial
direction is the original $i$ and the groupoid direction is the
original $j$. Using \cite[(8.12)]{BPaolT}, we can therefore define a
composition map:
\begin{equation*}
\tens{(\Nup{i}\hpu{i}\tX(\va))_{1}}{(\Nup{i}\hpu{i}\tX(\va))_{0}}~\lto~
(\Nup{i}\hpu{i}\tX(\va))_{1}~.
\end{equation*}
\noindent As the construction is functorial in
\w[,]{\va\in\Dox{n-2}} it defines a map in \w[,]{\Gpd^{n-1}} since it
consists of maps in sets commuting with compositions
in each of the different directions (see \cite[Appendix A]{BPaolT}). Thus
\w{\hpu{i} X} is a groupoid object in \w{\Gpd^{n-1}} \wwh that is,
\w[.]{\hpu{i} X\in\Gpd^{n}}

It remains to show that \w[.]{\hpu{i}X=\Pup{i}X} Since the (iterated) nerve
functor is fully faithful, again using Proposition \ref{ptadjoint},
we see that for any $n$-fold groupoid $Y$ we have natural isomorphisms
\begin{equation*}
\begin{split}
 & \Hom_{\Gpd^{n}}(\hpu{i} X,Y)~\cong~\Hom_{\sCx{n-2}{\sC{\Gpd}}}
        (\hpu{i}\tX,~\tY)= \\
 & ~=~\Hom_{\sCx{n-2}{\sC{\Gpd}}}(\tX,~\Nup{i}\tY)
         ~=~\Hom_{\sC{\Gpd^{n-1}}}(X,~\Nup{i} Y)~.
\end{split}
\end{equation*}
Hence \w{\hpu{i}} is left adjoint to \w[,]{\Nup{i}} as required.
\end{proof}

%
%
\begin{prop}\label{pntfibor}
If \w{X\in\sS} is a Kan complex, then \w{\osl{n} X} (cf.\ \S \ref{sorsum})
is \wwb{n,2}fibrant.
\end{prop}

See Appendix for the proof.

\begin{thm}\label{tadjoint}
The functor \w{\Qlo{n}} of \S \ref{dladjoint}, applied to a Kan complex
\w[,]{X\in\sS} is:
\begin{equation*}
\Qlo{n}X~=~\hpu{1}\hpu{2}\ldots\hpu{n}\osl{n} X~.
\end{equation*}
\end{thm}

\begin{proof}
We prove the Theorem by induction on \w[.]{n\geq 2} For \w[,]{n=2}
see \cite[Corollary 2.12]{BPaolT}.
Suppose the claim holds for \w[.]{n-1} The left adjoint
\w{\Plo{n}\col \Snx\to\Gpd^{n}} to \w{\Nlo{n}} is the composite
$$
\Snx\cong\sC{\Sx{n-1}}~\supar{\oPl{n-1}\up{1}}~\sC{\Gpd^{n-1}}~
\supar{\Pup{1}}~\Gpd^{n}~,
$$
\noindent where \w{\oPl{n-1}\up{1}} is induced by applying \w{\Plo{n-1}} in each
dimension in the first simplicial direction, and \w{\Pup{1}} is left adjoint
to the nerve \w[.]{\Nup{1}\col \Gpd^{n}\to\sC{\Gpd^{n-1}}} By
the induction hypothesis and \wref[,]{eqovt}
\begin{equation*}
\begin{split}
&~\oPl{n-1}\up{1}\osl{n}X~=~\oPl{n-1}\up{1}\ovt{2}{n-1}\osl{2}X~=~\\
=&~\ovl{Q}_{(n-1)}\osl{2}X~=~\hpu{2}\ldots\hpu{n}\ovt{2}{n-1}\osl{2}X~
=~\hpu{2}\ldots\hpu{n}\osl{n} X~.
\end{split}
\end{equation*}
Since $X$ is a Kan complex, \w{\osl{n} X} is \wwb{n,2}fibrant by Proposition
\ref{pntfibor}. Therefore, by Lemma \ref{lfibrant},
\w{\hpu{2}\ldots\hpu{n}\osl{n} X} is \wwb{n,2}fibrant. It follows by
Proposition \ref{padjoint} that
$$
\Pup{1}\hpu{2}\ldots\hpu{n}\osl{n}X~=~\hpu{1}\hpu{2}\ldots\hpu{n}\osl{n} X.
$$
\noindent Therefore,
\begin{equation*}
\begin{split}
\Qlo{n} X~=&~\Plo{n}\osl{n}X~=~\Pup{1}\oPl{n-1}\osl{n}X\\
=&~\Pup{1}\hpu{2}\ldots\hpu{n}\osl{n}X~=~\hpu{1}\hpu{2}\ldots\hpu{n}\osl{n}X~,
\end{split}
\end{equation*}
\noindent which concludes the induction step.
\end{proof}

\begin{remark}\label{reprod}
The functor \w{\osl{n}\col \sS\to\Snx} has a right adjoint, a generalized
Artin-Mazur codiagonal (cf.\ \cite[\S III]{AMaV} and
\cite{BCDusC,CGarC}), so both \w{\osl{n}} and \w{\Plo{n}} \wwh and thus
\w{\Qlo{n}} \wwh preserve colimits, and in particular coproducts.

On the other hand, clearly \w{\osl{n}} and \w{\hpi} preserve products when
applied to Kan complexes, so \w{\Qlo{n}} does, too. Therefore, \w{\Qlo{n}}
preserves fiber products over discrete simplicial sets.
\end{remark}

%
%
\sect{Weakly globular $n$-fold groupoids}
\label{cntng}

We now introduce the central notion of this paper: that of
a weakly globular $n$-fold groupoid. We will show in the next Section that
the fundamental $n$-fold groupoid \w{\hQlo{n}T} of a space $T$
(see \S \ref{dladjoint}) is such an object.

\supsect{\protect{\ref{cntng}}.A}
{Homotopically discrete $\mathbf{n}$-fold groupoids}

A homotopically discrete groupoid $G$ is one in which there is at most
one arrow between every two objects (that is, all automorphism groups
are trivial). Hence its classifying space is homotopically trivial
(that is, a disjoint union of contractible spaces \wh i.e., a $0$-type).
For such a $G$, the set of arrows \w{G\sb{1}} is simply
\w[.]{G\sb{0}\times\sb{\pi\sb{0}G}G\sb{0}}

In order to provide a higher-dimensional analogue of this notion, we observe
that this construction can be made in any category with suitable (co)limits,
so we can iterate it.  For this purpose we make the following

\begin{defn}\label{dgc}
Let \w{f\col A\to B} be a morphism in a category $\cC$ with finite
 limits. The diagonal map defines a unique section \w{s\col A\to\ata} (so that
\w[,]{p_{1}s=\Id_{A}=p_{2}s} where \w{\ata} is the pullback of
 \w{A\xrw{f}B\xlw{f}A} and \w{p_{1},p_{2}\col \ata\to A} are the two
projections). The commutative diagram
\begin{equation*}
    \xymatrix{
    \ata \ar[rr]^{p_{1}} \ar[d]_{p_{2}} && A \ar[d]_{f} &&
    \ata \ar[ll]_{p_{2}} \ar[d]^{p_{1}}\\
    A \ar[rr]_{f} && B && A \ar[ll]^{f}
    }
\end{equation*}
defines a unique morphism \w{m\col (\ata)\tiund{A}(\ata)\to\ata} such that
\w{p_{2}m=p_{2}\pi_{2}} and \w[,]{p_{1}m=p_{1}\pi_{1}} where \w{\pi_{1}} and
\w{\pi_{2}} are the two projections. We  denote by \w{A^{f}} the following
object of \w[:]{\Catc}
\mydiagram[\label{eqintgpd}]{
(\ata)\tiund{A}(\ata)\ar[rr]^<<<<<<<<{m} &&
\ata \ar@<1.5ex>[rr]^{p_{1}} \ar@<-.5ex>[rr]^{p_{2}} && A \ar@<1.5ex>[ll]^{s}
}
\noindent It is easy to see that \w{A^{f}} is an internal groupoid.
\end{defn}

\begin{defn}\label{dhdng}
We define a full subcategory \w{\Ghd{n}\subset\Gpd^{n}} of
\emph{homotopically discrete $n$-fold groupoids} by induction on
\w[:]{n\geq 1}

A groupoid is called \emph{homotopically discrete} if \w{G\cong A^{f}} for some
surjective map of sets \w[.]{f\col A\to B} In general, an $n$-fold groupoid
\w{G\in\Gpd^{n}} is \emph{homotopically discrete} if \w{G\cong A^{f}} for some
map \w{f\col A\to B} in \w{\Ghd{n-1}} with a section \w{f'\col B\to A} (that is,
\w[).]{f\circ f'=\Id\sb{B}}
\end{defn}

\noindent As noted above, for an (ordinary) groupoid $G$ this just means that
$\pi_{1}(\diN G,x)=$ $0$ for any \w[.]{x\in G_{0}}
\begin{remark}\label{addrem}
Note that the category \w{\Ghd{n}} is closed under pullbacks. We show this by
induction on $n$. When $n=1$, let \w[,]{f\col A\to B} \w[,]{f'\col A'\to B'} and
\w{g\col C\to D} be surjections in \w[.]{\Set} Then
\begin{myeq} [\label{addrem.eq1}]
A^f\tiund{C^g}A'^{f'}= (A\tiund{C}A')\up{f,f'}
\end{myeq}
where \w{(f,f')\col \tens{A}{C}\to \tens{A'}{C'}} is a surjection in \w[.]{\Set}
Thus \w[.]{A^f\tiund{C^g}A'^{f'}\in \Ghd{}}

Suppose the statement holds for \w[,]{n-1} and let \w{f'\col A'\to B'} and
\w{g\col C\to D} be maps with sections in \w[.]{\Ghd{n-1}} Then
\w{(f,f')\col \tens{A}{C}\to \tens{A'}{C'}} is a map in \w{\Ghd{n-1}} with
a section, by the inductive hypothesis, and \eqref{addrem.eq1} holds, showing that
\w[.]{A^f\tiund{C^g}A'^{f'}\in \Ghd{n}}
\end{remark}

\begin{example}\label{eghdtg}
Given a commuting (inner) square of sets:
\mydiagram[\label{eqdiscsquare}]{
    A \ar[rr]_{f} \ar[d]^{g} && B \ar[d]_{h} \ar@/_1pc/[ll]_{f'}  \\
    C \ar[rr]^{\ell} \ar@/^1pc/[u]^{g'} && D \ar@/_1pc/[u]_{h'} \ar@/^1pc/[ll]^{\ell'}
}
\noindent with \w[,]{ff'=\Id\sb{B}} \w[,]{gg'=\Id\sb{C}}
\w[,]{hh'=\ell\ell'=\Id\sb{D}} and \w[,]{fg'=h'\ell}  we obtain a morphism of
homotopically discrete groupoids \w[.]{v\col A^f\to C^{\ell}} The homotopically discrete
double groupoid $G$ associated to $v$ is  described in Figure \ref{fseq}, where we
abbreviate \w{(\tens{A}{B})\tiund{(\tens{C}{D})}(\tens{A}{B})} by
\w[,]{(\tens{A}{B})\tiund{(g,g)}(\tens{A}{B})} and so on.

%
%
\begin{figure}[htb]
\scriptsize{
\begin{picture}(420,90)(40,0)
%
%
\put(200,80){$(\tens{A}{B}\tiund{B}A)\tiund{(g,g,g)}(\tens{A}{B}\tiund{B}A)$}
\put(350,86){\vector(1,0){20}}
\put(350,82){\vector(1,0){20}}
\put(270,72){\vector(0,-1){20}}
\put(375,80){$A\tiund{B}A\tiund{B}A$}
\put(370,77){\vector(-1,0){20}}
\put(395,72){\vector(0,-1){20}}
%
%
\put(40,40){$(\tens{A}{B})\tiund{(g,g)}(\tens{A}{B})\tiund{(g,g)}(\tens{A}{B})$}
\put(185,44){\vector(1,0){30}}
\put(110,33){\vector(0,-1){20}}
\put(115,33){\vector(0,-1){20}}
\put(120,13){\vector(0,1){20}}
\put(220,40){$(\tens{A}{B})\tiund{(g,g)}(\tens{A}{B})$}
\put(325,46){\vector(1,0){50}}
\put(325,42){\vector(1,0){50}}
\put(260,33){\vector(0,-1){20}}
\put(265,33){\vector(0,-1){20}}
\put(270,13){\vector(0,1){20}}
\put(380,40){$\tens{A}{B}$}
\put(375,37){\vector(-1,0){50}}
\put(390,33){\vector(0,-1){20}}
\put(395,33){\vector(0,-1){20}}
\put(400,13){\vector(0,1){20}}
%
%
\put(85,0){$\tens{A}{C}\tiund{C}A$}
\put(140,3){\vector(1,0){105}}
\put(250,0){$\tens{A}{C}$}
\put(290,6){\vector(1,0){95}}
\put(290,2){\vector(1,0){95}}
\put(390,0){$A$}
\put(385,-3){\vector(-1,0){95}}
\end{picture}
}
\caption{A homotopically discrete double groupoid}
\label{fseq}
\end{figure}

Note that
$$
(\tens{A}{B})\tiund{(g,g)}(\tens{A}{B})~\cong~
(\tens{A}{C})\tiund{(f,f)}(\tens{A}{C})
$$
\noindent via the map \w[,]{(a,b,c,d)\mapsto(a,c,b,d)} and more generally
$$
\pro{(\tens{A}{B})}{(g,g)}{k}~\cong~\pro{(\tens{A}{C})}{(f,f)}{k}
$$
\noindent for each \w[.]{k\geq 2} It follows that
\begin{myeq}[\label{eqnerveone}]
(N\up{1}G)_{k-1}=
      \begin{cases}
        A^{f}, & \text{if~}k=1; \\
        (\pro{A}{C}{k})\up{f,\ldots, f} & \text{if~}k\geq 2
      \end{cases}
\end{myeq}
\noindent and
\begin{myeq}[\label{eqnervetwo}]
(N\up{2}G)_{k-1}=
      \begin{cases}
        A^g, & \text{if~}k=1; \\
        (\pro{A}{B}{k})\up{g,\ldots, g} & \text{if~}k\geq 2~.
      \end{cases}
\end{myeq}
\noindent Therefore \w{(N\up{1}G)_{k}} and \w{(N\up{2}G)_{k}} are
homotopically discrete groupoids for all \w[.]{k\geq 0}

Moreover, applying \w{\pi_{0}} vertically to each column in Figure
\ref{fseq} yields the groupoid \w[,]{B^{h}} that is:
\mydiagram[\label{eqpizero}]{
\tens{B}{D}\tiund{D}B\ar[r] &
\tens{B}{D}\ar@<1.0ex>[r]\ar[r] & B \ar@<1.0ex>[l]
}
\noindent Similarly, applying \w{\pi_{0}} horizontally in each row
yields \w[.]{C^{\ell}}
\end{example}

\begin{remark}\label{rhtdg}
The construction of \S \ref{eghdtg} makes sense in any category with enough
limits. Conversely, any map \w{v\col A\sp{f}\to C\sp{\ell}} with a section \w{v'}
has a map of objects \w{g\col A\to C} and induces a map \w{h\col B\to D}
on \w[,]{\pi\sb{0}} which fits into a commuting square as in
\wref[.]{eqdiscsquare}
\end{remark}

\begin{defn}\label{dbpz}
Recall from \S \ref{nsimp}(g) that if \w{X\in\sCx{n-1}{\Gpd}} is an
\wwb{n-1}fold simplicial object in groupoids, \w{\opz{n}X} is the
\wwb{n-1}fold simplicial set obtained by applying \w{\pi_{0}} (the coequalizer
of the source and target maps of the groupoid) in each \wwb{n-1}fold
simplicial dimension of $X$. If \w{\cons X} denotes the discrete
groupoid on a set $X$ (cf.\ \S \ref{ddiscgpd}), \w{c\col \Set\to\Gpd} is right
adjoint to \w[,]{\pi\sb{0}} and the unit of the adjunction
\w{\gamma\col \Id\to c\pi\sb{0}} induces a natural transformation of
\wwb{n-1}simplicial groupoids
$$
\ovl{\gamma}~\col ~X~\to~\ovl{c}\up{n}\,\opz{n}X~.
$$
\end{defn}

\begin{remark}\label{rbpz}
Let \w{G\in\Gpd\sp{n}} be an $n$-fold groupoid and
$$
{X=\Nup{n-1}\dotsc\Nup{1}G\in\sCx{n-1}{\Gpd}}\;.
$$
Let us suppose that \w{\opz{n}X} is the multinerve of an \wwb{n-1}fold
groupoid, denoted by \w[,]{\bPz{n}G} so that
$$
\opz{n}X= N\lo{n-1}\bPz{n}G\;.
$$
\noindent Then \w{\ovl{c}\up{n}\,\opz{n}X} is the multinerve of an $n$-fold
groupoid \w{\cons\up{n}\bPz{n}G} (discrete in the new $n$-th direction)
and
$$
\ovl{\gamma}~=~\Nup{n-1}\cdots \Nup{1}\gamma\up{n}
$$
for a map of $n$-fold groupoids \w[.]{\gamma\up{n}\col G\to \w{\cons\up{n}\bPz{n}G}}
\end{remark}

\begin{remark}\label{rfibprod}
Since \w{\pi_{0}\col \Gpd\to\Set} preserves products and coproducts, it
preserves fiber products over discrete groupoids. Therefore, the same
is true of \w[.]{\opz{n}}
\end{remark}

\begin{lemma}\label{lhdng}
Let \w{G\in\Ghd{n}} be a homotopically discrete $n$-fold groupoid. Then:
\begin{enumerate}
\renewcommand{\labelenumi}{(\alph{enumi})~}
\item If \w{\Nup{i}\col \Gpd^{n}\to\sC{\Gpd^{n-1}}} for some
\w{1\leq i\leq n} is as in \ \S \ref{addnote}, then \w{(\Nup{i}G)_{k}} is
homotopically discrete for all \w[.]{k\geq 0}
\item The \wwb{n-1}simplicial set \w{\opz{n} \Nup{n-1}\cdots \Nup{1}G} is
the multinerve of a homotopically discrete \wwb{n-1}fold groupoid
\w[,]{\bPz{n}G} and there is a commutative diagram
$$
\xymatrix{
\Ghd{n} \ar^{\Nup{n-1}\cdots \Nup{1}} [rrr] \ar_{\bPz{n}}[d] &&&
\sCx{n-1}{\Gpd} \ar^{\opz{n}}[d]\\
\Ghd{n-1} \ar^{\Nlo{n-1}}[rrr] &&& \Sx{n-1}
}
$$
\item The map of $n$-fold groupoids \w{\gup{n}\col G\to\cons\up{n}\bPz{n}G} of
\S \ref{rbpz} is a geometric weak equivalence (\S \ref{dwecs}).
\item The set \w{\bPz{1}\dotsc \bPz{n}G} is isomorphic to
\w{\pi_{0}\diN G} (cf.\ \S \ref{dwecs}).
\item If we let \w{\Glo{n}} denote the composite
\begin{equation*}
\begin{split}
G~\supar{\gamma\up{n}}~\cons\up{n}\bPz{n}G~\supar{\cons\up{n}\gamma\up{n-1}}&~
\cons\up{n-1}\cons\up{n}\bPz{n-1}\bPz{n}G\\
~\cdots & ~\to ~\cons\up{1}\dotsc\cons\up{n}\bPz{1}\dotsc\bPz{n}G~,
\end{split}
\end{equation*}
\noindent it induces a geometric weak equivalence:
\begin{myeq}[\label{eqhd}]
\iseg{k}~\col ~\pro{G_{1}}{G_{0}}{k}~\to~
\pro{G_{1}}{G_{0}^{d}}{k}\hsp\text{for all} \hs k\geq 2
\end{myeq}
\noindent (where \w{G_{0}^{d}} is as in \S \ref{ddiscgpd}).
\end{enumerate}
\end{lemma}

\begin{proof}
By Definition \ref{dhdng}, $G$ (as an object of
\w[)]{\hy{\Gpd^{2}}{\Gpd^{n-2}}} has the form of Figure
\ref{fseq} for some commuting square:
\mydiagram[\label{eqcds}]{
    A \ar[rr]_{f} \ar[d]^{g} && B \ar[d]_{h} \ar@/_1pc/[ll]_{f'}  \\
    C \ar[rr]^{\ell} \ar@/^1pc/[u]^{g'} && D \ar@/_1pc/[u]_{h'} \ar@/^1pc/[ll]^{\ell'}}
\noindent of \wwb{n-2}fold groupoids, as in \wref[,]{eqdiscsquare} by
Remark \ref{rhtdg}.

\noindent (a)~By \wref{eqnerveone} and \wref[,]{eqnervetwo} the
statement holds for \w[.]{n=2} Suppose by induction that it holds for
\w[:]{n-1}  then \w{(\Nup{1}G)_{0}=A^{f}} is in \w[.]{\Ghd{n-1}} Also
\w{(\Nup{1}G)_{k-1}=(\pro{A}{C}{k})^{(f, \cdots,f)}} for
\w[.]{k\geq 2} By definition and the induction hypothesis,
$$
(f,\ldots,f)\col \pro{A}{C}{k}~\lto~\pro{B}{D}{k}
$$
is a morphism with a section in \w[,]{\Ghd{n-1}} Hence, by definition,
\w[.]{(\Nup{1}G)_{k-1}\in\Ghd{n-1}} Similarly for any \w[\vsm .]{\Nup{i}G}

\noindent (b)~ By \wref{eqcds} and (a), \w{\Nup{1}\opz{n}G} is the
nerve of the \wwb{n-1} fold homotopically discrete groupoid
\w[,]{\bPz{n}G:=B^{h}}  and the map \w{\hgup{n}} lifts to a map of
$n$-fold groupoids\vsm .

\noindent (c)-(e)~ By induction on \w[.]{n\geq 2}
For \w[,]{n=2} we saw that \w[,]{\bPz{2}G=B^{h}} and since each column
in Figure \ref{fseq} is homotopically discrete,
we see from \wref{eqpizero} that the rightmost column is equivalent to
$B$, the next to \w[,]{B\times_{D}B} and so on. Thus
\w{\Nup{1}\gup{2}\col \Nup{1}G\to\Nup{1}c\bPz{2}G} induces
dimensionwise weak equivalences of simplicial spaces, so a weak
equivalence of classifying spaces.  Since \w{B^{h}} is a homotopically
discrete groupoid, it is weakly equivalent to \w{\cons D} (in the notation
of \wref[),]{eqcds} which is \w[.]{(\pi_{0}\diN G)^{d}}

By \wref{eqnerveone} for each \w[:]{k\geq 2}
$$
\pro{G_{1}}{G_{0}}{k}~=~(N\up{1}G)_{k-1}~=~(\pro{A}{C}{k})\up{f,\ldots,f}
~\cong~\pro{B}{D}{k}~,
$$
\noindent while since \w{G_{1}} is homotopically discrete and \w{G\sb{0}\sp{d}}
is discrete,  \w{\pro{G_{1}}{G_{0}^{d}}{k}} is homotopically discrete
(see Remark \ref{addrem}), so it is also weakly equivalent to
\w[.]{\pro{B}{D}{k}} Thus \wref{eqhd} holds for \w[.]{n=2}

In the induction step, \w{\Nup{1}G} is a simplicial \wwb{n-1}fold
homotopically discrete groupoid (by \wref{eqnerveone} again), and thus
by the induction hypothesis for \w{n-1} we have a weak equivalence
$$
(\Nup{1}\Glo{n-1})_{r}\col (\Nup{1}G)_{r}~\to~
(\cons\up{2}\dotsc\cons\up{n}\Nup{1}\bPz{2}\dotsc\bPz{n}G)_{r}~=\col P_{r}
$$
\noindent in each simplicial dimension \w[.]{r\geq 0} Applying
the \wwb{n-1}fold nerve \w{\Nlo{n-1}} to both sides, we obtain a map
of $n$-fold simplicial sets \w{\Nlo{n}G\to\Pd} which is a weak
equivalence in each simplicial
dimension, so induces a weak equivalence
$$
\Dlo{n}\Nlo{n}G\to\Dlo{n}\Pd~.
$$
\noindent However, \w{\Pd} is discrete in all but the first simplicial
direction, where it is (the nerve of) a homotopically discrete
groupoid \w[.]{H:=\bPz{2}\dotsc\bPz{n}G}  In fact,
\w[,]{H=(B^{d})^{h^{d}}} in the notation of \S \ref{dgc}, where
\w{h^{d}\col B^{d}\to D^{d}} is the  discretization of the map \w{h\col B\to D}
in \wref[.]{eqcds}

Therefore, \w{\Dlo{n}\Pd=\diN H} has
\w{\pi_{0}\diN H=\pi_{0}H^{d}=\pi_{0}\diN G} while \w{\pi_{i}\diN H=0}
for \w[,]{i\geq 1} and the map
\w{\Glo{n}=\gup{1}\circ\Glo{n-1}} induces the requisite
weak equivalence.  Since also
\w[,]{\Glo{n}=\cons\up{n}\Glo{n-1}\circ\gup{n}} we deduce by
induction that \w{\gup{n}} is a geometric weak equivalence, too.

To show \wref[,]{eqhd} note that by \wref{eqpizero} we have:
\begin{equation*}
\begin{split}
(\bPz{n}G)_{2}~=&~\bPz{n-1}(\tens{G_{1}}{G_{0}})~=~\tens{B}{D}\tiund{D}B\\
~=&~      (\tens{B}{D})\tiund{B}(\tens{B}{D})~,
\end{split}
\end{equation*}
\noindent which by the induction hypothesis \wref{eqhd} and Remark
\ref{rfibprod} equals:
\begin{equation*}
\begin{split}
\tens{\bPz{n-1}G_{1}}{\bPz{n-1}G_{0}}~\simeq&~
\tens{\bPz{n-1}G_{1}}{(\bPz{n-1}G_{0})^{d}}\\
=&~\bPz{n-1}(\tens{G_{1}}{G_{0}^{d}})~.
\end{split}
\end{equation*}
\noindent That is, we have a commuting square
$$
\xymatrix{
\tens{G_{1}}{G_{0}}\ar[rr]^{\gup{n-1}} \ar[d]_{\iseg{2}} &&
\bPz{n-1}(\tens{G_{1}}{G_{0}}) \ar[d]^{\simeq}\\
\tens{G_{1}}{G_{0}^{d}} \ar[rr]^{\gup{n-1}} && \bPz{n-1}(\tens{G_{1}}{G_{0}^{d}})~
}
$$
\noindent in which three of the maps are geometric weak equivalences, so
\w{\iseg{2}} is, too.

Similarly for all \w[.]{k>2}
\end{proof}

From (d) of the Lemma we see:

\begin{corollary}\label{chdng}
If $G$ is a homotopically discrete $n$-fold groupoid, the map
\w{\Glo{n}\col G\to\cons\up{1}\dotsc\cons\up{n}\bPz{1}\dotsc\bPz{n}G}
is a geometric weak equivalence, so \w{\diN G} is homotopically trivial
(i.e., \w{\pi_{i}\diN G=0} for all \w[).]{i\geq 1}
\end{corollary}

\supsect{\protect{\ref{cntng}}.B}
{Weakly globular $\mathbf{n}$-fold groupoids}

We are now in a position to define the main notion of this section.
At first glance, it does not appear to be fully algebraic, since it
uses the concept of a geometric weak equivalence.
However, as we shall show in Corollary \ref{chtpygp}
below, this concept has an equivalent purely algebraic description.

\begin{defn}\label{dntng}
For each \w[,]{n\geq 1}  the full subcategory \w{\Gpt{n}} of \w[,]{\Gpd^{n}}
whose objects are called  \emph{weakly globular $n$-fold groupoids}, is
defined by induction on $n$, as follows:

For \w[,]{n=1} any groupoid is weakly globular; suppose we have
defined \w[.]{\Gpt{n-1}} We say that an $n$-fold groupoid
\w{G=(G\up{1}_{1}\toto G\up{1}_{0})} is \emph{weakly globular} if
\begin{enumerate}
\renewcommand{\labelenumi}{(\roman{enumi})}
\item \www{G_{0}:=G\up{1}_{0}} is in \w[;]{\Ghd{n-1}}
\item \www{G_{1}:=G\up{1}_{1}} is in \w[,]{\Gpt{n-1}} and for each
  \w[,]{k\geq 2} \w{\pro{G_{1}}{G_{0}}{k}} is in \w[;]{\Gpt{n-1}}
\item The \wwb{n-1}simplicial set \w{\opz{n}\Nup{n-1}\cdots \Nup{1}G} is
the nerve of a weakly globular \wwb{n-1}fold groupoid \w{\bPz{n}G} such that
$$
\Nlo{n-1} \bPz{n}G = \opz{n}\Nup{n-1}\cdots \Nup{1}G\;.
$$
\item The map of \wwb{n-1}fold groupoids
$$
\pro{G_{1}}{G_{0}}{k}~\xrw{\iseg{k}}~\pro{G_{1}}{G_{0}^{d}}{k}
$$
\noindent induced by \w{\gamma_{(n)}\col G_{0}\to G_{0}^{d}} is a geometric weak
equivalence for all \w[.]{k\geq 2}
\end{enumerate}

Note the special role played by the \emph{first} of the $n$-directions in this
definition. Also, note that we have a functor \w{\bPz{n}} making the
following diagram commute:
$$
\xymatrix{
\Gpt{n} \ar^{\Nup{n-1}\cdots \Nup{1}} [rrr] \ar_{\bPz{n}}[d] &&&
\sCx{n-1}{\Gpd} \ar^{\opz{n}}[d]\\
\Gpt{n-1} \ar^{\Nlo{n-1}}[rrr] &&& \Sx{n-1}.
}
$$
\end{defn}

\begin{remark}\label{rgpoid}
For \w[,]{n=2} the above definition is slightly more general than
\cite[Definition 2.21]{BPaolT}. In fact, in \cite{BPaolT} $G$ is required to
be symmetric, and both maps \w{G_{1}\toto G_{0}} are required to be
fibrations of groupoids; the latter implies conditions (iii) and (iv).

Note also that if \w[,]{G\in\Gpt{n}} not only is
\w{\pro{G_1}{G_0}{k}\in\Gpt{n-1}} (by Definition \ref{dntng}), but also
\w[.]{\pro{G_1}{G_0^d}{k}\in\Gpt{n-1}} We show this for \w[,]{k=2} the general
case being similar. In fact we observe more generally that the pullback $P$
of \w{G \to H \leftarrow G'} with \w{G,\,G'} in \w{\Gpt{n}} and \w{H} discrete
is an object of \w[.]{\Gpt{n}}

We proceed by induction on \w[:]{n} For \w{n=1} the statement is clear, since
\w[.]{\Gpt{1}=\Gpd} Suppose it is true for \w[.]{n-1} We have
\w{P_0=G_0\tiund{H_0}G'_0 \in \Ghd{n-1}} since \w[,]{G_0,\,G'_0\in \Ghd{n-1}}
and \w{H\sb{0}} is discrete (using Remark \ref{addrem}). Furthermore,
\w{P_1=G_1\tiund{H_1}G'_1 \in \Ghd{n-1}} by the induction hypothesis.

Likewise, since \w{H} is discrete,
\begin{myeq}\label{rgpoideq1}
\begin{split}
   & \tens{P_1}{P_0} \cong (\tens{G_1}{G_0})\tiund{(\tens{H_1}{H_0})}
(\tens{G'_1}{G'_0}) = \\
    & =(\tens{G_1}{G_0})\tiund{H_0} (\tens{G'_1}{G'_0})\;.
\end{split}
\end{myeq}
\noindent Thus \w{\tens{P_1}{P_0} \in\Gpt{n-1}} by the induction hypothesis.
For the same reason, \w{\pro{P_1}{P_0}{k}} is in \w[.]{\Gpt{n-1}}
Since \w{\pi\sb{0}} commutes with fiber products over discrete objects,
we have \w[,]{\opz{n}P = \bPz{n}G \tiund{H}\bPz{n}G} and this is in
\w{\Gpt{n-1}} by the induction hypothesis.

Finally,
\begin{myeq}\label{rgpoideq2}
 \tens{P_1}{P^d_0} = (\tens{G_1}{G^d_0})\tiund{H_0} (\tens{G'_1}{G'^d_0})\;.
\end{myeq}
Since there are geometric weak equivalences
\w{\tens{G_1}{G_0}\to \tens{G_1}{G^d_0}} and
\w[,]{\tens{G'_1}{G'_0}\to \tens{G'_1}{G'^d_0}} by \wref{rgpoideq1} and
\wref{rgpoideq2} this induces a geometric weak equivalence
$$
\tens{P_1}{P_0} \to \tens{P_1}{P^d_0}\;.
$$
Similarly, one shows that for each \w[,]{k\geq 2} there is a geometric
weak equivalence
$$
\pro{P_1}{P_0}{k} \to \pro{P_1}{P^d_0}{k}\;.
$$
This completes the proof that \w[.]{P\in \Gpt{n}}
\end{remark}

\begin{defn}\label{diao}
For any $n$-fold groupoid $G$ and \w[,]{1\leq k\leq n} we
define its \emph{$k$-fold object of arrows} to be the \wwb{n-k}fold groupoid:
$$
\Wlo{n,k}G~:=~G\up{1\dotsc k}_{1\underset{k}{\cdots}1}~,
$$
\noindent using the indexing conventions of \S \ref{nsimp}(b).
\end{defn}

\begin{remark}\label{cwg}
Note that by Definition \ref{dntng}(ii), if $G$ is weakly globular,
so is \w[,]{\Wlo{n,1}G} so by induction we have a functor
\w[,]{\Wlo{n,k}\col \Gpt{n}\to\Gpt{n-k}} since
\begin{myeq}[\label{eqiao}]
\Wlo{n,k}~=~\Wlo{n-k+1,1}\Wlo{n-k+2,1}\dotsc\Wlo{n-1,1}\Wlo{n,1}~.
\end{myeq}
\noindent
\end{remark}

\begin{mysubsection}{Algebraic homotopy groups and algebraic weak
equivalences}
\label{sahgp}
For any weakly globular $n$-fold groupoid $G$, we define the
\emph{$k$-th algebraic homotopy group} of $G$ at
\w{x_{0}\in G_{0\underset{n}{\cdots}0}} to be:
\begin{myeq}[\label{eqpinag}]
\omega_{k}(G;x_{0})~\cong~
\begin{cases}
\Wlo{n,n}G(x_{0},x_{0}) & \hs\text{if}~k=n\\
\Wlo{n-k,n-k}(\bPz{k+1}\dotsc\bPz{n}G)(x_{0},x_{0}) &
\hs\text{if}~0<k<n\\
\end{cases}
\end{myeq}
with the \emph{$0$-th algebraic homotopy set} of $G$ defined:
$$
\omega_{0}(G)~:=~\bPz{1}\dotsc\bPz{n}G~.
$$
\noindent Here \w{\Wlo{n,n}G(a,b)} (cf.\ \S \ref{diao}) is the set of morphisms
from $a$ to $b$ in the groupoid \w{\Wlo{n,n-1}G} (in the $n$-th direction),
so in particular \w{\Wlo{n,n}G(a,a)} is the
group of automorphisms of $a$ (which is abelian for \w[).]{n\geq 2}

A map \w{f\col G\to G'} of weakly globular $n$-fold groupoids is called an
\emph{algebraic weak equivalence} if it induces bijections on the $k$-th
algebraic homotopy groups (set) for all \w{x_{0}\in G_{0\underset{n}{\cdots}0}} and
\w[.]{0\leq k\leq n}
\end{mysubsection}

\begin{defn}\label{dpostnikov}
For each \w[,]{n\geq 0} let \w{\PT{n}} denote the full subcategory of \w{\Top}
consisting of spaces $X$ for which the natural map
\w{X\to \Po{n}X} is a weak equivalence (that is, \w{\pi_{i}(X,x)=0} for
all \w{x\in X} and
\w[).]{i>n} An \emph{$n$-type} is an object in \w{\PT{n}} (or in the
corresponding full subcategory \w{\ho(\PT{n})} of \w[).]{\ho\Top}

We use similar notation for $n$-Postnikov simplicial sets (where for a Kan
complex $X$ (cf.\ \cite[I.3]{GJarS}), we can use \w{\csk{n+1}X}
as a model for the $n$-th Postnikov section \w[).]{\Po{n}X}

For any \w[,]{n\geq 0} a map \w{f\col X\to Y} in \w{\sS} (or in \w[)]{\Top}is
called an \emph{$n$-equivalence} if it induces isomorphisms
\w{f_{\ast}\col \pi_{0}X\to\pi_{0}Y} (of sets), and
\w{f_{\#}\col \pi_{i}(X,x)\to\pi_{i}(Y,f(x))} for every \w{1\leq i\leq n} and
\w[.]{x\in X_{0}}
\end{defn}

We recall the following notion and fact from \cite{BPaolT}:

\begin{defn}\label{dndiag}
A map \w{f\col W\to V} of bisimplicial sets is called a
\emph{diagonal $n$-equivalence} if \w{f_{k}^{h}\col W_{k}^{h}\to V_{k}^{h}} is an
\ww{(n-k)}-equivalence for each \w[.]{k\leq n}
\end{defn}

%
%
\begin{prop}[\protect{\cite[Prop.\ 3.9]{BPaolT}}]\label{pndiag}
If \w{f\col W\to V} is a diagonal $n$-equivalence, then the induced map
\w{\Diag f\col \Diag W\to\Diag V} is an $n$-equivalence.
\end{prop}

\begin{lemma}\label{lnmoeq}
For any \w[,]{G\in\Gpt{n}} the map $\ovl{\gamma}$
of Definition \ref{dbpz}  corresponds to a map of $n$-fold
groupoids \w{\gup{n}\col G\to\cons\up{n}\bPz{n}G} with
\w[,]{\ovl{\gamma}= N\up{n-1}\dots N\up{1}\gamma\up{n}} which induces an
\wwb{n-1}equivalence \w{\diN\gup{n}\col \diN G\to\diN\cons\up{n}\bPz{n}G}
on classifying spaces.
\end{lemma}

\begin{proof}
By Definition \ref{dntng} and Remark \ref{rbpz} the map \w{\ovl{\gamma}}
corresponds to a map of $n$-fold groupoids as stated.
We show that this is an \wwb{n-1}equivalence by induction on $n$. It is clear for
\w[.]{n=1} Suppose, inductively, it holds for \w[.]{n-1}

By construction we have
$$
(\bPz{n}G)_{r}~:=~(\Nup{n}\bPz{n}G)\up{n}_{r}~=~\bPz{n-1}(\Nup{n}G)\up{n}_{r}~,
$$
and therefore, for each \w{r\geq 0} there is a map
$$
(\Nup{n}\gup{n-1})_{r}\col (\bPz{n}G)_{r}\to(\cons\up{n}\bPz{n-1}G)_{r}~.
$$
\noindent By taking realizations, we obtain a map of simplicial spaces
\w[.]{\diN\gup{n-1}}  We claim that the corresponding map of
bisimplicial sets is a diagonal \wwb{n-1}equivalence (cf.\ \S
\ref{dndiag}). In fact, since \w{G_{0}=(\Nup{n}G)\up{n}_{0}} is
homotopically discrete, by Lemma \ref{lhdng},
\w{(\diN\gup{n-1})_{0}}is a weak equivalence, hence in
particular an \wwb{n-1}equivalence. By the induction hypothesis
\w{(\diN\gup{n-1})_{r}} is a \wwb{n-2}equivalence for all \w[.]{r\geq 1} Hence
\w{\diN\gup{n-1}} is an \wwb{n-1}equivalence by Proposition \ref{pndiag}.
\end{proof}

\begin{remark}\label{rshdnt}
From Lemmas \ref{lhdng} and \ref{lnmoeq} we see that a
homotopically discrete $n$-fold groupoid is weakly globular.
\end{remark}

%
%
\sect{$n$-Types}
\label{cnt}

In this section we prove one of the main result of this paper, Theorem
\ref{teqcat}, which asserts that all $n$-types are modelled by weakly
globular $n$-fold groupoids.

\supsect{\protect{\ref{cnt}}.A}
{The homotopy type of a weakly globular $\mathbf{n}$-fold groupoid}

We start by showing that if \w[,]{G\in \Gpt{n}} then its
classifying space \w{\diN G} (cf.\ \S \ref{dwecs}) is an $n$-type; that is,
\w{\pi_{i}(\diN G,x)=0} for all \w{x\in\diN G} and \w[.]{i>n}
We prove this using a spectral sequence computation of
\w[.]{\pi_{i}(\diN G,x)} In Section \ref{ctmng}, we give an alternative proof
using a comparison with Tamsamani's weak $n$-groupoids.

In \cite{QuiS}, Quillen constructed a spectral sequence for a bisimplicial
group, which was generalized in \cite[Appendix B]{BFrieH} to define
the \emph{Bousfield-Friedlander spectral sequence} of a bisimplicial
set \w[,]{\Xdd\in\sCx{2}{\Set}} with
\begin{myeq}[\label{eqbfried}]
E^{2}_{s,t}~=~\pi^{h}_{s}\pi^{v}_{t}\Xdd~\Rw~\pi_{s+t}\Diag\Xdd~.
\end{myeq}
\noindent See \cite[\S 8.4]{DKStB} for an alternative construction
when \w{\Xdd} is connected in each simplicial dimension. The spectral sequence
need not converge otherwise; however, we have the following
\emph{sufficient} condition for convergence (cf.\ \cite[B.3]{BFrieH}):

\begin{defn}\label{dpikan}
Think of a bisimplicial set \w{\Xdd\in\Sx{2}} as a
(horizontal) simplicial object in \w{\sS} (with the simplicial
direction inside \w{\sS} thought of as being vertical). In this notation,
a \hyp{k}{\pi_{t}}-\emph{matching collection} at \w{a\in X_{n,0}}
(for \w[)]{0\leq k\leq n} is a set of elements
\w{x_{i}\in\pi_{t}(X_{n-1\bullet},d_{i}^{h}a)} \wb[,]{0\leq i\leq n, i\neq k}
such that:
\begin{myeq}[\label{eqpitm}]
(d_{i}^{h})_{\ast}x_{j}~=~(d_{j-1}^{h})_{\ast}x_{i}
\end{myeq}
\noindent for every \w{0\leq i<j\leq n}
\wb[.]{i,j\neq k}

We say that \w{\Xdd} satisfies the \ww{\pis}-\emph{Kan condition} if
for every \w[,]{n,t\geq 1} \w[,]{0\leq k\leq n}  \w[,]{a\in X_{n,0}} and
\hyp{k}{\pi_{t}}-matching  collection \w{(x_{i})_{i\neq k}^{n}}
at $a$, there is a fill-in \w{w\in\pi^{v}_{t}(X_{n\bullet},a)} such that
\w{(d_{i}^{h})_{\ast}w=x_{i}} for all \w{0\leq i\leq n} \wb[.]{i\neq k}

By \cite[Theorem B.5]{BFrieH}, if \w{\Xdd} satisfies the
  \ww{\pis}-Kan condition \wh for example, if each \w{X_{n\bullet}} is
  connected \wh then the spectral sequence \wref{eqbfried} converges.
\end{defn}

\begin{notation}\label{nsdgpd}
For any simplicial set $Y$ and \w[,]{t\geq 1} the $t$-th
homotopy group \w[,]{\pi_{t}(Y,y)} as \w{y\in Y} varies, constitutes a
\emph{semi-discrete groupoid}, in the sense of
\cite[\S 1]{BPaolT} \wh that is, a disjoint union of groups (abelian,
if \w[).]{t\geq 2} We denote it by \w[.]{\hp{t}Y}
\end{notation}

\begin{lemma}\label{lpikan}
Let \w{\Gd\in\Gpd(\sS)} be a groupoid in \w[,]{\sS} such that
$$
\pro{G_{1}}{G_{0}}{k}~\to~\pro{G_{1}}{\cons\pi\sb{0}G\sb{0}}{k}
$$
\noindent is a weak equivalence of simplicial sets for all \w[,]{k\geq 2}
with \w{G_{0}} a homotopically trivial simplicial set.
Then the bisimplicial set \w{\Xdd:=\cN\Gd} satisfies the \ww{\pis}-Kan
condition, and for each \w[,]{t\geq 1} \w{\hp{t}\Xdd} is  a groupoid
object in semi-discrete groupoids, so is $2$-coskeletal.
\end{lemma}

\begin{proof}
We think of the simplicial direction as vertical.
Let \w[.]{X_{k}=(\cN\Gd)_{k}} Since \w{X\sb{0}=G\sb{0}} is homotopically
trivial (that is, a disjoint union of contractible spaces), the
groupoid \w{\hp{t}X_{0}} is discrete on \w[,]{\pi\sb{0}G\sb{0}}
so any \hyp{k}{\pi_{t}}-matching collection for \w{n=1} is trivial.

For \w[,]{n=2} note that \w[,]{X_{2}=X_{1}\times_{X_{0}}X_{1}} so any \w{a\in
X_{2,0}} is of the form \w[,]{a=(a',a'')} where \w{d_{1}a'=d_{0}a''=\col b}
Moreover, \w[,]{d_{0}a=a'} \w{d_{1}a=a'\star a''} (where $\star$
denotes the groupoid composition), and \w[.]{d_{2}a=a''}

Thus for \w[,]{t\geq 1} there are three cases for a \hyp{k}{\pi_{t}}-matching
collection \w{(x_{i}\in\pi^{v}_{t}(X_{1},d_{i}a))_{i\neq k}} at $a$:

\begin{enumerate}
\renewcommand{\labelenumi}{(\roman{enumi})}
\item When \w[,]{k=1} the fill-in \w{w\in\pi^{v}_{t}(X_{2},a)} for
\w{x_{0}} and \w{x_{2}} is the pull-back pair \w{(x_{0},x_{2})} in
$$
\pi^{v}_{t}(X_{2},a)~=~
\pi^{v}_{t}(X_{1},a')~\times_{\pi^{v}_{t}(X_{0},b)}~\pi^{v}_{t}(X_{1},a'')~.
$$
\item When \w[,]{k=0} the fill-in \w{w=(y,x_{2})} for \w{x_{1}} and  \w{x_{2}}
should satisfy \w[,]{x_{1}=d_{1}w=y\star x_{2}} so
\w[,]{y=x_{1}\star(x_{2})^{-1}} using the groupoid structure on
\w[.]{\hpv{t}X_{1}}
\item The case \w{k=2} is similar.
\end{enumerate}

For \w{n>2} the proof of the \ww{\pis}-Kan condition is analogous;
however, because \w{\hpv{t}\Xdd} is $2$-coskeletal, we do not even
need to verify it, since the spectral sequence \wref{eqbfried} from
the \ww{E^{2}}-term on then depends only on the $2$-truncation of
\w{\Xdd} in the horizontal direction.
\end{proof}

In order to study the homotopy groups of the $n$-fold diagonal \w{\dN G} of
an $n$-fold groupoid, we think of it as an iterative construction in which
we take diagonals in successive bisimplicial bidirections. The weak globularity
allows us to iteratively apply Lemma \ref{lpikan}, and thus the
Bousfield-Friedlander spectral sequence.

\begin{thm}\label{thdnt}
For any weakly globular $n$-fold groupoid \w[,]{G\in \Gpt{n}} \w{\diN G}
is an $n$-type, and for each base point \w{x_{0}\in G_{0\underset{n}{\cdots}0}}
we have natural isomorphisms
\begin{myeq}[\label{eqpintg}]
\pi_{k}(\diN G;x_{0})~\cong~\omega_{k}(G;x_{0})\ \ \text{for} \ \ 0<k\leq n\ \
\text{and} \ \ \pi_{0}\diN G~\cong~\omega_{0}(G)
\end{myeq}
\noindent (see \wref[).]{eqpinag}
\end{thm}

\begin{proof}
Since \w{\diN G} is the geometric realization of \w[,]{\dN G}
we prove the Theorem simplicially, for \w[,]{\dN G}
by induction on $n$.

Using the convention of \S \ref{rnfoldg}, for each \w{\va\in\bDelta^{n-2}}
we have a double groupoid \w{G\up{1,2}(\va)\in\Gpd^{2}} (in the notation
of \S \ref{nsimp}(b)).  Assuming that the first
of the $n$ directions of $G$ is not among those of \w[,]{\bDelta^{n-2}}
\w{\Nup{1}G\up{1,2}(\va)\in\sC{\Gpd}} satisfies the hypotheses of
Lemma \ref{lpikan}, by Definition \ref{dntng}.
Therefore, the Bousfield-Friedlander spectral sequence for the
bisimplicial set
$$
X(\va)~:=~\Nup{1,2}G\up{1,2}(\va)
$$
\noindent converges to \w[.]{\pis\Diag X(\va)} Moreover,
\w{\pi^{v}_{t}X(\va)} is $2$-coskeletal for each \w[,]{t\geq 1}
by the Lemma, as is
\w{\pi^{v}_{0}X(\va)} (by Definition \ref{dntng} again).
Thus in the \ww{E^{2}}-term of the spectral sequence only the two
right columns of two bottom rows can be non-zero, so that
\w{\Diag X(\va)} is a $2$-type. In fact, the rightmost column is zero
(except at the bottom), so we can read off the homotopy groups of
\w{\Diag X(\va)} from those of \w[.]{X(\va)}

Since \w{\Diag} is functorial in \w[,]{\va\in\Dox{n-2}} we see that the
resulting object \w{Y:=\Dup{1,2}\Nup{1,2}G} is in \w[,]{\sC{\Gpd^{n-2}}}
with each \w{Y(\va)\in\sS} a simplicial $2$-type. Since \w{G_{0}} was a
homotopically discrete \wwb{n-1}fold groupoid, the object \w{Y^{v}_{0}}
(in dimension $0$ in the first (simplicial) direction) is a homotopically
discrete \wwb{n-2}fold groupoid. Moreover, for any choice of a third
(groupoid) direction $i$, and each \w[,]{\vb\in\bDelta^{n-3}}
by Definition \ref {dntng}, we have a bisimplicial groupoid
$$
\Zdd{\ast}~:=~\Nup{1,2}G\up{1,2,i}(\vb)
$$
\noindent (where the third index is
the groupoid direction). This has a weak equivalence of bisimplicial sets
$$
\Zdd{k}~=~
\pro{\Zdd{1}}{\Zdd{0}}{k}~\supar{\simeq}~\pro{\Zdd{1}}{G_{0}^{d}}{k}
$$
\noindent for each \w[,]{k\geq 2} natural in $\vb$ (note that
\w{G^{d}} is independent of $\vb$).
This map therefore induces a weak equivalence in the bisimplicial
direction (cf.\ \S \ref{rgwe}). Thus each simplicial
groupoid \w{Y(\vb)=\Diag\Zdd{\ast}} satisfies the hypotheses of Lemma
\ref{lpikan}.

Now assume by descending induction on \w{2\leq k<n} that we have
\w[,]{Y\in\sC{\Gpd\sp{n-k}}} with \w{Y(\va)\in\sS} a $k$-type for
each \w[,]{\va\in\bDelta\sp{n-k}} with \w{Y\sp{v}\sb{0}} a homotopically
discrete \wwb{n-k}fold groupoid. Here the first (vertical) direction
is simplicial.

For any choice of a second (groupoid) direction, and
each \w[,]{\vb\in\bDelta^{n-k-1}} the simplicial groupoid
\w{Y\up{1,2}(\vb)\in\sC{\Gpd}} satisfies the hypotheses of Lemma
\ref{lpikan}. Therefore, \wref{eqbfried} converges, with only the two
right columns of the bottom $k$ rows non-zero, and
\w{\Diag Y(\va)} is thus a \wwb{k+1}type.
When \w[,]{k=n-1} $Y$ is a simplicial groupoid which is an
\wwb{n-1}type in the simplicial direction, with \w{\diN G}
appearing as the realization of \w[.]{\Diag Y}

For any weakly globular double groupoid $G$, the \ww{E^{2}}-term of the
Bousfield-Friedlander spectral sequence for the bisimplicial set
\w{\Xdd=\cN^{h}\cN^{v}G} survives to \w[.]{E^{\infty}}
Moreover, because \w{G_{0}} is homotopically trivial,
\w[,]{E^{2}_{1,0}=\pi_{1}\pi_{0}\Xdd=0} so in fact by Lemma \ref{lpikan}
$$
\pi_{i}(\Diag\Xdd,x_{0})~=~\begin{cases}
E^{2}_{0,0}=\pi_{0}\pi_{0}(\Xdd,x_{0}) & \text{if}~i=0\\
E^{2}_{0,1}=\pi_{0}\pi_{1}(\Xdd,x_{0}) & \text{if}~i=1\\
E^{2}_{1,1}=\pi_{1}\pi_{1}(\Xdd,x_{0}) & \text{if}~i=2~,
\end{cases}
$$
\noindent for each choice of a base-point \w{x_{0}} in \w[.]{G_{00}}
Actually, \w{\pi_{1}\pi_{1}(\Xdd,x_{0})} is just the automorphism
group of \w[,]{G_{1}} i.e., \w{\Wlo{2,2}G(x_{0},x_{0})}

Therefore, given a weakly globular $n$-fold groupoid $G$, by what we have
shown above we see that
$$
\pi_{n}(\diN G;x_{0})~\cong~\omega_{n}(G;x_{0})
$$
\noindent for each \w[.]{x_{0}\in G_{0,\dotsc,0}} Moreover, by Lemma
\ref{lnmoeq} we have
$$
\pi_{i}(\diN G,x_{0})~\cong~\pi_{i}(B\bPz{n-k+1}\dotsc\bPz{n}G,~x_{0})
$$
\noindent for all \w[,]{0\leq i\leq n-k} and \w{\bPz{n-k+1}\dotsc\bPz{n}G}
is an \wwb{n-k}weakly globular \wwb{n-k}-fold groupoid, so in particular
\wref{eqpintg} holds for each \w[.]{0\leq k\leq n}
\end{proof}

Observe that Theorem \ref{thdnt} provides an intrinsic algebraic definition
of the notion of geometric weak equivalences among weakly globular $n$-fold
groupoids, since we have:

\begin{corollary}\label{chtpygp}
(a) A map of weakly globular $n$-fold groupoids is a geometric
weak equivalence (\S \ref{dwecs}) if and only if it is an
algebraic weak equivalence (\S \ref{sahgp}).

\noindent (b) \ The notion of a weakly globular $n$-fold groupoid $G$ is
purely algebraic.
\end{corollary}
\begin{remark}\label{rhtpygp}
It follows from above that the functor \w{\bPz{n}\col \Gpt{n}\to \Gpt{n-1}}
preserves geometric weak equivalences and serves as an algebraic
\wwb{n-1}Postnikov section functor.
\end{remark}
%
\vspace{-5mm}

\supsect{\protect{\ref{cnt}}.B}{An iterative description of \w{\Qlo{n}}}

We now use the notions of the previous section to provide a more
transparent iterative description of the fundamental $n$-fold groupoid
functor \w{\Qlo{n}X} (Definition \ref{dladjoint}) for a Kan complex $X$.

\begin{defn}\label{dllo}
For any simplicial set $X$, let
\begin{myeq}\label{eqllo}
\Llo{k}X~:=~\begin{cases}\Dex & \text{if}\hs k=0\\
\pro{\Dex}{X}{k+1} & \text{if}\hs k\geq 1~.
\end{cases}
\end{myeq}
\end{defn}

\begin{remark}\label{rkan}
If $X$ is a Kan complex, we have a natural fibration of simplicial
sets \w{u\col \Dex\to X} (cf.\  \S \ref{snssn}), yielding the
internal groupoid \w{(\Dex)^{u}\in\Gpd\sS} of \S \ref{dgc}. We see that
\begin{myeq}\label{eqllok}
(\cN(\Dex)^{u})_{k}~=~\Llo{k}X~=~\pro{\Llo{1}X}{\Dex}{k}
\end{myeq}
\noindent for all \w[,]{k\geq 1} so we may denote the bisimplicial set
\w{\cN(\Dex)^{u}} by \w[.]{\Ld X}
This is depicted in Figure \ref{fldot}, where the vertical maps are induced
by those indicated in the rightmost column, and the horizontal maps are
structure maps for the pullbacks, as in \wref[.]{eqintgpd}
%
%
\begin{figure}[ht]
\begin{center}
$$
\entrymodifiers={++++[]}
\xymatrix@R=15pt@C=15pt{
\cdots~~~
X_{3}\times_{X_{2}}~X_{3}\times_{X_{2}}~X_{3}\;
\ar@<2ex>[rr] \ar[rr]\ar@<-2ex>[rr]
     \ar@<2.5ex>[d] \ar[d]\ar@<-2.5ex>[d] &&
X_{3}\times_{X_{2}}~X_{3} \ar@<1ex>[rr]^{p\sb{2}} \ar@<-1ex>[rr]^{p\sb{1}}
     \ar@<2.5ex>[d] \ar[d]\ar@<-2.5ex>[d] &&
X_{3}  \ar@<2.5ex>[d]^{d_{2}} \ar[d]^{d_{1}}\ar@<-2.5ex>[d]^{d_{0}} \\
\cdots~~~
X_{2}\times_{X_{1}}~X_{2}\times_{X_{1}}~X_{2}\;
\ar@<2ex>[rr] \ar[rr] \ar@<-2ex>[rr]
     \ar@<0.5ex>[d] \ar@<-0.5ex>[d] &&
X_{2}\times_{X_{1}}~X_{2} \ar@<1ex>[rr]^{p\sb{2}} \ar@<-1ex>[rr]^{p\sb{1}}
     \ar@<0.5ex>[d] \ar@<-0.5ex>[d] &&
X\sb{2} \ar@<0.5ex>[d]^{d_{1}} \ar@<-0.5ex>[d]_{d_{0}} \\
\cdots~~~
X_{1}\times_{X_{0}}~X_{1}\times_{X_{0}}~X_{1}\;
\ar@<2ex>[rr] \ar[rr] \ar@<-2ex>[rr]
    \ar@{}[u] &&
X_{1}\times_{X_{0}}~X_{1} \ar@<1ex>[rr]^{p\sb{2}} \ar@<-1ex>[rr]^{p\sb{1}} && X_{1}
}
$$
\end{center}
\caption{Corner of \w{\Ld X}}
\label{fldot}
\end{figure}

If $X$ is reduced, \w{\Dex} is contractible, so
\w{\Llo{1}X} models the loop space \w[.]{\Omega X} In general,
\w{\Llo{1}X} is homotopy equivalent to the ``path object''
\w{\bP  X} of \cite[\S 2.2]{DuskSM}.
\end{remark}

\begin{lemma}\label{lnerve}
Let $X$ be a Kan complex, and \w{\cons X} the corresponding bisimplicial set,
constant in the horizontal direction\vsm .

(a) There is a natural map of bisimplicial sets \w[,]{\phi\col \cons X\to\Ld X}
which is a dimensionwise weak equivalence (as horizontal simplicial sets,
in each vertical dimension \wh see Figure \ref{fldot}), so induces a
weak equivalence \w[\vsm.]{\Diag\phi\col X\to\Diag\Ld X}

(b)~We have \w{\Nup{n}\Qlo{n}X=\oQl{n}{n-1}\Ld X} \wwh i.e.,
for each \w[:]{k\geq 0}
\begin{myeq}\label{eqnervezero}
(\Nup{n}\Qlo{n}X)_{k}~=~\Qlo{n-1}\Llo{k}X~.
\end{myeq}
\noindent Thus for each \w[:]{k\geq 1}
\begin{myeq}\label{eqnerveqk}
\Qlo{n-1}\Llo{k}X~\cong~\pro{\Qlo{n-1}\Llo{1}X}{\Qlo{n-1} \Dex}{k}~.
\end{myeq}

(c)~If $X$ is homotopically trivial, then for \w[;]{k\geq 1}
\begin{myeq}\label{eqprodec}
\Qlo{n}\Llo{k}X\cong \pro{\Qlo{n} \Dex}{\Qlo{n} X}{k+1}~.
\end{myeq}
\end{lemma}

\begin{proof}
\noindent (a) The section \w{\sigma\col X\to\Dex} to the augmentation
\w[,]{\var=\ovl{d}_{\ast}\col \Dex\to X} given in dimension $i$
by the degeneracy \w{s_{i}\col X\sb{i}\to X\sb{i+1}} (cf.\ \S \ref{sdec}),
fits into a diagram of vertical arrows in \w[:]{\sS}
\mydiagram[\label{eqsectionphi}]{
X\ar[d]^{=}\ar[rr]^{\sigma} && \Dex\ar[d]^{\var}\ar[rr]^{\var} && X\ar[d]^{=}\\
X\ar[rr]^{=} && X\ar[rr]^{=} && X
}
\noindent (where the horizontal composite is the identity). Applying the
construction of \S \ref{dgc} to each vertical arrow we obtain
\w[.]{\cons X\xrw{\phi}\Ld X\xrw{\ovl{\var}}\cons X} Here the map
of simplicial sets \w{\phi_{i}\col \cons\,(X_{i})\to(\Ld X)_{i}} in each
internal simplicial dimension $i$ is given by the vertical maps in:
\mydiagram[\label{eqsectwe}]{
\dotsc & X_{i}\ar[d]^{(s_{i},s_{i},s_{i})}\ar@<2ex>[r]^{=}\ar[r]^{=}\ar@<-2ex>[r]^{=} &
X_{i}\ar[d]^{(s_{i},s_{i})}\ar@<1ex>[r]^{=}\ar@<-1ex>[r]^{=} & X_{i}\ar[d]^{s_{i}}\\
\dotsc & X_{i+1}\times_{X_{i}}~X_{i+1}\times_{X_{i}}~X_{i+1}
\ar@<1ex>[r]\ar[r] \ar@<-1ex>[r] &
X_{i+1}\times_{X_{i}}~X_{i+1}\ar@<1ex>[r]\ar@<-1ex>[r] & X_{i+1}~.
}

Since the lower row in \wref{eqsectwe} is the nerve of a
homotopically discrete groupoid, the vertical map is a weak equivalence
(with inverse induced by the right square in \wref[\vsm).]{eqsectionphi}

\noindent (b)~~We will show that for \w[:]{n\geq 2}
\begin{myeq}\label{eqprodec1}
\Nup{n}\Qlo{n}X~=~\oQl{2}{n-1}\Nup{2}\hpu{2}\osl{2}X~,
\end{myeq}
\noindent where \w{\oQl{2}{n-1}} is obtained by applying \w{\Qlo{n-1}}
in each simplicial dimension in the second direction
to the bisimplicial object \w[.]{\Nup{2}\hpu{2}\osl{2}}

By Lemma \ref{lnuph}, we have
\begin{myeq}\label{eqnuphh}
\ovt{2}{n-1}\Nup{2}\hpu{2} \osl{2}X~=~\Nup{n}\hpu{n}\osl{n}X~,
\end{myeq}
\noindent and since by definition of \w[:]{\Qlo{n-1}}
$$
\oQl{2}{n-1}\Nup{2}\hpu{2}\osl{2}X~=~
\hpu{1}\dotsc\hpu{n-1}\ovt{2}{n-1}\Nup{2}\hpu{2}\osl{2}X
$$
\noindent we deduce that:
\begin{myeq}\label{eqorhpi}
\oQl{2}{n-1}\Nup{2}\hpu{2}\osl{2}X~=~
\hpu{1}\dotsc\hpu{n-1}\Nup{n}\hpu{n}\osl{n}X~.
\end{myeq}

Since \w{\Qlo{n}X:=\hpu{1}\dotsc\hpu{n}\osl{n}X} and \w{\osl{n}X} is
\wwb{n,2}fibrant, in order to show \wref{eqprodec1} it suffices to
show by induction on \w{n\geq 2} that
\begin{myeq}\label{eqorhhh}
\Nup{n}\hpu{1}\dotsc\hpu{n}Y~=~\hpu{1}\dotsc\hpu{n-1}\Nup{n}\hpu{n}Y
\end{myeq}
\noindent for any \wwb{n,2}fibrant $n$-fold simplicial set $Y$.
For \w[,]{n=2}
$$
\Nup{2}\hpu{1}\hpu{2}Y=\hpu{1}\Nup{2}\hpu{2}Y
$$
\noindent by Lemma \ref{ltpi} and Proposition \ref{ptfibrant}.

In the induction step, let \w{\Gd} be the simplicial \wwb{n-1}fold groupoid
\w[.]{\hpu{2}\dotsc\hpu{n}Y} By Lemma \ref{lfibrant}, \w{\Gd} is
\wwb{n-1,2}fibrant, so for each \w[,]{\va\in\bDelta^{n-2}}
the simplicial groupoid \w{\Gd(\va)} (in the first groupoid
direction of \w[)]{\Gd} is \wwb{2,2}fibrant. Thus by Lemma \ref{ltpi}
we have \w[,]{\Nup{n}\hpu{1}\Gd(\va)=\hpu{1}\Nup{n}\Gd(\va)} so
$$
\Nup{n}\hpu{1}\dotsc\hpu{n}Y=\Nup{n}\hpu{1}\Gd=\hpu{1}\Nup{n}\Gd=
\hpu{1}\Nup{n}\hpu{2}\dotsc\hpu{n}Y.
$$

If we think of $Y$ as a simplicial \wwb{n-1,2}fibrant \wwb{n-1}fold
simplicial set \w{Y\up{1}_{\bullet}} (in the first direction),
by the induction hypotheses
$$
\Nup{n}\hpu{2}\dotsc\hpu{n}Y\up{1}_{m}~=~
\hpu{2}\dotsc\hpu{n-1}\Nup{n}\hpu{n}Y\up{1}_{m}
$$
\noindent for each \w[,]{m\geq 0} so \wref{eqorhhh} holds for $Y$, too.
This concludes the proof of \wref[\vsm .]{eqprodec1}

Observe that
\begin{myeq}\label{eqaf}
\Nup{2}\hpu{2}\osl{2} X~=~\cN A^{f}
\end{myeq}
\noindent for \w{A^{u}\in\Gpd(\sS)} as in \wref[,]{eqintgpd} where
$u$ is the map of simplicial sets \w[.]{u\col \Dex\to X}
In fact, \w[,]{\hpu{2}\osl{2} X} thought of as a simplicial object in
\w[,]{\Gpd} has \w{(\hpu{2}\osl{2} X)_{k}=\hpi\Dec^{k} X} in
simplicial dimension $k$. This is isomorphic to the homotopically
discrete groupoid \w{(X_{k})^{u_{k}}} (where
\w{u_{k}\col X_{k}\to X_{k-1}} is a map of sets). Hence from
\wref{eqprodec1} and \wref{eqaf} we conclude that
$$
\Nup{n}\Qlo{n}X~=~\ovl{Q}\lo{n-1}\cN A^{u}~.
$$
\noindent Since \w{(\cN A^{u})_{k}=\Llo{k}X} for each \w[,]{k\geq 0}
\wref{eqnervezero} follows.

In particular, since \w[,]{\Qlo{n} X\in\Gpt{n}} we have, for \w[,]{k\geq 2}
$$
\Qlo{n-1}\Llo{k}X\!=\!(\Nup{n}\Qlo{n} X)\up{n}_{k}\!\cong\!
    \pros{(\Nup{n}\Qlo{n} X)\up{n}_{1}}{(\Nup{n}\Qlo{n}
      X)\up{n}_{0}}{k}
$$
\noindent so by \wref{eqnervezero} we have:
$$
\Qlo{n-1}\Llo{k}X~\cong~\pro{(\Qlo{n-1}\Llo{1}X)}{(\Qlo{n-1}\Dex)}{k}~\vsm.
$$

\noindent (c)~~By induction on $n$. For \w[,]{n=1} \w[.]{\Qlo{1}=\hpi}
Since by hypothesis $X$ is homotopically trivial and \w{u\col \Dex\to X} is
a fibration, \w{\Llo{1}X=\tens{\Dex}{X}} is also homotopically trivial; hence
\w{\hpi\Llo{1}X} is a homotopically discrete groupoid, and is therefore
isomorphic to \w{A^{f}} where \w{f\col A\to B} is the obvious map
$$
\tens{X_{1}}{X_{0}}~\lto~\tens{X_{0}}{\pi_{0} X}~.
$$
\noindent On the other hand, \w{\hpi\Dex\cong(X_{1})^{d_{0}}} and
\w{\hpi X=(X_{0})^{\gamma}} (for \w[),]{\gamma\col X_{0}\to\pi_{0}X} so:
\begin{equation*}
\hpi\Llo{1}X~\cong~\tens{\hpi\Dex}{\hpi X}~.
\end{equation*}

In the induction step, applying \w{\Nup{n}} to both sides of
\wref[,]{eqprodec} we must show that for each \w{k\geq 1} and \w{i\geq 1}
we have
$$
(\Nup{n}\Qlo{n}\Llo{k}X)\up{n}_{i-1}\!\cong\!
\pros{(\Nup{n}\Qlo{n}\Dex)\up{n}_{i-1}}{(\Nup{n}\Qlo{n}X)\up{n}_{i-1}}{k-1}
$$
or equivalently (after applying (b)), that:
\begin{myeq}\label{eqprodecnew}
\Qlo{n-1}\Llo{i}(\Llo{k}X)~\cong~
\pros{\Qlo{n-1}\Llo{i}\Dex}{\Qlo{n-1}\Llo{i}X}{k+1}
\end{myeq}

Since $X$ is homotopically trivial, so are \w{\Dex} and \w{\Llo{k}X} (since
\w{u\col \Dex\to X} is a fibration), so we can apply induction hypothesis
(c) for \wb{n-1} to replace the left hand side of \wref{eqprodecnew} by:
$$
\pro{\Qlo{n-1}\Dec(\Llo{k}X)}{\Qlo{n-1}\Llo{k}X}{i+1}~,
$$
\noindent and since \w{\Dec} commutes with fiber products, and thus
with \w[,]{\Llo{k}} this equals:
{\scriptsize
$$
\pros{(\Qlo{n-1}(\pros{\Dec^{2}\!X}{\Dex}{k+1}))}{(\Qlo{n-1}
(\pros{\Dex}{X}{k+1}))}{i+1}.
$$
}
\noindent If we write \w[,]{A:=\Qlo{n-1}\Dec^{2}X}
\w[,]{B:=\Qlo{n-1}\Dex} and \w[,]{C:=\Qlo{n-1}X} applying (b) for
\w{n-1} to this last expression yields:
\begin{myeq}\label{eqprodecnewleft}
\pro{(\pro{A}{B}{k+1})}{(\pro{B}{C}{k-1}))}{i+1}~.
\end{myeq}

Similarly, (c) applied to the right hand side of \wref{eqprodecnew} yields
\begin{myeq}\label{eqprodecnewright}
\pro{(\pro{A}{B}{i+1}))}{(\pro{B}{C}{i-1}))}{k+1}~,
\end{myeq}
\noindent and the two limits \eqref{eqprodecnewleft} and
\eqref{eqprodecnewright} are evidently equal, proving \eqref{eqprodecnew}.
\end{proof}

\supsect{\protect{\ref{cnt}}.C}{Modelling $\mathbf{n}$-types}

In the last part of this section we finally show that weakly globular
$n$-fold groupoids indeed model $n$-types.

\begin{prop}\label{pnequiv}
Let $X$ be a Kan complex. Then:
\begin{enumerate}
\renewcommand{\labelenumi}{(\alph{enumi})~}
\item There is a natural $n$-equivalence \w[.]{\psi\sp{X}\lo{n}\col X\to\dN\Qlo{n}X}
\item \www{\Qlo{n}} preserves weak equivalences of Kan complexes.
\item If $X$ is homotopically trivial (i.e., all higher homotopy
  groups vanish), then \w{\Qlo{n}X} is a homotopically trivial $n$-fold groupoid.
\item \www{\Qlo{n} X} is a weakly globular $n$-fold groupoid, and
\w{\bPz{n}\Qlo{n} X} is isomorphic to \w[.]{\Qlo{n-1} X}
\end{enumerate}
\end{prop}

\begin{proof}
By induction on $n$. The claim is immediate for \w{n=1}
(with \w{\Qlo{0}X:=\pi_{0}X} and
\w{\psi\sp{X}\lo{1}\col X\to\cN\hpi X\simeq \Po{1}X} the Postnikov structure map)\vsm.

\noindent (a)~ We assume that we have a map
$$
\psi\sp{X}\lo{n-1}\col X~\to~\Dlo{n-1}\Nlo{n-1}\Qlo{n-1}X~,
$$
\noindent natural in $X$.
Applying this to the simplicial object \w{\Ld X\in\sC{\sS}} (which
is fibrant in each simplicial dimension, by \S \ref{rdex}), we
obtain a map of bisimplicial sets
$$
\psi\sp{\Ld X}\lo{n-1}\col \Ld X~\to\oDlo{n-1}\up{n}\oNl{n-1}\oQl{n}{n-1}\Ld X~,
$$
\noindent which is an \wwb{n-1}equivalence in each simplicial dimension.

However, in simplicial dimension $0$ we have \w[,]{\Llo{0}X=\Dex} which is
homotopically trivial, while \w{\Qlo{n-1}\Dex} is a homotopically discrete
\wwb{n-1}fold groupoid by induction assumption (c) for \w[,]{n-1} so
\w{\dN\Qlo{n-1}\Dex} is homotopically trivial by
Corollary \ref{chdng}. Thus \w{\psi\sp{\Llo{0}X}\lo{n-1}} is actually a
geometric weak equivalence, and thus \w{\psi\sp{\Ld X}\lo{n-1}} is a diagonal
$n$-equivalence (cf.\ \S \ref{dndiag}), which implies that

\begin{myeq}\label{eqndiageq}
\Diag\psi\sp{\Ld X}\lo{n-1}~\col ~\Diag\Ld X~\to~
\Dlo{n}\oNl{n-1}\oQl{n}{n-1}\Ld X
\end{myeq}
\noindent is an $n$-equivalence by Proposition \ref{pndiag}.

Now by \wref[,]{eqnervezero} \w[,]{\Nup{n}\Qlo{n}X=\oQl{n}{n-1}\Ld X} so together
with the map \w{\phi\col \cons X\to\Ld X} of Lemma \ref{lnerve}(a) we have maps of
bisimplicial sets:
$$
\cons X~\xrw{\phi}~\Ld X~\xrw{\psi\sp{\Ld X}\lo{n-1}}~
\oDlo{n-1}\up{n}\oNl{n-1}\oQl{n}{n-1}\Ld X~=~
\oDlo{n-1}\up{n} N\up{n}\Qlo{n}X~.
$$

Applying \w{\Diag} to both maps we see that the first is a weak
equivalence, while the second is an $n$-equivalence, because
\wref{eqndiageq} is such. We define the composite to be
\w[,]{\psi\sp{X}\lo{n}\col X\to\dN\Qlo{n}X} which is therefore an
$n$-equivalence\vsm.

\noindent (b)~ Let \w{f\col X\to Y} be a weak equivalence of Kan
complexes. Since by (a), \w{X\to\Dlo{n}\Qlo{n}X} and \w{Y\to\Dlo{n}\Qlo{n}Y}
are $n$-equivalences, it follows that \w{\Dlo{n}\Qlo{n} f} is an
$n$-equivalence. By Theorem \ref{thdnt}, \w{\Dlo{n}\Qlo{n}X} and
\w{\Dlo{n}\Qlo{n}Y} are $n$-types. Hence \w{\Dlo{n}\Qlo{n}f} is a weak
equivalence\vsm .

\noindent(c)~ Since $X$ is homotopically trivial, by Lemma
\ref{lnerve} for each \w{k\geq 1} we have:
$$
(\Nup{n}\Qlo{n} X)_{k}\!=\!\Qlo{n-1}\Llo{k}X=
\pro{\Qlo{n-1}\Dex}{\Qlo{n-1} X}{k+1}.
$$
\noindent Therefore \w[,]{\Qlo{n} X=A^{f}} where \w{A=\Qlo{n-1}\Dex}
and by induction
$$
f:=\Qlo{n-1}\var\col \Qlo{n-1}\Dex\to\Qlo{n-1} X
$$
\noindent is a map of homotopically discrete \wwb{n-1}fold
groupoids with a section \w{\Qlo{n-1}\sigma} (cf.\ \S \ref{sdec}).
Hence, \w{\Qlo{n}X} is homotopically discrete, by definition\vsm .

\noindent (d)~To show that \w{\Qlo{n} X} is weakly globular
(Definition \ref{dntng}), we think of it as a groupoid in \w[,]{\Gpd^{n-1}}
with \wwb{n-1}fold groupoid of objects \w{(\Qlo{n} X)_{0}} and
\wwb{n-1}fold groupoid of arrows \w[.]{(\Qlo{n} X)_{1}}
Note that \w{(\Qlo{n} X)_{0}=\Qlo{n-1}\Dex} by \wref{eqnervezero}
for \w[,]{k=0} and since \w{\Dex} is
homotopically discrete, \w{(\Qlo{n} X)\up{n}_{0}} is homotopically
discrete, by (c).

Similarly, \w[,]{(\Nup{n}\Qlo{n}X)\up{n}_{1}=\Qlo{n-1}\Llo{1}X\in\Gpt{n-1}}
and by \wref[:]{eqnerveqk}
\begin{equation*}
\begin{split}
(\Nup{n}\Qlo{n}X)_{k}~=&~\pro{(\Qlo{n}X)_{1}}{(\Qlo{n}X)_{0}}{k}\\
~=&~\Qlo{n-1}(\pro{\Llo{1}X}{\Dex}{k})~,
\end{split}
\end{equation*}
\noindent so \w{(\Nup{n}\Qlo{n} X)\sb{k}} is weakly globular for each
\w[.]{k\geq 0}

If we apply \w{\bPz{n-1}} in each simplicial dimension in the $n$-th direction,
by Lemma \ref{lnerve}(b) and the induction hypothesis:
$$
(\oPz\up{n-1}\Nup{n}\Qlo{n}X)\up{n}_{k}~=~\bPz{n-1}\Qlo{n-1}\Llo{k}X~=~
\Qlo{n-2}\Llo{k}X~=~(\Qlo{n-1}X)_{k}~.
$$
\noindent where \w{(\Nup{n-1}\Qlo{n-1} X)\up{n-1}_{k}} is abbreviated
to \w[.]{(\Qlo{n-1} X)_{k}}

This shows that \w{\bPz{n}G} lands in weakly globular
\wwb{n-1}fold groupoids, and that:
$$
\bPz{n}\Qlo{n} X\cong\Qlo{n-1}X~\vsm.
$$

To prove that \w[,]{\Qlo{n} X\in\Gpt{n}} it remains to show that
in each simplicial dimension \w{k\geq 2} (in the $n$-th direction),
the map:
\begin{myeq}\label{eqproq}
\begin{split}
&\pro{(\Qlo{n} X)_{1}}{(\Qlo{n} X)_{0}}{k}~\to\\
&\hspace*{20mm}\pro{(\Qlo{n} X)_{1}}{(\Qlo{n}X)_{0}^{d}}{k}
\end{split}
\end{myeq}
\noindent is a geometric weak equivalence. By Lemma \ref{lnerve}(b)
we have:
\begin{equation*}
\begin{split}
\pro{(\Qlo{n}X)_{1}}{(\Qlo{n}X)_{0}}{k}~=~(\Nup{n}\Qlo{n}X)_{k}&\\
=~\Qlo{n-1}\Llo{k}X~\cong~
\Qlo{n-1}(\pro{\Llo{1}X}{\Dex}{k})&
\end{split}
\end{equation*}
\noindent where the second equality is \wref{eqnervezero} and
the third is \wref[.]{eqnerveqk}

Since \w{\Dex} is homotopically trivial, \w{\Qlo{n-1}\Dex} is homotopically
discrete by (c), so
$$
(\Qlo{n}X)\sb{0}\sp{d}~=~(\Qlo{n-1}\Dex)\sp{d}~=~
\Qlo{n-1}c(\pi_{0}\Dex)~=~\Qlo{n-1}c(X\sb{0})
$$
\noindent by (a) and Lemma \ref{lhdng}(d), where \w{c(X_{0})} is
the constant simplicial set on \w[.]{X_{0}}

Since $X$ is a Kan complex and \w{\Qlo{n-1}} commutes with fiber
products over discrete objects, by Remark \ref{reprod}, we have:
\begin{equation*}
\begin{split}
\pros{(\Qlo{n} X)_{1}}{(\Qlo{n} X)_{0}^{d}}{k}~=&
~\pros{\Qlo{n-1}\Llo{1}X}{\Qlo{n-1}c(X_{0})}{k}\\
~\cong&~\Qlo{n-1}(\pro{\Llo{1}X}{c(X_{0})}{k})~.
\end{split}
\end{equation*}
\noindent Since \w{\Dex\to X} is a fibration, so is \w[,]{\Llo{1}X\to\Dex}
and \w{\Dex\to c(X_{0})} is a weak equivalence; thus the map
$$
\pro{\Llo{1}X}{\Dex}{k}\to\pro{\Llo{1}X}{c(X_{0})}{k}
$$
\noindent is a weak equivalence of Kan complexes. Therefore, by (b),
\wref{eqproq} is a weak equivalence, as required.
\end{proof}

Recall from \S \ref{dpostnikov} that \w{\PT{n}} denotes the full
subcategory of \w{\Top} consisting of spaces $T$ for which the
natural map \w{T\to \Po{n}T} is a weak equivalence, and \w{\ho\PT{n}}
is the corresponding full subcategory of the homotopy category \w{\ho\Top} of
topological spaces.

\smallskip
\begin{defn}\label{dsimb}
Let \w{\ho\Gpt{n}} denote the localization of the category \w{\Gpt{n}}
with respect to the (algebraic) weak equivalences (see
Corollary \ref{chtpygp} and compare \cite{GZiC}).
\end{defn}

\smallskip
\begin{thm}\label{teqcat}
The functors \w{\hQlo{n}\col \Top\to\Gpt{n}} and
\w{\diN\col \Gpt{n}\to\Top} induce functors:
\begin{myeq}\label{eqfaith}
\ho\PT{n}~\adj{\diN}{\hQlo{n}}~\ho\Gpt{n}~,
\end{myeq}
\noindent with \w[,]{\diN\circ\hQlo{n}\cong\Id\sb{\ho\PT{n}}} so
\w{\hQlo{n}\col \ho\PT{n}\to \ho\Gpt{n}} is a faithful embedding.
\end{thm}

\begin{proof}
By Theorem \ref{thdnt} and Proposition \ref{pnequiv} both functors
\w{\hQ_n = \Qlo{n} S} and $\diN$ preserve weak equivalences, and therefore
induce corresponding functors on the homotopy categories. Also, for any
\w[,]{T\in\PT{n}} by Theorem \ref{thdnt} and Proposition \ref{pnequiv},
there is a span
\begin{myeq}\label{eqzigzagh}
B\hQlo{n}T \leftarrow |ST| \to T~,
\end{myeq}
\noindent where the map on the left is a homotopy equivalence and the map on
the right is a weak homotopy equivalence. It follows that $T$ and
\w{\diN\hQlo{n}T} are weakly equivalent in \w[;]{\PT{n}} that is,
\w[.]{\diN\circ\hQlo{n}\cong\Id\sb{\ho\PT{n}}}
\end{proof}

\smallskip
\begin{mysubsection}{Weakly globular double groupoids}
\label{swgdg}
For \w[,]{n=2} we can strengthen Theorem \ref{teqcat} to obtain an equivalence
\w[,]{\ho\PT{2}\approx \ho\Gpt{2}} where on the right hand side we use the
(internally defined) algebraic weak equivalences of \w{\Gpt{2}} itself:

As in \cite[Theorem 2.5]{BCDusC}, for any double groupoid $G$ one can
construct a map \w[.]{\var\sb{\bullet}\col \osl{2}\dN G\to\Nlo{2}G}
By \cite[Theorem 8]{CHRemeD}, if $G$ is weakly globular (and
therefore \wwb{2,2}fibrant), \w{\dN G} is a Kan complex. Therefore,
\w{\Plo{2}\osl{2}\dN G} and \w{\Plo{2}\Nlo{2}G=G} are weakly globular double
groupoids. Since we have a homotopy equivalence of Kan complexes
\w[,]{\xi\col \dN G\to\cS\|\dN G\|=\cS\diN G} we also have a geometric
weak equivalence of weakly globular double groupoids:
\begin{myeq}\label{eqhtpyeq}
\Qlo{2}\dN G~\xrw{\Qlo{2}\xi}~\Qlo{2}\cS\diN G=\hQlo{2}\diN G~.
\end{myeq}
\noindent Therefore, the algebraic homotopy groups \w{\omega_{\ast}\Qlo{2}\dN G}
are isomorphic by \wref{eqpinag} to
$$
\pi_{\ast}\diN\Qlo{2}\dN G~\cong~\pi_{\ast}\diN\hQlo{2}\diN G~\cong~\pi_{\ast}\diN G
$$
\noindent (using \wref{eqzigzagh} for \w[).]{T:=\diN G} By Theorem \ref{thdnt},
\w[,]{\omega_{\ast}G\cong\pi_{\ast}\diN G} and since
\w[,]{\omega_{\ast}G\cong\pi_{\ast}\dN G} also by \wref[,]{eqpinag} we
conclude that \w[.]{\omega_{\ast}\Qlo{2}\dN G\cong\omega_{\ast}G}
One can verify that this isomorphism is induced by the map
$$
\Plo{2}\var\sb{\bullet}\col \Plo{2}\osl{2}\dN G=\Qlo{2}\dN G\to\Plo{2}\Nlo{2}G=G~,
$$
which is therefore a geometric weak equivalence of double groupoids.

Together with a map of double groupoids induced by \wref[,]{eqhtpyeq} we
obtain a zig-zag of geometric weak equivalences:
$$
\hQlo{2}\diN G~\xlw{\Qlo{2}\xi}~\Qlo{2}\dN G
~\xrw{\Plo{2}\var\sb{\bullet}}~G~.
$$
\noindent This implies that \wref{eqfaith} is an equivalence of localized
categories when \w[.]{n=2}
\end{mysubsection}

%
%

\medskip
\sect{Tamsamani's model and weakly globular $n$-fold groupoids}
\label{ctmng}

In this section we construct a comparison functor from weakly globular
$n$-fold groupoids to Tamsamani's weak $n$-groupoids, which preserves
homotopy types.

\supsect{\protect{\ref{ctmng}}.A}{Tamsamani's weak $\mathbf{n}$-groupoids}

We begin with a brief recapitulation of the notion of a Tamsamani
weak $n$-groupoid, starting with a modified definition.
This differs somewhat from the original definition in \cite[\S 5]{TamsN}
(compare \cite[\S 15.2]{SimpH} and \cite[\S 8]{PaolW}),
which was motivated by the goal of modelling higher \emph{categories},
rather than groupoids.

\begin{defn}\label{dtamn}
The category \w{\Tam{n}} of \emph{Tamsamani weak $n$-groupoids} is a full
subcategory of \w[,]{\sCx{n-1}{\Gpd}} defined inductively as follows:

\begin{enumerate}
\renewcommand{\labelenumi}{(\alph{enumi})~}
\item \www{\Tam{1}:=\Gpd} is the category of groupoids.
\item Each \w{X\in\Tam{n}} is a simplicial object in
\w{\Tam{n-1}} (in the first simplicial direction). We therefore have an
inclusion functor \w[.]{J\sb{n}\col \Tam{n}\to \sCx{n-1}{\Gpd}}
\item The $0$-th Tamsamani weak \wwb{n-1}groupoid \w{X_{0}} is discrete
(that is, a constant \wwb{n-1}fold simplicial set).
\item The Segal maps \w{\seg{k}\col X_{k}\to\pro{X_{1}}{X_{0}}{k}}
(Definition \ref{dsegal}) are geometric weak equivalences of
Tamsamani weak \wwb{n-1}groupoids for each \w[:]{k\geq 2}
that is, \w{\TamR\seg{k}\col \TamR X_{k}\to\TamR(\pro{X_{1}}{X_{0}}{k})}
is a weak equivalence of topological spaces, where
\w{\TamR\col \Tam{n}\to\Top} is the realization functor of \wref[.]{eqrealize}
\item The \wwb{n-1}simplicial set \w{\opz{n}J\sb{n} X} is the nerve of a
Tamsamani weak \w{(n-1)} groupoid \w[,]{\bPz{n}X} and we have a
commutative diagram
$$
\xymatrix{
\Tam{n} \ar^{J\sb{n}}[rrrr] \ar_{\bPz{n}}[d] &&&& \sCx{n-1}{\Gpd} \ar^{\opz{n}}[d] \\
\Tam{n-1} \ar_{J\sb{n-1}}[rr] && \sCx{n-2}{\Gpd} \ar_{\ovl{N}}[rr] && \sCx{n-1}{\Set}
}
$$
Furthermore,  \w{\bPz{n}} preserves geometric weak equivalences.
\end{enumerate}
\end{defn}

\begin{mysubsection}{Tamsamani's original definition}\label{sbptam}
Tamsamani's original approach (as described in \cite[\S 8]{PaolW}) gave
an inductive definition of the category \w{\tTam{n}\subseteq\sCx{n-1}{\Gpd}}
equipped with a class of maps called \emph{$n$-equivalences} for each
\w[,]{n\geq 1}  The following assumptions must be satisfied\vsm:

\begin{itemize}
\item[(a)] \www{\tTam{1}:=\Gpd} (with $1$-equivalences
being equivalences of groupoids).
\item[(b)] Each \w{X\in\tTam{n}} is a simplicial object in
\w[.]{\tTam{n-1}}
\item[(c)] \www{X_{0}} is discrete.
\item[(d)] The Segal maps \w{\seg{k}\col X_{k}\to\pro{X_{1}}{X_{0}}{k}}
are \wwb{n-1}equivalences in \w{\tTam{n-1}} for each \w[.]{k\geq 2}
\item[(e)] The functor \w[,]{\opz{1}\opz{2}\dotsc\opz{n}\col \tTam{n}\to\Set}
(cf.\ \S \ref{dbpz}) takes $n$-equivalences to bijections and preserves
fiber products over discrete objects.
\end{itemize}

Note that (d) and (e) together imply that the Tamsamani
\emph{fundamental groupoid} functor
$$
\TTm{n}~:=~\opz{2}\cdots\opz{n-1}\opz{n}~,
$$
\noindent when applied to \w[,]{X\in\tTam{n}} lands in groupoids\vsm.

\begin{itemize}
\item[(f)] For every $a$ and $b$ in the set \w[,]{X_{0}} the fiber of
\w{X\lo{a,b}} of
\w{(d_{0},d_{1})\col X_{1}\to X\sb{0}\times X\sb{0}}
is a Tamsamani weak \wwb{n-1}groupoid.
\item[(g)] A map \w{f\col X\to Y} in \w{\tTam{n}}is an $n$-equivalence if
and only if:
\begin{enumerate}
\renewcommand{\labelenumi}{\roman{enumi}.~}
\item The map \w{\TTm{n}f\col \TTm{n}X\to\TTm{n}Y} is an equivalence
of groupoids.
\item \w{f\lo{a,b}\col X\lo{a,b}\to Y\lo{a,b}} is an \wwb{n-1}equivalence
for every \w[.]{(a,b)\in X\sb{0}\times X\sb{0}}
\end{enumerate}
\end{itemize}
\end{mysubsection}

\begin{remark}\label{rtequivdefns}
Note that if \w{g\col X\to Y} is a morphism in \w{\tTam{n}} with $Y$
discrete, then $X$ is isomorphic to
\w[,]{\coprod\sb{y\in Y}~g^{-1}\{y\}} where
the coproduct is taken in \w{\tTam{n}} (compare
Lemma \ref{lpsgcoprod} below).

This implies that if \w[,]{X\in\tTam{n}}  then
\w{X\sb{1}} is isomorphic to the coproduct over all \w{a,b\in X\sb{0}}
of \w{X\sb{1}(a,b)\in\tTam{n-1}} (where \w{X\sb{1}(a,b)} is the fiber of
\w[).]{(d\sb{0},d\sb{1})\col X\sb{1}\to X\sb{0}\times X\sb{0}}

From this and from (e) we deduce that if \w[,]{X\in\tTam{n}} the
\wwb{n-1}simplicial set \w{\opz{n}J\sb{n}X} is the nerve of an object
\w{\bPz{n}X} of \w{\tTam{n-1}} and we have the commutative diagram
$$
\xymatrix{
\tTam{n} \ar^{J_n}[rrrr] \ar_{\bPz{n}}[d] &&&& \sCx{n-1}{\Gpd} \ar^{\opz{n}}[d] \\
\tTam{n-1} \ar_{J_{n-1}}[rr] && \sCx{n-2}{\Gpd} \ar_{\ovl{N}}[rr] && \sCx{n-1}{\Set}
}
$$
\noindent Furthermore, \w{\bPz{n}} takes $n$-equivalences to
\wwb{n-1}equivalences, and one can therefore replace (g) in the definition
above by\vsm:

\begin{enumerate}
\renewcommand{\labelenumi}{\roman{enumi}.~}
\item  The map \w{\bPz{n}f\col \bPz{n}X\to\bPz{n}Y} is an
\wwb{n-1} equivalence in \w[.]{\tTam{n-1}}
\item \w{f\lo{a,b}\col X\lo{a,b}\to Y\lo{a,b}} is an \wwb{n-1}equivalence
for every \w[.]{(a,b)\in X\sb{0}\times X\sb{0}}
\end{enumerate}
\end{remark}

We recall the following fact from \cite[Lemma 10.1]{PaolW}:

\begin{lemma}\label{lnequiv}
A map \w{f\col X\to Y} in \w{\tTam{n}} is an $n$-equivalence if and only if
it is a geometric weak equivalence.
\end{lemma}

\begin{prop}\label{ptequivdefns}
The categories \w{\Tam{n}} and \w{\tTam{n}} are identical.
\end{prop}

\begin{proof}
By induction on \w[,]{n\geq 1} starting with
\w[.]{\Tam{1}=\Gpd=\tTam{1}} The fact that \w{\tTam{n}} is contained
in \w{\Tam{n}} is immediate (by the induction hypothesis and Lemma
\ref{lnequiv}), while the other direction follows from Remark \ref{rtequivdefns}
and Lemma \ref{lnequiv} again.
\end{proof}

\begin{defn}\label{dsimbtam}
Let \w{\ho\Tam{n}} denote the localization of the category \w{\Tam{n}}
with respect to the $n$-equivalences.
\end{defn}

\begin{thm}[\protect{\cite[Theorem 8.0]{TamsN}}]\label{ttamseq}
There is a \emph{Poincar\'{e} $n$-groupoid} functor \w{\FTm{n}\col \Top\to\Tam{n}}
which, together with \w[,]{\TamR\col \Tam{n}\to\Top} induces equivalences
of categories:
\begin{myeq}\label{eqtamsequiv}
\ho\PT{n}~\adj{\TamR}{\FTm{n}}~\ho\Tam{n}~.
\end{myeq}
\noindent For every \w[,]{T\in\Top} there is a zigzag of weak equivalences in
\w{\PT{n}} between \w{\TamR\FTm{n}T} and \w[,]{\Po{n}T} and for every
\w[,]{X\in\Tam{n}} there is a natural weak equivalence \w[.]{X\to\FTm{n}\TamR X}
\end{thm}

\supsect{\protect{\ref{ctmng}}.B}
{Comparison with weakly globular $\mathbf{n}$-fold groupoids}

We construct iteratively a discretization functor
\w[,]{D\sb{n}\col \Gpt{n}\to\Tam{n}} which preserves the homotopy type.

\begin{mysubsection}{Two simplicial constructions}
\label{ssrcons}
Let $\cC$ be a (co)complete category, \w{X\in\sC{\cC}} a simplicial object,
and \w{\gamma\col X\sb{0}\to W} a map in $\cC$. In this context we mimic the
construction of a new simplicial object \w{Y\in\sC{\cC}} described in
\cite[\S 3]{BPaolT}, as follows:

Consider the pushout in $\cC$
\mydiagrm[\label{eqsimpcons}]{
    X\sb{0} \ar[r]^{s\lo{n}} \ar[d]_{\gamma} & X_{n} \ar[d]^{f_{n}}\\
    W \ar[r]_{\sigma\lo{n}} & Y_{n}
    }
\noindent where \w{s\lo{n}} is induced by the unique morphism
\w{\bze\to\bn} in \w[.]{\Dop} For any morphism \w{\phi\col \bn\to\bm} in
\w[,]{\Dop} \w{\phi s\lo{n}=s\lo{m}} by the uniqueness, so that
$$
f_{m}\phi s\lo{n}=f_{m}s\lo{m}=\sigma_{(m)}f_{0}\col X_{0}\to Y_{m}~.
$$

By the universal property of pushouts there exists a unique
\w{\hphi\col Y_{n}\to Y_{m}} with \w{\hphi f_{n}=f_{m}\hphi} and
\w[.]{\hphi\sigma_{(n)}=\sigma\lo{m}} In particular, we have maps
\w{\hat{d}_{i}\col Y_{n}\to Y_{n-1}} for \w[,]{0\leq i\leq n} and
\w{\hat{\sigma}_{i}\col Y_{n-1}\to Y_{n}} for \w[.]{0\leq i<n} The maps
\w{\hat{d}_{i}} and \w{\hat{\sigma}_{i}} satisfy the simplicial
identities, so that $Y$ is a simplicial object in $\cC$. In fact, if
\w{\bn\rws{\phi}\bm\rws{\psi}\bk} are
morphisms in \w{\Dop} and \w[,]{\xi=\psi\circ\phi} then
$$
\hxi\sigma_{(n)}~=~\sigma_{(k)}~=~\hpsi\sigma_{(m)}~=~
\hpsi\hphi\sigma_{(n)}
$$
and
$$
\hxi f_{n}~=~f_{k}\hxi~=~f_{k}\hpsi\hphi~=~\hpsi f_{m}\hphi~=~
\hpsi\hphi f_{n}~.
$$

It follows by universal property of pushouts that \w[.]{\hxi=\hpsi\hphi} In
particular, since the simplicial identities are satisfied by \w{d_{i}} and
\w[,]{\sigma_{i}} they are satisfied by \w{\hat{d}_{i}} and
\w[.]{\hat{\sigma}_{i}}  So we have a map of simplicial objects \w[.]{f\col X\to Y}

Note that if \w{\gamma'\col W\to X\sb{0}} is a section for $\gamma$
(with \w[),]{\gamma\gamma'=\Id} we may construct a new simplicial object
\w{X\sp{\gamma}\in\sC{\cC}} by setting
$$
X\sp{\gamma}\sb{n}=\begin{cases}
W& \text{if~}n=0\\
X_{n} & \text{if~}n>0~.
\end{cases}
$$
\noindent Let \w{d_{0},d_{1}\col X_{1}\to X_{0}}
and \w{\sigma_{0}=s\lo{1}\col X_{0}\to X_{1}} be the face and degeneracy maps of
$X$, and let \w[,]{d'_{0}} \w[,]{d'_{1}\col X_{1}\to W} and
\w[,]{\sigma'_{0}\col W\to X_{1}} respectively, denote
\w{d'_{i}=\gamma d_{i}} \wb[,]{i=0,1} and
\w[.]{\sigma'_{0}=\sigma_{0}\gamma'} All other face
and degeneracy operators of \w{X\sp{\gamma}} are the same as those of $X$.

Finally, we define a map \w{h\col X\sp{\gamma}\to Y} in \w{\sC{\cC}}
by setting \w{h_{0}:=\Id} and \w{h_{n}:=f_{n}} for
\w[.]{n>0} In fact, \w[;]{d'_{i}=\hat{d}_{i} f_1} also,
\w[,]{f_{1}\sigma_{0}=\hat{\sigma}_{0}\gamma} which implies
\w[.]{f_{1}\sigma_{0}\gamma'=\hat{\sigma}_{0}\gamma\gamma'=\hat{\sigma}_{0}}
All other identities are the same as for $f$.
\end{mysubsection}

\begin{mysubsection}{The functor $D$}
\label{sfuncd}
Let \w{\sS^{2}_{h}} be the full subcategory of bisimplicial sets
$X$ such that the simplicial set \w{X_{0}} is homotopically trivial,
through a weak equivalence \w{\gamma\col X_{0}\to X_{0}^{d}} with a
section \w{\gamma'\col X_{0}^{d}\to X_{0}} with \w[,]{\gamma\gamma'=\Id}
where \w{X_{0}^{d}} is the constant simplicial set on
\w[.]{\pi_{0}X_{0}} Let \w{\sS^{2}_{d}} denote the full subcategory of
bisimplicial sets $X$ such that the simplicial set \w{X_{0}} is
constant. We construct a functor
$$
D~\col ~\sS^{2}_{h}~\to~\sS^{2}_{d}
$$
\noindent by setting \w{DX:=X\sp{\gamma}} (in the notation of \S \ref{ssrcons}).
\end{mysubsection}

\begin{lemma}\label{lhtype}
Let \w{D\col \sS^{2}_{h}\to\sS^{2}_{d}} be as above. Then for each
\w[,]{X\in\sS^{2}_{h}} \w{DX} and $X$ have the same homotopy type.
\end{lemma}

\begin{proof}
We construct a bisimplicial set $Y$ and weak equivalences
$$
X~\rws{f}~Y~\lws{h}~DX
$$
\noindent using the construction of \S \ref{ssrcons}, for \w[,]{\cC=\sS}
\w{W:=X\sb{0}\sp{d}} and \w{\gamma\col X\sb{0}\to X\sb{0}\sp{d}} as above.
Since $\gamma$ is a weak equivalence and \w{s\lo{n}} is a cofibration
of simplicial sets, the right vertical map \w{f\sb{n}} in \wref{eqsimpcons} is
a weak equivalence for each \w{n\geq 0} \wwh that is, we have a map of
bisimplicial sets \w{f\col X\to Y} which is a levelwise weak equivalence.
Thus \w{\diN f} is also a weak equivalence.

Since the map \w{h\col DX\to Y} of \S \ref{ssrcons} is a levelwise
weak equivalence, \w{\TamR h} is a weak equivalence. In conclusion, $f$ and
$h$ are weak equivalences, so that \w[.]{\Diag X\simeq \Diag DX}
\end{proof}

\begin{defn}\label{ddisc}
We define the $0$-\emph{discretization} functor
$$
\diz\col \Gpt{n}\to\sC{\Gpt{n-1}}
$$
\noindent on any weakly globular $n$-fold groupoid $G$ as follows:
set
$$
(\diz G)_{k}~:=~\begin{cases}
G^{d}_{0} & \text{if~}k=0\\
(\Nup{1}G)_{k} & \text{if~}k>0
\end{cases}
$$
\noindent (cf.\ \S \ref{addnote}). If
\w{d_{0},d_{1}\col G_{1}\to G_{0}} are the source
and target maps, and \w{\sigma_{0}\col G_{0}\to G_{1}} is the degeneracy operator
(all in  \w[),]{\Gpd^{n-1}}   we define \w{d'_{0},d'_{1}\col (\diz
G)_{1}\to(\diz G)_{0}} and \w{\sigma'_{0}\col (\diz G)_{0}\to(\diz G)_{1}} by
\w{d'_{i}=\gamma d_{i}} \wb{i=0,1} and
\w[.]{\sigma'_{0}=\sigma_{0}\gamma'}
All other face and degeneracy operators of \w{\diz G} are those of
$G$. Since \w[,]{\gamma\gamma'=\Id} all simplicial identities hold for
\w[.]{\diz G}
\end{defn}

\begin{lemma}\label{ldisc}
For any weakly globular $n$-fold groupoid \w[,]{G\in \Gpt{n}} \w{\Dlo{n} G}
and \w{\Dlo{n}\diz G} are weakly equivalent.
\end{lemma}

\begin{proof}
\w{\Dlo{n} G} is the diagonal of the bisimplicial set $X$
with
$$
X_{k}~:=~\Dlo{n-1}(\Nup{n}G)\up{n}_{k}
$$
\noindent for all \w[,]{k\geq 0} while
\w{\Dlo{n}\diz G} is the diagonal of the bisimplicial set $Y$ with
\w{Y_{0}:=G^{d}_{0}} and \w{Y_{k}:=\Dlo{n-1}(\Nup{n}G)\up{n}_{k}} for
\w[.]{k\geq 1} By construction, \w{X\in\sS^{2}_{h}}  and \w[.]{Y=DX}
Hence, by Lemma \ref{lhtype},
\w[.]{\Dlo{n} G=\Diag X\simeq\Diag Y=\Dlo{n}\diz G}
\end{proof}

\begin{notation}\label{ntwg}
Let \w{\Twg{n}\col \Gpt{n}\to\Gpd} denote the weakly globular
\emph{fundamental groupoid} functor \wh that is, the composite
\begin{myeq}\label{eqtn}
\Twg{n}~:=~\bPz{2}\cdots\bPz{n-1}\bPz{n}
\end{myeq}
\noindent (see Definitions \ref{dbpz} and \ref{dntng}).

By construction, for all \w[,]{i\geq 0}
\begin{myeq}\label{eqtnminus}
(\Twg{n}G)_{i}=\pi_{0}\Twg{n-1}G_{i}~.
\end{myeq}
\end{notation}

\begin{defn}\label{dtndn}
For each \w[,]{n\geq 1} we define discretization functors
$$
D\sb{n}\col \Gpt{n}~\to~\sCx{n-1}{\Gpd}
$$
\noindent by induction on $n$, starting with \w[.]{D\sb{1}:=\Id\col \Gpd\to\Gpd}
For \w[,]{n\geq 2} we let \w{D\sb{n}} be the composite:
$$
\Gpt{n}~\supar{\Nup{n}}~\sC{\Gpt{n-1}}~\supar{\diz}~
 \sC{\Gpt{n-1}}~\supar{\ovl{D}_{n-1}}~\sCx{n-1}{\Gpd}
$$
\noindent where \w{\ovl{D}_{n-1}} is obtained by applying \w{D_{n-1}}
in each simplicial dimension.
\end{defn}

\begin{thm}\label{ttamsamani}
The functor \w{D_{n}} lands in \w[.]{\Tam{n}} Furthermore,
\w[,]{\TTm{n}D_{n}=\Twg{n}} and for each
\w[,]{G\in\Gpt{n}} we have a natural weak equivalence
$$
\Dlo{n} G~\simeq~\Dlo{n} D_{n}G~.
$$
\end{thm}

\begin{proof}
By induction on \w[.]{n\geq 2} For \w[,]{n=2} note that \w{D_{2}G=\diz\Nup{2}G}
is in \w{\Tam{2}} for any weakly globular double groupoid $G$, since for each
\w{k\geq 2} by Definition \ref{dntng}(iv) we have:
\begin{equation*}
(D_{2}G)_{k}~=~\pro{G_{1}}{G_{0}}{k}\simeq\pro{G_{1}}{G_{0}^{d}}{k}\simeq
\pros{(D_{2}G)_{1}}{(D_{2}G)_{0}}{k}~.
\end{equation*}

Furthermore ,
\w{\TTm{2}D_{2}G=\Twg{2}G=\bPz{2}G} is a groupoid. Hence by
definition, \w[.]{D_{2}G\in\Tam{2}} By Lemma
\ref{ldisc}, \w{\TamR D_{2}G\simeq\dN G} since \w[\vsm .]{G\in\Gpt{2}}

In the induction step, note that \w{(D_{n}G)_{0}=G^{d}_{0}} is
discrete. So to prove that \w{D_{n}G} is in \w[,]{\Tam{n}} it remains
to show:
\begin{enumerate}
\renewcommand{\labelenumi}{(\alph{enumi})~}
\item The Segal maps
$$
\seg{k}~\col ~(D\sb{n} G)\sb{k}~\to~\pro{(D\sb{n} G)\sb{1}}{(D\sb{n}G)\sb{0}}{k}~
$$
\noindent are geometric weak equivalences.
\item \w{\TTm{n} D\sb{n} G} is a groupoid\vsm .
\end{enumerate}

Note that by Definition \ref{dtndn} and by the inductive hypothesis,
for \w{k\geq 2} we have:
$$
 \Dlo{n-1} (D_{n} G)_{k}=\Dlo{n-1} D_{n-1}(\pro{G_{1}}{G_{0}}{k})\simeq
            \dN(\pro{G_{1}}{G_{0}}{k})~,
$$
\noindent and by Definition \ref{dntng}(iv) and the inductive
  hypothesis again this is weakly equivalent to
$$
\dN(\pro{G_{1}}{G_{0}^{d}}{k})~\simeq~\pro{\Dlo{n-1} D_{n-1}G_{1}}{\dN G_{0}^{d}}{k}
$$
\noindent which is \w{\Dlo{n}( \pro{(D\sb{n} G)\sb{1}}{(D\sb{n} G)\sb{0}}{k})}
by Definition \ref{dtndn}. Thus each Segal map \w{\seg{k}} is a geometric weak
equivalence. This proves (a)\vsm .

To prove (b), note that by definition of \w[,]{\TTm{n}}
\wref[,]{eqtn} and \wref[,]{eqtnminus} we have:
\begin{equation*}
\begin{split}
 (\TTm{n}(D\sb{n} G))\sb{0}~=&~\pi\sb{0}\TTm{n-1}( D\sb{n} G)\sb{0}~=~
      \pi\sb{0}\TTm{n-1} G\sp{d}\sb{0}~=~G\sp{d}\sb{0}\\
~=&~
          \pi\sb{0}\cN\Twg{n-1}G\sb{0}~=~(\Twg{n} G)\sb{0}~.
\end{split}
\end{equation*}
\noindent where \w[.]{\pi\sb{0}\TTm{n}X=\bPz{1}\bPz{2}\cdots\bPz{n}X}

Furthermore:
$$
(\TTm{n}D_{n} G)\sb{k}~=~\pi\sb{0}\TTm{n-1}(D\sb{n} G)\sb{k}~=~
\pi\sb{0}\TTm{n-1} D\sb{n-1}(\Nup{n}G)\sb{k}~
$$
\noindent for \w[.]{k\geq 1} By induction we therefore have:
\begin{equation*}
\pi\sb{0}\TTm{n-1} D\sb{n-1} (\Nup{n}G)\sb{k}~=~
\pi_{0}\Twg{n-1}(\Nup{n}G)\sb{k}~=~
(\Twg{n}G)_{k}~,
\end{equation*}

It follows that \w[,]{\TTm{n}D\sb{n} G=\Twg{n} G} as claimed. Since
\w{\Twg{n}G} is a groupoid, so is \w[.]{\TTm{n}D\sb{n} G} This concludes
the proof that \w[.]{D\sb{n}G\in\Tam{n}}

Finally, we show that \w[.]{\Dlo{n} D_{n} G\simeq\dN G}
Let \w[.]{Y=\diz\Nup{n}G\in\sC{\Gpt{n-1}}} By
Lemma \ref{ldisc}, \w[.]{\dN Y\simeq\dN G} Furthermore, \w{\Dlo{n} D_{n} G} is
the realization of the bisimplicial set $Z$ with
\w[.]{Z_{k}:=\Dlo{n-1} D_{n-1}Y_{k}} By induction,
\w[,]{Z_{k}\simeq\Dlo{n-1} Y_{k}} so that \w[,]{\Diag Z\simeq\dN Y\simeq\dN G}
as required.
\end{proof}

\begin{remark}\label{rhtpy}
Since by \cite[\S 8]{TamsN}, \w{\TamR D_{n}G} is an $n$-type, Theorem
\ref{ttamsamani} implies that the realization of a weakly globular
$n$-fold groupoid is an $n$-type. This provides an alternative proof
of the first statement in Theorem \ref{thdnt}. Moreover,
\cite[\S 5]{TamsN} provides a formula for the homotopy groups:
$$
\pi_{n}(\TamR D_{n}G,x)~=~\Aut_{\cC_{n}(D_{n}G)}(\Id_{x})~,
$$
\noindent where \w{\cC_{n}(D_{n}G)} is the groupoid \w[.]{\Wlo{n,n-1}G}
This matches \wref[.]{eqpintg}
\end{remark}

%
%
\sect{Weakly globular pseudo $n$-fold groupoids}
\label{cwgpng}

We now introduce the category \w{\PsG{n}} of weakly globular pseudo
$n$-fold groupoids, and prove Theorem \ref{tpsg}, stating that there
is a zig-zag of weak equivalences between any \w{X\in\PsG{n}} and
\w[.]{\hQlo{n}\diN X} This implies our second main result (Theorem
\ref{tpswgntype}), stating that \w{\hQlo{n}} induces an equivalence
\w[.]{\ho\PT{n}\simeq\ho\PsG{n}}

\supsect{\protect{\ref{cwgpng}}.A}
{Types of pseudo $\mathbf{n}$-fold groupoids}

The notion of a weakly globular pseudo $n$-fold groupoid is a further
relaxation of \w[,]{\Gpt{n}} similarly defined using a subcategory of
homotopically discrete objects.

\begin{defn}\label{dphdng}
For each $n$, we introduce a full subcategory \w{\PhG{n}} of
\w[,]{\sCx{n-1}{\Gpd}} whose objects are called \emph{homotopically discrete
pseudo $n$-fold group\-oids}.
These categories are defined by induction on \w[,]{n\geq 1} as follows:

\begin{enumerate}
\renewcommand{\labelenumi}{(\alph{enumi})~}
\item \www{\PhG{1}=\Ghd{1}} consists of the homotopically discrete groupoids.
\item If \w[,]{X\in\PhG{n}} then \w{X\sb{k}\in\PhG{n-1}} for all
\w[,]{k\geq 0} where $k$ is the simplicial dimension in the first direction
(cf.\ \S \ref{dntng}).
\item If \w[,]{X\in \PhG{n}} the \wwb{n-1}simplicial set \w{\opz{n} J\sb{n}X}
is the nerve of an object \w{\bPz{n}X} of \w{\PhG{n-1}} and the following
diagram commutes (where \w{J\sb{n}} denotes the inclusion)
$$
\xymatrix{
\PhG{n}\ar^{J\sb{n}}[rrrr] \ar_{\bPz{n}}[d] &&&& \sCx{n-1}{\Gpd} \ar^{\opz{n}}[d] \\
\PhG{n-1} \ar_{J\sb{n-1}}[rr] && \sCx{n-2}{\Gpd} \ar_{\ovl{N}}[rr] &&
\sCx{n-1}{\Set}\;.
}
$$
\noindent Furthermore, the map $\ovl{\gamma}$ of \S \ref{dbpz} induces a map
\w{\gamma\up{n}\col X\to\cons\up{n}\bPz{n}X} in \w{\PhG{n}}
which is a weak equivalence of groupoids in each multi-simplicial
dimension (and thus a geometric weak equivalence by Remark \ref{rgwe}).
\item For each \w[,]{k\geq 2} the induced Segal map:
\begin{myeq}\label{eqsegal}
X\sb{k}~\xrw{\iseg{k}}~\pro{X_{1}}{X_{0}^{d}}{k}
\end{myeq}
\noindent of \wref{eqindsegal} is a geometric weak equivalence.
\end{enumerate}

Note that condition (c) implies that the composite \w{\Glo{n}} of
\begin{myeq}\label{eqglon}
X~\supar{\gamma\up{n}}~\cons\up{n}\bPz{n}X~\supar{\cons\up{n}\gamma\up{n-1}}~
\cdots~\cons\up{1}\dotsc\cons\up{n}\bPz{1}\dotsc\bPz{n}X~,
\end{myeq}
\noindent is a geometric weak equivalence, so that \w{\diN X} is a homotopically
trivial simplicial set (i.e., a $0$-type)\vsm .
\end{defn}

\begin{defn}\label{dwgpng}
We now use Definition \ref{dphdng} to specify, for each \w[,]{n\geq 1}
another full subcategory \w{\PsG{n}} of \w[,]{\sCx{n-1}{\Gpd}} whose
objects are called \emph{weakly globular pseudo $n$-fold groupoids},
defined by induction on \w[.]{n\geq 1}

\begin{enumerate}
\renewcommand{\labelenumi}{(\alph{enumi})~}
\item \www[.]{\PsG{1}:=\Gpd}
\item If \w[,]{X\in\PsG{n}} then \w{X\sb{0}\in\PhG{n-1}} and
\w{X\sb{k}\in\PsG{n-1}} for all \w{k\geq 1}
\item If \w[,]{X\in \PhG{n}} the \wwb{n-1}simplicial set \w{\opz{n} J\sb{n}X}
is the nerve of an object \w{\bPz{n}X} of \w{\PhG{n-1}} and the following
diagram commutes (where \w{J\sb{n}} denotes the inclusion)
$$
\xymatrix @R=35pt{
\PsG{n}\ar^{J\sb{n}}[rrrr] \ar_{\bPz{n}}[d] &&&& \sCx{n-1}{\Gpd} \ar^{\opz{n}}[d] \\
\PsG{n-1} \ar_{J\sb{n-1}}[rr] && \sCx{n-2}{\Gpd} \ar_{\ovl{N}}[rr] &&
\sCx{n-1}{\Set}\;.
}
$$
\noindent Furthermore, \w{\bPz{n}} preserves geometric weak equivalences.
\item For each \w[,]{k\geq 2} the induced Segal map
$$
X\sb{k}~\xrw{\iseg{k}}~\pro{X_{1}}{X_{0}^{d}}{k}
$$
\noindent is a geometric weak equivalence.
\end{enumerate}
\end{defn}

\begin{remark}\label{rpsg}
Both \w{\Tam{n}} and \w{\Gpt{n}} are full subcategories of \w[,]{\PsG{n}}
and \w{\Ghd{n}} is a full subcategory of \w[.]{\PhG{n}}
\end{remark}

\smallskip
\begin{example}\label{egpsg}
When \w[,]{n=2} a weakly globular pseudo double groupoid is just a
simplicial object in groupoids \w{X\in\sC{\Gpd}} such that \w{X\sb{0}}
is a homotopically discrete groupoid, the simplicial set \w{\opz{2}X} is
the nerve of a groupoid, and for each \w[,]{k\geq 2} the induced Segal map
$$
X\sb{k}~\xrw{\iseg{k}}~\pro{X_{1}}{X_{0}^{d}}{k}
$$
\noindent is an equivalence of groupoids.
\end{example}

\smallskip
\begin{lemma}\label{lpsgcoprod}
If \w{f\col X\to Y} is a map in \w[,]{\PsG{n}} and $Y$ is discrete (that is, the
constant \wwb{n-1}fold simplicial object on a discrete groupoid), then
$X$ is the coproduct in \w{\PsG{n}} of the fibers \w[,]{X\sp{-1}(a)}
taken over all \w[.]{a\in Y}
\end{lemma}

\begin{proof}
By induction on \w[,]{n\geq 1} where for \w[,]{n=1} $X$ is a groupoid,
which is a coproduct of its connected components. The $n$-th step follows
from the \wwb{n-1}st, since coproducts in \w{\PsG{n}} are those of
\w[,]{\sCx{n-1}{\Gpd}} namely, disjoint unions, which are therefore
taken dimensionwise.
\end{proof}

\smallskip
\begin{corollary}\label{cpsgcoprod}
If \w[,]{X\in\PsG{n}} then \w{X\sb{1}} is isomorphic to
the coproduct in \w{\PsG{n-1}} of \w{X\sb{1}(a,b)}
(the fiber of
\w[,]{(\gamma\lo{n}d\sb{0},\gamma\lo{n}d\sb{1})\col X\sb{1}\to
X\sp{d}\sb{0}\times X\sp{d}\sb{0}}
taken over all \w[).]{(a,b)\in X\sp{d}\sb{0}\times X\sp{d}\sb{0}}
%
\end{corollary}

\smallskip
\begin{defn}\label{dneqpsg}
We now define the notion of \emph{$n$-equivalence} for maps of
weakly globular pseudo $n$-fold groupoids by induction on \w[,]{n\geq 1}
where a $1$-equivalence is simply an equivalence of groupoids:

A map \w{f\col X\to Y} in \w{\PsG{n}} is an $n$-equivalence if:

\begin{enumerate}
\renewcommand{\labelenumi}{(\alph{enumi})~}
\item \www{\bPz{n}f\col \bPz{n}X\to\bPz{n}Y} is an \wwb{n-1} equivalence in
\w[;]{\PsG{n-1}}
\item For every \w[,]{a,b\in X\sb{0}\sp{d}} the map
\w{f(a,b)\col X\sb{1}(a,b)\to Y\sb{1}(f(a),f(b))} is also an \wwb{n-1} equivalence
in \w[.]{\PsG{n-1}}
\end{enumerate}
\end{defn}

\supsect{\protect{\ref{cwgpng}}.B}
{Comparison with Tamsamani's weak $\mathbf{n}$-groupoids}

We describe a procedure for transforming a weakly globular pseudo
$n$-fold groupoid $X$ into a Tamsamani weak $n$-groupoid, without altering
the homotopy type.  The construction is done in two stages:

In the first, we use the general construction of \S \ref{ssrcons} to produce
\w[,]{\diz X\in\PsG{n}} in which only  \w{X\sb{0}} is discretized (as in
Subsection \ref{ctmng}.B).
This time we must proceed by induction on the $n$ simplicial directions
in order to obtain a zig-zag of intermediate objects (in Lemma \ref{ldiscpsg}),
all weakly equivalent in \w{\PsG{n}} (which was not possible in \w[).]{\Gpt{n}}

In the second stage, we define the full discretization
functors \w{D\sb{n}\col \PsG{n}\to\Tam{n}} by induction on \w[,]{n\geq 2}
with \w[,]{D\sb{2}:=\diz} so as to make each \w{X\sb{k}} a Tamsamani weak
\wwb{n-1}groupoid\vsm .

First, we need some technical facts about weakly globular pseudo
$n$-fold groupoids:

\begin{lemma}\label{lsegal}
If \w{f\col X\to Y} is a map in \w{\sC{\PsG{n-1}}} which is a weak equivalence in
each simplicial dimension, with \w{Y\sb{0}\in\PhG{n-1}} and \w[,]{X\in\PsG{n}}
then for each \w{k\geq 2} the induced Segal maps of \wref{eqsegal}
for $Y$ are geometric weak equivalences.
\end{lemma}

\begin{proof}
First note that $f$ induces an isomorphism
\w[,]{X\sb{0}\sp{d}\cong Y\sb{0}\sp{d}} so by Corollary \ref{cpsgcoprod}
\w{f_{1}\col X\sb{1}\to Y\sb{1}} is the coproduct over
\w{(a,b)\in X\sp{d}\sb{0}\times X\sp{d}\sb{0}} of its restrictions
\w[.]{f\sb{1}(a,b)\col X\sb{1}(a,b)\to Y\sb{1}(a,b)}
Since the classifying space functor \w{\diN\col \PsG{n-1}\to\Top} commutes
with disjoint unions, the fact that \w{f\sb{1}} is a weak equivalence
implies that each \w{f\sb{1}(a,b)} is a geometric weak equivalence in
\w[.]{\PsG{n-1}}

Moreover, since \w{X\sb{0}\sp{d}} is discrete,
\begin{myeq}\label{eqsegaldeco}
\pro{X\sb{1}}{X\sb{0}\sp{d}}{k}~\cong~
\coprod\sb{a_{0},\dotsc, a_{k}\in X\sb{0}\sp{d}}\
X\sb{1}(a\sb{0},a\sb{1})\times X\sb{1}(a\sb{1},a\sb{2})\times\dotsc\times
X\sb{1}(a\sb{k-1},a\sb{k})
\end{myeq}
\noindent and similarly for $Y$.

Now consider the commutative diagram of induced Segal maps:
\mydiagrm[\label{eqisegalmaps}]{
X\sb{k} \ar[rr]^{\iseg{k}\sp{X}}_{\simeq} \ar[d]_{f\sb{k}}^{\simeq} &&
\pro{X\sb{1}}{X\sb{0}\sp{d}}{k} \ar[d]^{(f\sb{1}\times\dotsc\times f\sb{1})}\\
Y\sb{k} \ar[rr]^{\iseg{k}\sp{Y}}  && \pro{Y\sb{1}}{Y\sb{0}\sp{d}}{k}
}
\noindent where the left vertical map is a geometric weak equivalence by
assumption, as is the top horizontal induced Segal map
(since \w[),]{X\in\PsG{n}} while the right horizontal map is a geometric
weak equivalence because of \wref[.]{eqsegaldeco}  Therefore,
\w{\seg{k}\sp{Y}} is a geometric weak equivalence, too.
\end{proof}

\begin{lemma}\label{lpushouts}
Consider a pushout diagram in \w[:]{\sCx{n-1}{\Gpd}}
\mydiagrm[\label{eqpopsg}]{
A \ar@{^{(}->}[rr]^{j} \ar[d]_{\gamma\up{n}}^{\simeq} && B \ar[d]^{g} \\
\cons\up{n}\bPz{n} A \ar[rr]^{h} && C
}
\noindent with \w{A\in\PhG{n}} and $j$ monic. Then:

\begin{enumerate}
\renewcommand{\labelenumi}{(\alph{enumi})~}
\item If \w[,]{B\in\PsG{n}} so is $C$.
\item If \w[,]{B\in\PhG{n}} so is $C$.
\end{enumerate}
\end{lemma}

\begin{proof}
By induction on \w[:]{n\geq 1}

First note that for any \w[,]{n\geq 1} $g$ is a geometric weak equivalence,
since $f$ is, because \w{\Dlo{n}} preserves pushouts, $A$ is in \w[,]{\PhG{n}}
and \w{\Dlo{n}j} is a cofibration of simplicial sets.

When \w[,]{n=1} \wref{eqpopsg} is a diagram of groupoids, so (a) is clear, and
(b) follows from \cite[Corollary 3]{JStreP}.

In general, since the pushout is taken in a diagram category, \w{C\sb{0}} is
the pushout of the objects in simplicial dimension $0$, which is therefore
in \w{\PhG{n-1}} by (b) for \w[,]{n-1} while for \w[,]{k\geq 1} \w{C_{k}} is
in \w{\PsG{n-1}} by (a) for \w[.]{n-1}

Since the functor \w{\bPz{n}} is defined by applying \w{\pi_{0}} to each
groupoid, \w{\pi_{0}} commutes with pushouts of groupoids, and \w{\pi_0\gamma}
is an isomorphism, we see that \w{\bPz{n}C=\Pi_0\up{n}B} is in
\w{\PsG{n-1}} by (a) for \w[.]{n-1}

Finally, the Segal condition follows from Lemma \ref{lsegal} for $g$, since
\w{g\sb{k}} is a weak equivalence for each \w[,]{k\geq 0} \w[,]{B\in\PsG{n}}
and \w[.]{C\sb{0}\in\PhG{n-1}}

 This shows (a). (b) is immediate.
\end{proof}

\begin{prop}\label{ppsgpushout}
Assume given a weakly globular pseudo $n$-fold groupoid $X$, and let
\w{Y\in\sCx{n-1}{\Gpd}} be the result of applying the construction of
\S \ref{ssrcons} to the map \w{\gamma\col X\sb{0}\to W} for
\w{W:=(\cons\up{n}\bPz{n}X)\sb{0}} and \w[;]{\cC=\sCx{n-1}{\Gpd}}
then $Y$ is actually in \w[.]{\PsG{n}} Moreover, the maps
\begin{myeq}\label{eqpsgpushout}
X~\xrw{f}~Y~\xlw{h}~X\sp{\gamma}
\end{myeq}
\noindent are geometric weak equivalences in \w[,]{\PsG{n}} where
\w{X\sp{\gamma}} is as in \S \ref{ssrcons}.
\end{prop}

\begin{proof}
First, note that \w{Y\sb{0}:=W} is in \w[,]{\PhG{n-1}} by
\S \ref{dwgpng}. Furthermore, for any \w{k\geq 1} \w{Y\sb{k}} is defined by
the pushout square of \wref[:]{eqsimpcons}
\mydiagrm[\label{eqsimpconspsg}]{
    X\sb{0} \ar[r]^{s\lo{k}} \ar[d]_{\gamma} & X_{k} \ar[d]^{f_{k}}\\
    (\cons\up{n}\bPz{n}X)\sb{0} \ar[r]_{\sigma\lo{k}} & Y_{k}
    }
\noindent where $\gamma$ is a geometric weak equivalence since \w{X\sb{0}}
is in \w[,]{\PhG{n-1}} and the iterated degeneracy map \w{s\lo{k}} is
one-to-one since it has a left inverse \w[.]{d\lo{k}} Thus by
Lemma \ref{lpushouts}, \w[.]{Y_k\in\PsG{n-1}}

The maps \w{f\sb{k}} in \wref{eqsimpconspsg} are geometric weak equivalences,
since after applying \w{\Dlo{n}} we obtain a pushout of a weak equivalence
along a cofibration in \w[.]{\sS} Therefore, by Lemma \ref{lsegal} applied
to $f$, the induced Segal maps for $Y$ are weak equivalences.

Finally, \w{\bPz{n}Y} is obtained by applying \w{\pi_{0}} to each groupoid of
\w[,]{Y\in\sCx{n-1}{\Gpd}} and since this commutes with pushouts and
\w{\pi\sb{0}\gamma} is an isomorphism, we see that
\w[,]{\bPz{n}Y\cong\bPz{n}X} so in particular it is in \w[.]{\PsG{n-1}}
This shows that \w[.]{Y\in\PsG{n}}

Since each \w{f\sb{k}} is a geometric weak equivalence, as is
\w{h\sb{0}=\gamma} and \w{h\sb{k}=\Id} for \w[,]{k\geq 1}
the two maps $f$ and $h$ are geometric weak equivalences in \w[.]{\PsG{n}}
\end{proof}

\begin{notation}\label{nfundgppsg}
Let \w{\Tps{n}\col \PsG{n}\to\Gpd} denote the \emph{fundamental groupoid} functor
for \w{\PsG{n}} \wwh that is, the composite
\begin{myeq}\label{eqtsgn}
\Tps{n}~:=~\bPz{2}\cdots\bPz{n-1}\bPz{n}
\end{myeq}
\end{notation}

\begin{defn}\label{ddiscpsg}
For each \w{n\geq 2} we define a sequence of functors
\w{\diz\up{k}\col \PsG{n}\to\PsG{n}} \wb{1\leq k\leq n} by setting
\w{\diz\up{k}X:=X\sp{\Glo{n}\up{k}}} (in the notation of \S \ref{ssrcons}),
where
$$
\Glo{n}\up{k}\col X\sb{0}\to
(\cons\up{k}\dotsc\cons\up{n}\bPz{k}\dotsc\bPz{n}X)\sb{0}
$$
\noindent is the composite of the first $k$ maps of \wref{eqglon} in
dimension $0$. We write \w{\diz} for \w[.]{\diz\up{1}}
\end{defn}

\begin{lemma}\label{ldiscpsg}
For each \w{X\in\PsG{n}} we have a sequence of natural geometric
weak equivalences
$$
\xymatrix@R=15pt@C=15pt{
X \ar[rd]^{f\up{n}} && \diz\up{n}X \ar[ld]_{h\up{n}}\ar[rd]^{f\up{n-1}} &&&&
\diz X \ar[ld]_{h\up{1}} \\
& Y\up{n} && Y\up{n-1}&\dotsc& Y\up{1} &
}
$$
\end{lemma}

\begin{proof}
Each \w{Y\up{k}} is obtained by applying Proposition \ref{ppsgpushout} to
\w[,]{\diz\up{k+1}X} where \w[,]{\diz\up{n+1}X:=X} and using
\wref[.]{eqpsgpushout}
\end{proof}

\begin{defn}\label{ddnpsg}
We now define \emph{discretization functors}
$$
D\sb{n}~\col ~\PsG{n}~\to~\sCx{n-1}{\Gpd}
$$
\noindent for each \w{n\geq 1} by induction on $n$, starting with
\w[.]{D\sb{1}:=\Id\col \Gpd\to\Gpd}
For \w[,]{n\geq 2} we define \w{D\sb{n}} inductively to be the composite:
$$
\PsG{n}~\hra~\sC{\PsG{n-1}}~\supar{\diz}~
 \sC{\PsG{n-1}}~\supar{\ovl{D}_{n-1}}~\sCx{n-1}{\Gpd}
$$
\noindent where \w{\ovl{D}_{n-1}} is obtained by applying \w{D_{n-1}}
in each simplicial dimension.
\end{defn}

Note that \w{D\sb{2}} is simply \w[.]{\diz\col \PsG{2}\to\Tam{2}}

\begin{thm}\label{tpsg}
The functor \w{D_{n}} lands in \w{\Tam{n}} and preserves geometric weak
equivalences and fiber products over discrete objects. Moreover,
for every weakly globular pseudo $n$-fold groupoid \w[,]{X\in\PsG{n}}
the groupoid \w{\TTm{n} D\sb{n}X} is isomorphic to
\w[,]{\Tps{n}X} and there is a zigzag of weak equivalences in \w{\PsG{n}}
between \w{D\sb{n}X} and $X$.
\end{thm}

\begin{proof}
By induction on \w[.]{n\geq 2} For \w[,]{n=2} \w{D\sb{2}X=\diz X} is
clearly in \w{\Tam{2}} for any \w[.]{X\in\PsG{2}}

In the induction step, note that \w{(D\sb{n}X)\sb{0}=X\sp{d}\sb{0}} is
discrete and \w{(D\sb{n}X)\sb{k}=D\sb{n-1}X\sb{k}} is in
\w[,]{\Tam{n-1}} by induction. So to prove that \w{D\sb{n}X} is in
\w[,]{\Tam{n}} it remains to show:
\begin{enumerate}
\renewcommand{\labelenumi}{(\alph{enumi})~}
\item The Segal maps
\begin{myeq}\label{eqsegaltams}
\seg{k}\col (D\sb{n}X)_{k}\to\pro{(D\sb{n}X)\sb{1}}{(D\sb{n}X)\sb{0}}{k}
\end{myeq}
\noindent are \wwb{n-1}equivalences.
\item \w{\bPz{n}D\sb{n}X} is in \w[\vsm.]{\Tam{n-1}}
\end{enumerate}

To show (a), note that since \w[,]{X\in\PsG{n}} the induced Segal maps
$$
X_{k}~\xrw{\iseg{k}}~\pro{X\sb{1}}{X\sp{d}\sb{0}}{k}
$$
\noindent are geometric weak equivalences for all \w[.]{k\geq 2}
Since by induction \w{D\sb{n-1}} preserves geometric weak equivalences,
we have weak equivalences:
$$
D\sb{n-1}X_{k}~\xrw{\simeq}~D\sb{n-1}(\pro{X\sb{1}}{X\sp{d}\sb{0}}{k})~.
$$
\noindent Moreover, \w{(D\sb{n}X)\sb{1}=D\sb{n-1}X\sb{1}} and
\w{(D\sb{n}X)\sb{0}=X\sp{d}\sb{0}} is discrete, so the right hand side is an
iterated fiber product over discrete objects, and thus (again by induction)
$$
D\sb{n-1}(\pro{X\sb{1}}{X\sb{0}}{k})~=~
\pro{(D\sb{n}X)\sb{1}}{(D\sb{n}X)\sb{0}}{k}~,
$$
\noindent which proves (a) for $n$\vsm.

To show (b), by \S \ref{sbptam} and Proposition \ref{ptequivdefns} it
suffices to show that \w{\TTm{n}D\sb{n}X} is a groupoid, which we do
by showing that it is isomorphic to \w[.]{\Tps{n}X} We have
$$
 (\TTm{n}D\sb{n} X)_{0}~=~\pi\sb{0}\TTm{n-1}( D\sb{n}X)\sb{0}~=~
      \pi\sb{0}\TTm{n-1}X\sp{d}\sb{0}~=~X\sp{d}\sb{0}~.
$$
\noindent and
\begin{equation*}
\begin{split}
(\TTm{n}D\sb{n} X)\sb{k}~=&~\pi\sb{0}\TTm{n-1}(D\sb{n}X)\sb{k}~=~
\pi\sb{0}\TTm{n-1}D\sb{n-1}X\sb{k}\\
~=&~\pi\sb{0}\Tps{n-1}X\sb{k}~=~(\Tps{n}X)\sb{k}~.
\end{split}
\end{equation*}
\noindent for \w[,]{k\geq 1} where we use the induction hypothesis for the
equality before last.

It follows that \w[,]{\TTm{n} D_{n} X=\Tps{n} X} and since the latter is a
groupoid, so is \w[.]{\TTm{n}D_{n} X} This concludes the proof that
\w{D_{n}X} is in \w[\vsm.]{\Tam{n}}

Finally, we obtain the required natural zig-zag of geometric
weak equivalences:
$$
D\sb{n}X \to \dotsc \leftarrow \diz X \to \dotsc \leftarrow X~,
$$
\noindent by induction on \w[,]{n\geq 1} where the righthand zig-zag is
provided by Lemma \ref{ldiscpsg}:

For \w[,]{n=1}  we have \w[,]{D\sb{1}X=X}  while for \w{n\geq 2} we use
Definition \ref{ddnpsg} to identify \w{(D\sb{n}X)\sb{k}} with
\w{(\ovl{D}\sb{n-1}X)\sb{k}} for \w[:]{k\geq 1}
$$
\xymatrix@R=25pt{
D\sb{n}X \ar[d] &&
\dotsc D\sb{n-1}X\sb{2}\ar[d]\ar@<1.5ex>[r]\ar[r]\ar@<-1.5ex>[r]
& D\sb{n-1}X\sb{1} \ar[d] \ar@<1ex>[r] \ar@<-1ex>[r]  &
X\sb{0}\sp{d} \ar[d]^{=} \\
\vdots && \vdots & \vdots & X\sb{0}\sp{d} \\
\diz X \ar[u] && \dotsc X\sb{2}\ar[u]\ar@<1.5ex>[r]\ar[r]\ar@<-1.5ex>[r]  &
X\sb{1}\ar[u] \ar@<1ex>[r] \ar@<-1ex>[r]  & X\sb{0}\sp{d} \ar[u]_{=}
}
$$
\noindent using the induction to obtain the righthand vertical zig-zag in
each simplicial dimension.
\end{proof}

\begin{remark}\label{rextendd}
The functor \w{D\sb{n}\col \PsG{n}\to\Tam{n}} extends the functor
\w{D\sb{n}\col \Gpt{n}\to\Tam{n}} of \S \ref{dtndn}.
\end{remark}

\begin{remark}\label{rextendb}
It follows from Theorem \ref{tpsg} that if \w[,]{X\in\PsG{n}} \w{\diN X} is
an $n$-type.
\end{remark}

\begin{defn}\label{dsimbps}
Let \w{\ho\PsG{n}} denote the localization of the category \w{\PsG{n}}
with respect to the geometric weak equivalences.
\end{defn}

\begin{thm}\label{tpswgntype}
The functors \w{\hQlo{n}\col \Top\to\Gpt{n}} and \w[,]{\diN\col \PsG{n}\to\Top}
together with the inclusion \w[,]{J\col \Gpt{n}\hra\PsG{n}}
induce equivalences of categories
\begin{myeq}\label{eqpswgequiv}
\ho\PT{n}~\adj{\diN}{J\circ\hQlo{n}}~\ho\PsG{n}~.
\end{myeq}
\noindent Moreover, for every \w[,]{T\in\Top} there is a zigzag of
weak equivalences in \w{\PT{n}} between \w{\Po{n}T} and \w[,]{\diN\hQlo{n}T}
and for \w{X\in\PsG{n}} there is a zig-zag of geometric weak equivalences
in \w{\PsG{n}} between $X$ and \w[.]{\hQlo{n}\diN X}
\end{thm}

\begin{proof}
All three functors preserve weak equivalences, so we have induced functors
as in \wref[.]{eqpswgequiv} For any $n$-type $T$, we have an
isomorphism in \w{\ho\PT{n}} between \w{T} and \w{\diN\hQlo{n}T}
 by Theorem \ref{teqcat}, which also
implies (see Remark \ref{rextendb}) that for any \w{X\in\PsG{n}} we have a
homotopy equivalence (of CW complexes) in \w[:]{\Top}
\begin{myeq}\label{eqbqb}
\diN X~\xrw{\simeq}~\diN\hQlo{n}\diN X~.
\end{myeq}

By Theorem \ref{tpsg} we also have zig-zags of geometric weak equivalences in
 \w[:]{\PsG{n}}
\begin{myeq}\label{eqbqbqn}
D\sb{n}X~\to\dotsc \leftarrow ~X\hspace*{7mm}\text{and}\hspace*{5mm}
D\sb{n}\hQlo{n}\diN X~~\to\dotsc \leftarrow  ~\hQlo{n}\diN X
\end{myeq}
\noindent Therefore, after applying $\diN$ to \wref{eqbqbqn} we have
homotopy equivalences of CW complexes:
$$
\diN D\sb{n}X~\xrw{\simeq}~\diN X\hspace*{7mm}\text{and}\hspace*{5mm}
\diN \hQlo{n}\diN X~\xrw{\simeq}~\diN D\sb{n}\hQlo{n}\diN X~.
$$
Combining these with \wref{eqbqb} yields a weak equivalence
$$
\diN D\sb{n}X~\to~\diN D\sb{n}\hQlo{n}\diN X
$$
\noindent in \w[,]{\Top} which by
Theorem \ref{ttamseq} implies that \w{D\sb{n}X} and
\w{D\sb{n}\hQlo{n}\diN X} are isomorphic in \w[,]{\ho\Tam{n}}
and thus in \w[.]{\ho\PsG{n}} By \wref{eqbqbqn} we see that $X$ and
\w{J\hQlo{n}\diN X} are weakly equivalent through a zig-zag in \w[.]{\PsG{n}}
\end{proof}

\begin{remark}\label{rzigzag}
Note that Theorem \ref{tpsg} implies that the functor \w{D\sb{n}} induces
an equivalence of categories
$$
\w{\ho\PsG{n} \simeq \ho\Tam{n}}\;.
$$
Together with Theorem \ref{teqcat} and Theorem \ref{ttamseq} this implies
the equivalence of categories \eqref{eqpswgequiv}.
In the course of the proof of Theorem \ref{tpswgntype}
we have further shown that any weakly globular pseudo $n$-fold
groupoid \w{X\in\PsG{n}} has two different
functorial partial strictifications:  the Tamsamani weak $n$-groupoid
\w[,]{D\sb{n}X} and the weakly globular $n$-fold groupoid
\w[,]{\hQlo{n}\diN X\in\Gpt{n}} each equipped with zig-zags of weak
equivalences in \w{\PsG{n}} from $X$:
\begin{myeq}\label{eqtwozig}
D\sb{n}X~\to\dotsc \leftarrow ~X~\to\dotsc \leftarrow  ~\hQlo{n}\diN X~.
\end{myeq}
\end{remark}

\begin{defn}\label{diaopsg}
As in \S \ref{diao}, for any weakly globular pseudo $n$-fold groupoid $X$
and \w[,]{1\leq k\leq n} we define its \emph{$k$-fold object of arrows} to
be the pseudo \wwb{n-k}fold groupoid
\w[.]{\Wlo{n,k}X~:=~X\up{1\dotsc k}_{1\underset{k}{\cdots}1}}
\end{defn}

\begin{mysubsection}{Algebraic homotopy groups and algebraic weak
equivalences}
\label{sahgpps}
In analogy to \S \ref{sahgp}, for any weakly globular pseudo $n$-fold
groupoid \w[,]{X\in\PsG{n}} we define the \emph{$k$-th algebraic homotopy
group} of $X$ at \w{x\sb{0}\in X\sb{0\underset{n}{\cdots}0}} to be:
\begin{myeq}[\label{eqpinagps}]
\omega_{k}(X;x_{0})~\cong~
\begin{cases}
\Wlo{n,n}X(x_{0},x_{0}) & \hs\text{if}~k=n\\
\Wlo{n-k,n-k}(\bPz{k+1}\dotsc\bPz{n}X)(x_{0},x_{0}) &
\hs\text{if}~0<k<n\\
\end{cases}
\end{myeq}
with the \emph{$0$-th algebraic homotopy set} of $X$ defined:
$$
\omega_{0}(X)~:=~\bPz{1}\dotsc\bPz{n}X~.
$$

A map \w{f\col X\to Y} of weakly globular pseudo $n$-fold groupoids is called an
\emph{algebraic weak equivalence} if it induces bijections on the $k$-th
algebraic homotopy groups (set) for all \w{x_{0}\in X_{0\underset{n}{\cdots}0}} and
\w[.]{0\leq k\leq n}
\end{mysubsection}

\begin{remark}\label{rhtpyps}
As for weakly globular $n$-fold groupoids (see Remark \ref{rhtpy}), our
definition of algebraic homotopy groups for \w{\PsG{n}}
generalizes that of \cite[\S 5]{TamsN}, and since \w{D\sb{n}X} and $X$
by Remark \ref{rzigzag} have the same algebraic homotopy groups,
by construction, both provide an algebraic way of calculating the
homotopy groups of \w[,]{\diN X} as in Theorem \ref{thdnt}.

Using this fact, one can show that a map \w{f\col X\to Y} in \w{\PsG{n}} is an
$n$-equivalence (Definition \ref{dneqpsg}) if and only if it is a geometric
weak equivalence.
\end{remark}

%
%
\sect{Applications and further directions}
\label{cappl}

In this section we provide an application for our model of $n$-types,
and indicate some directions for future work.


\supsect{\protect{\ref{cappl}}.A}
{Modelling $(\mathbf{k-1})$-connected $\mathbf{n}$-types}

We now provide an algebraic model of \wwb{k-1}connected
$n$-types, and relate it to the homotopy types of iterated loop
spaces. This was mentioned in \cite{BDolaH} as a desirable feature for
models of $n$-types (see also \cite{BergD}).

Recall that a space $X$ is \emph{\wwb{k-1}connected} if \w{\pi_{0}X=0}
and  \w{\pi_{i}(X,x)=0} for \w[,]{1\leq i\leq k-1} and all \w[.]{x\in X}
We denote the category of \wwb{k-1}connected pointed $n$-types by
\w[.]{\Pud{n}{k}}

\begin{lemma}\label{lkconn}
If $X$ is a $k$-connected pointed Kan complex, $X$ is naturally weakly
equivalent to a \wwb{k-1}\emph{reduced} Kan complex $\hat{X}$ \wh that is,
\w{\hat{X}_{i}=\{\ast\}} for \w[.]{1\leq i\leq k-1}
\end{lemma}

\begin{proof}
See \cite[III, \S 3]{GJarS}.
\end{proof}

\begin{defn}\label{dredsing}
For any $k$-connected pointed topological space \w[,]{T\in\Topa} let
\w{\cS\red T} denote the canonical $k$-reduced version \w{\widehat{\cS T}}
of the singular set \w[.]{\cS T}
\end{defn}

\begin{defn}\label{dcont}
A homotopically discrete pseudo $n$-fold groupoid \w{X\in\PhG{n}} is
\emph{contractible} if \w{\pi_{0}\diN X} is trivial (so that
\w{\diN X} is contractible).

More generally, a weakly globular pseudo $n$-fold groupoid \w{X\in\PsG{n}}
is called \wwb{n,k}\emph{weakly globular} if for each \w[,]{0\leq r< k} the
homotopically discrete pseudo \wwb{n-r-1}fold groupoid
$$
X\up{1\dotsc r+1}_{{1\underset{r}{\cdots}10}}~=~(\Wlo{n,r}X)\up{r+1}_{0}
$$
\noindent is contractible. This is the pseudo \wwb{n-r-1}fold groupoid of
objects of the pseudo \wwb{n-r}fold groupoid \w{\Wlo{n,r}X\in\PsG{n-r}}
(see \S \ref{diaopsg}).

In particular, when \w[,]{r=0} this just means that the pseudo \wwb{n-1}fold
 groupoid of objects \w{X\up{n}_{0}} of $X$ in the $n$-th direction
 (which is a homotopically discrete pseudo \wwb{n-1}fold groupoid) is in fact
contractible.

We let \w{\PsGk{n,k}} denote the full subcategory of
\wwb{n,k}weakly globular pseudo $n$-fold groupoids in \w[.]{\PsG{n}}
Similarly, \w{\Gptk{n,k}} is the full subcategory of
\wwb{n,k}weakly globular pseudo $n$-fold groupoids in \w[.]{\Gpt{n}}
\end{defn}

We now want to show that \w{\PsGk{n,k}} is an algebraic model of
\wwb{k-1}connected $n$-types. For this, we need the following:

\begin{lemma}\label{lcontract}
If $X$ is a \wwb{k-1}reduced Kan complex, then \w{\Qlo{n}X} is
\wwb{n,k}weakly globular.
\end{lemma}

\begin{proof}
By Lemma \ref{lnerve}(b), \w[.]{(\Qlo{n}X)\up{n}_{0}=\Qlo{n-1}\Dex} Since
\w{\Dex\simeq c(X_{0})=c(\ast}) and \w{\Qlo{n-1}} preserves weak equivalences
of Kan complexes by Proposition \ref{pnequiv}(b), we have
\w[.]{\Qlo{n-1}\Dex\simeq\Qlo{n-1}(\ast)=\ast}
Therefore, \w{\dN(\Qlo{n} X)\up{n}_{0}} is contractible.

We now show by induction on \w{1\leq r<k} that
\begin{myeq}\label{eqqn}
\Wlo{n,r}\Qlo{n-1}X~:=~
(\Nup{n-r+1,\dotsc,n}\Qlo{n}X)\up{n-r+1,\dotsc,n}_{{1\underset{r}{\cdots}1}}~=
~\Qlo{n-r}\Llo{1}^{r}X~
\end{myeq}
\noindent (in the notation of \wref[,]{eqllo} where
\w{\Llo{1}^{r}X:=\Llo{1}^{r-1}(\Llo{1}X)} for \w[,]{r\geq 2} and
\w[).]{\Llo{1}^1 X=\Llo{1}X)} The case \w{r=1} is
Lemma \ref{lnerve}(b) for \w[,]{k=1} which implies that
we have an isomorphism of \wwb{n-1}fold groupoids:
\begin{myeq}\label{eqllonr}
\Wlo{n,1}\Qlo{n-1}X~:=~(\Nup{n}\Qlo{n}X)\up{n}_{1}~\cong~\Qlo{n-1}\Llo{1}X~.
\end{myeq}

In the induction step,  since \w{\Llo{1}X} is still a Kan complex, by
\wref{eqiao} we can apply the induction hypothesis to the right hand side
of \wref{eqllonr} (using the fact that \w[,]{\Wlo{n,r}=\Wlo{n-1,r-1}\Wlo{n,1}}
by \S \ref{eqiao}), to deduce that:
$$
\Wlo{n,r}\Qlo{n-1}\Llo{1}X~\cong~\Qlo{n-r}\Llo{1}^{r-1}(\Llo{1}X)~,
$$
\noindent which yields \wref[.]{eqqn} From this and Lemma \ref{lnerve}(a)
(for \w[)]{k=0} we have
$$
(\Wlo{n,r}\Qlo{n}X)_{0}~=~
\Qlo{n-r-1}\Dec\Llo{1}^{r}X~,
$$
\noindent and since \w[,]{\Dec\Llo{1}^{r}X\simeq c(\Llo{1}^{r}X)_{0}}
we have \w[.]{\dN(\Wlo{n,r}\Qlo{n}X)_{0}\simeq c(\Llo{1}^{r}X)_{0}}

Note that since $X$ is \wwb{k-1}reduced, \w[,]{\Dex} and thus
\w[,]{\Llo{1}X} are \wwb{k-2}reduced, so by induction \w{\Llo{1}^{r}X} is
\wwb{k-r-1}reduced. Thus as long as \w[,]{r<k} \w{\Llo{1}^{r}X} is
$0$-reduced, so \w{\dN(\Wlo{n,r}\Qlo{n}X)_{0}} is contractible.
\end{proof}

\begin{prop}\label{pkconn}
The functors \w{\Qlo{n}} and $\diN$ induce equivalences of categories:
\begin{equation*}
\ho\Pud{n}{k}~\adj{\diN}{J\circ\Qlo{n}\circ\cS\red}~\ho\PsGk{n,k}~.
\end{equation*}
\end{prop}

\begin{proof}
If \w[,]{T\in\Pud{n}{k}} then \w{\cS\red T} is \wwb{k-1}reduced, so
 by Lemma \ref{lcontract}, \w[.]{\Qlo{n}X\in \PsGk{n,k}} The result follows
immediately from Theorem \ref{tpswgntype}.
\end{proof}

\begin{remark}\label{rnkloop}
Note that the composition of \w{\Wlo{n,k}} of \S \ref{diaopsg} with
the classifying space functor \w{\diN} lands in \w[,]{\PT{n}} by Theorem
\ref{teqcat}, so its restriction to \w{\PsG{(n,k)}} lands in the category
\w{\PT{n-k}} of \wwb{n-k}types.

Moreover, if \w{T=\Omega^{k}Y} is an \wwb{n-k}type $k$-fold loop space, applying
the $k$-fold delooping functor \w{E\lo{k}\col \Po{n-k}_{\Omega^k}\to\Pud{n}{k}} of
\cite[Theorem 13.1]{MayG} yields the \wwb{k-1}connected $n$ type
\w[.]{Y=E\lo{k}T\in\Pud{n}{k}} In fact:
\end{remark}

\begin{prop}\label{pkloop}
For any \wwb{k-1}connected $n$ type \w[,]{Y\in\Pud{n}{k}} we have a
zigzag of weak equivalences in  \w{\PT{n-k}} between \w{\diN\Wlo{n,k}\hQlo{n}Y} and
\w[,]{\Omega^{k}Y} so the weakly globular \wwb{n-k}fold groupoid
\w{\Wlo{n,k}\hQlo{n}Y} is an algebraic model for \w[.]{\Omega^{k}Y}
\end{prop}

\begin{proof}
By induction on $k$. Let \w[,]{G:=\hQlo{n}Y\in\Gpt{(n,k)}} so
\w{\diN G\cong Y} in \w[.]{\ho\Pud{n}{k}}

For \w[,]{k=1} consider the simplicial \wwb{n-1}fold groupoid
\w[.]{\Nup{n}G} Applying the classifying space functor
\w{\diN\col \Gpd^{n-1}\to\Top} in each simplicial dimension yields a
simplicial space \w[.]{\Yd=(\odiN{n}\Nup{n}G)\sb{\bullet}} Thus
\w{Y_{0}=\diN(\Nup{n}G)_{0}\up{n}} is contractible, and the Segal
maps for \w{\Yd} are isomorphisms (since  \w{\Nup{n}\Gd\up{n}} is the
nerve of an internal groupoid), hence in particular geometric
weak equivalences.

As $G$ is weakly globular, applying the functor \w{\Twg{n}} of
\S \ref{dtndn} yields a groupoid, and \w[.]{\pi_{0}\Yd=\cN\Twg{n}G}
Since \w{Y_{0}} is contractible, \w{\pi_{0}\Yd} is the nerve
of a group. Thus \w{Y_{1}} has a homotopy inverse (cf.\ \cite[(6.3,4)]{DolP}),
so it follows from \cite[Proposition 1.5]{SegCC} that
\w[.]{Y_{1}\simeq\Omega\|\Yd\|} That is,
$$
\diN G\up{n}_{1}~=~\diN\Wlo{n,1}G~\simeq~\Omega\diN G~.
$$
Since \w{\Omega\diN G\cong\Omega Y} in \w[,]{\ho\Pud{n-1}{k-1}} it follows
that \w{\diN\Wlo{n,1}G\cong \Omega Y} in \w[.]{\ho\Pud{n-1}{k-1}}
In the induction step, let
$$
H~:=~\Wlo{n,1}G~=~(\Nup{n}G)\up{n}_{1}
$$
\noindent in \w[,]{\Gptk{n-1,k-1}} where
by the inductive hypothesis \w{\diN\Wlo{n-1,k-1}H\cong\Omega^{k-1}\diN H}
in \w[.]{\ho\Pud{n-k}{}} By what we have shown above for \w{k=1} we have
\w[.]{\diN H=\diN(\Nup{n}G)\up{n}_{1}\simeq\Omega\diN G} It follows that
there are isomorphisms in \w{\ho\Pud{n-k}{}}

\medskip
$\; \quad\diN\Wlo{n,k}G~=~\diN\Wlo{n-1,k-1}H~\cong~\Omega\sp{k-1}\diN H~
\cong~\Omega\sp{k-1}(\Omega\diN G)=\Omega\sp{k} \diN G~. $
\end{proof}

\supsect{\protect{\ref{cappl}}.B}{Further directions}
One  motivation in constructing our
model for $n$-types was to obtain useful algebraic approximations of
homotopy theories \wh that is, of simplicially enriched categories.

Recall that if \w{\lra{\cV,\otimes,I}} is any monoidal category, we denote by
\w{\VC} the collection of all (not necessarily small)
$\cV$-categories, that is, categories enriched in $\cV$ (see
\cite[Vol.~II, \S 6.2]{BorcH}).
We obtain further variants by applying any (strictly) monoidal functor
\w{P\col \lra{\cV,\otimes}\to\lra{\cV',\otimes'}} to a \ww{\cV}-category
$\cC$. For example, given a simplicially enriched category \w[,]{\Xd} for each
\w{n\geq 1} we have a \ww{\PS{n}}-category \w[,]{\Yd:=\Po{n}\Xd} in which each
mapping space \w{\Yd(a,b)} is the $n$-th Postnikov section \w[.]{\Po{n}\Xd(a,b)}

\begin{mysubsection}{$\mathbf{n}$-Track categories}\label{sntc}
For \w[,]{n\geq 2} an \emph{$n$-track category} is a category enriched
in weakly globular $n$-fold groupoids \w[,]{(\Gpt{n},\times)} with
respect to the cartesian monoidal structure. The category of $n$-track
categories is denoted by \w[.]{\Track{n}}
\end{mysubsection}

Since \w{\Qlo{n}\col \sS\to\Gpt{n}} preserves products (see \S
\ref{reprod}), it induces a functor
\begin{equation*}
S\lo{n}\col \sS\text{-}\Cat~\lto~\Track{n}
\end{equation*}
\noindent from simplicial categories to $n$-track categories. Furthermore,
the functors \w{\bPz{n}\col \Gpt{n}\to\Gpt{n-1}} giving the
Postnikov decomposition of \w{\Gpt{n}} induce functors
\begin{equation*}
\Po{n-1}\col \Track{n}\to\Track{n-1}
\end{equation*}
\noindent providing the Postnikov decomposition of simplicially
enriched categories.

For \w[,]{n=1} the corresponding $k$-invariant was
described in \cite{BWirC} in terms of the Baues-Wirsching
cohomology of categories, and a similar result was obtained in
\cite{BPaolT} for \w[,]{n=2} using an algebraically-defined cohomology
of track categories. The extension of this to general $n$
via an appropriate cohomology of \wwb{n-1}track categories will be investigated
in the future.

\begin{mysubsection}{Spectral sequences}\label{sss}
In \cite{BBlaS}, the authors introduced the notion of the
\emph{Postnikov $n$-stem} \w{\cP{n}X} of a topological space $X$ \wh
that is, the system of \wwb{k-1}connected \wwb{n+k}Postnikov sections
\w{\Po{n+k}X\lra{k-1}} \wb[,]{k=0,1,\dotsc} with the natural maps between them.

They then show that the \ww{E^{n+2}}-term of the homotopy spectral
sequence of a (co)simplicial space \w{\Wd} (respectively, \w[)]{\Wu}
 depends only on the simplicial $n$-stems \w{\cP{n}\Wd} or
 \w[.]{\cP{n}\Wu} Thus we can in principle use the \wwb{n+k}fold
 groupoid models of each \w{W_{m}} or \w[,]{W^{m}} as in \S
 \ref{dcont} to extract information about the \ww{d^{n+1}}-differentials.

 However, in many cases of interest \wh including the (stable or
 unstable) Adams spectral sequence, the Eilenberg-Moore spectral
 sequence, and others \wh a more ``algebraic'' approach can be used,
 using the notion of \emph{$n$-th order derived functors} introduced
 in \cite{BBlaH}.

For example, the (unstable) \ww{\FFp}-Adams spectral sequence for a
(simply connected) space $X$  constructed in \cite{BKanH} is the
homotopy spectral sequence of a cosimplicial space
\w{\Wu} obtained as a \ww{\FFp}-resolution of $X$.
It can be shown that the \ww{E^{n+2}}-term of this spectral sequence
depends only on the $n$-Postnikov sections of the mapping spaces
\w{\map(X,E)} and \w{\map(E,E')} for various products of
\ww{\FFp}-Eilenberg-Mac~Lane space $E$ and \w[.]{E'}
Thus we do not need a full algebraic model for the \ww{\PS{n}}-category
\w[,]{\Top} but only for the small subcategory with objects $X$ and
$E$ as above. Since all mapping spaces in this category are themselves
simplicial \ww{\FFp}-vector spaces, the associated $n$-track
category is correspondingly simplified.
The case \w{n=1} was treated in great detail in \cite{BauAS},
and some progress on the case \w{n=2} has been made in \cite{BFran}. However,
it is clear from \cite{BBlaHD} that a better conceptual framework, such
as an algebraic model for such ``linear'' $n$-track categories, will be needed
before any further progress can be made for \w[.]{n\geq 2}
\end{mysubsection}
\newpage
%
%
\section*{Appendix:  \ Fibrancy conditions on $n$-fold simplicial sets}
\label{afibgpd}

In this appendix we prove some technical facts about \w[:]{\osl{n}}

\begin{remark}\label{rfunor}
Given a map of simplicial sets \w{f\col A\to B} and \w[,]{m\geq 2} let
\w[,]{P:=\osl{m}A} \w[,]{Q:=\osl{m}B} and \w[.]{F=\osl{m}f\col P\to Q}
From the description in \S \ref{sorsum} we see by induction on $m$
(using \wref[)]{eqovt} that for every multi-index \w{(p\sb{1}\dotsc p\sb{m})}
the map of sets
\w{F\sb{(p\sb{1}\dotsc p\sb{m})}\col P\sb{(p\sb{1}\dotsc p\sb{m})}\to
Q\sb{(p\sb{1}\dotsc p\sb{m})}}
is simply \w[,]{f_{\ell}\col A_{\ell}\to B_{\ell}} for
\w{\ell:=m-1+p\sb{1}+\dotsc p\sb{m}} (cf.\ \wref[).]{eqorn}
\end{remark}

\begin{lemma}\label{lordec}
If \w{Y=\osl{n}X\in\Sx{n}} for some \w{X\in\sS} and \w[,]{n\geq 2} then for
any two of its $n$ directions \w[,]{1\leq p\neq q\leq n} the lower right corner
of the bisimplicial set \w{Z=Y^{(p,q)}\in\Sx{2}} (see \S \ref{nsimp}(b))
has the form:
\vspace{-3mm}
\mydiagrm[\label{eqsquare}]{
&& X_{s+2} \ar@<2ex>[rr]^{d_{k+3}} \ar@<-1ex>[rr]^{d_{k+2}}
     \ar@<4ex>[d]^{d_{i+2}} \ar[d]^{d_{i+1}}\ar@<-2.5ex>[d]^{d_{i}} &&
X_{s+1}  \ar@<4ex>[d]^{d_{i+2}} \ar[d]^{d_{i+1}}\ar@<-2.5ex>[d]^{d_{i}} \\
X_{s+2} \ar@<2.5ex>[rr]^{d_{k+3}} \ar[rr]^{d_{k+2}}\ar@<-2ex>[rr]^{d_{k+1}}
     \ar@<0.5ex>[d]^{d_{i+1}} \ar@<-0.5ex>[d]_{d_{i}} &&
X_{s+1} \ar@<1ex>[rr]^{d_{k+2}} \ar@<-1ex>[rr]^{d_{k+1}}
     \ar@<0.5ex>[d]^{d_{i+1}} \ar@<-0.5ex>[d]_{d_{i}} &&
X_{s}       \ar@<0.5ex>[d]^{d_{i+1}} \ar@<-0.5ex>[d]_{d_{i}} \\
X_{s+1} \ar@<2.5ex>[rr]^{d_{k+2}} \ar[rr]^{d_{k+1}}\ar@<-2ex>[rr]^{d_{k}} &&
X_{s} \ar@<1ex>[rr]^{d_{k+1}} \ar@<-1ex>[rr]^{d_{k}} && X_{s-1}
}
\noindent for some \w{s\geq n} and \w[.]{0\leq i<k<s}
\end{lemma}

\begin{proof}
By induction on \w[,]{n\geq 2} where the case \w{n=2} is depicted in Figure
\ref{for} of Section \ref{cfng}. Using \wref[,]{eqovt} we see that
\w[,]{Y=\ovt{2}{n-1}\osl{2}X} so if we number the $n$ directions of $Y$ so
as to the start with the horizontal direction of \w[,]{\osl{2}X}
then for any \w{1< p\neq q\leq n}
the bisimplicial set \w{Z=Y^{(p,q)}\in\Sx{n}} is contained in the \wwb{n-1}fold
simplicial set \w[,]{\osl{n-1}Q_{t\bullet}} for one of the vertical simplicial
sets of \w[.]{Q:=\osl{2}X} Thus the claim for such a $Z$ follows by the
induction hypothesis.

Thus it suffices to treat the case \w[.]{1=p<q}
Since the corresponding \emph{vertical} maps in each of the vertical
simplicial sets \w[,]{Q_{t\bullet}} for various $t$, have the same labels
(in terms of the original face maps of $X$), the same will be true after
applying the functor \w{\osl{n-1}} to each of them. This implies that the
vertical maps in \wref{eqsquare} are indeed both labelled \w[,]{d_{i},d_{i+1}}
for some \w[.]{i<k+1} However, since each of the simplicial sets \w{Q_{t\bullet}}
is obtained by repeated applications of \w{\Dec} to $X$ (see \wref[),]{eqorder}
we must have omitted at least the maximally labelled face map
\w[,]{d_{k+1}\col X_{k+1}\to X_{k}} by definition of \w[.]{\Dec} Therefore, among
the various face maps of $X$ appearing in \w[,]{\osl{n-1}Q_{t\bullet}} the map
\w{d_{k+1}\col X_{k+1}\to X_{k}} cannot appear. Thus, we actually have \w[.]{i<k}

From Figure \ref{for} (or from the fact that the bisimplicial set $Q$, as
a (vertical) simplicial object over \w[,]{\sS} is the resolution of $X$ produced
by the comonad \w[),]{\Decp} we see that the horizontal maps in $Q$ are always
the face maps of \emph{maximal} consecutive indices for any given
\w[:]{Q_{i,j}=X_{i+j+1}} e.g., the bottom left horizontal maps in Figure \ref{for}
are \w[.]{d_{1},d_{2}\col X_{2}\to X_{1}}

On the other hand, by Remark \ref{rfunor} (for \w[),]{m=n-1} the two pairs of
\emph{horizontal} maps in \wref{eqsquare} are just those that appear in
the rightmost sequence of horizontal maps in Figure \ref{for}:  namely,
\w{d_{k},d_{k+1}\col X_{k+1}\to X_{k}} and \w[.]{d_{k+1},d_{k+2}\col X_{k+2}\to X_{k+1}}
Thus when \w[,]{p=1} in fact \w{k=s-1} in \wref{eqsquare} (as for the front
and back squares in Figure \ref{fort}).
\end{proof}

%
%
\begin{propa}[\protect{\ref{pntfibor}}]
If \w{X\in\sS} is a Kan complex, then \w{Y:=\osl{n} X} is \wwb{n,2}fibrant.
\end{propa}

\begin{proof}
For every \w[,]{1\leq p\leq n} the simplicial set \w{Y^{(p)}} is obtained from
$X$ by repeated applications of \w{\Dec} and \w[,]{\Decp} so it is still
a Kan complex, and the same is true of \w[.]{\csk{2}Y^{(p)}}

For each bisimplicial set of the form \wref[,]{eqsquare} denote by $W$ and $Z$
the middle and right vertical simplicial sets, respectively, with
\w{\phi\col W\to Z} the horizontal map in \w{\sS} given by
\w[,]{d_{k}\col W_{0}=X_{s}\to Z_{0}=X_{s-1}} \w[,]{d_{k+1}\col W_{1}=X_{s+1}\to Z_{1}=X_{s}}
and so on.
Similarly, denote by $U$ and $V$ the middle and bottom horizontal simplicial
sets, respectively, with \w{\psi\col U\to V} the vertical map in \w{\sS} given by
\w{d_{i}\col U_{j}=X_{s+j}\to V_{j}=X_{s+j-1}} for all \w[.]{j\geq 0}

By Definition \ref{dfibrant} and Lemma \ref{lordec}, in order to verify that
$Y$ is \wwb{n,2}fibrant, we must check that \w{\csk{2}\phi} and
\w{\csk{2}\psi} are fibrations for any choice of \wref{eqsquare} with
\w[.]{i<k} This means that we must show that a lifting $\widehat{g}$ exists
for every solid commuting square of one of the two following forms:
\mydiagram[\label{eqlift}]{
\Lambda^{j}[m] \ar[d]^{i_{j}} \ar[rr]^{f} && U\ar[d]^{\psi}&&&
\Lambda^{j}[m] \ar[d]^{i_{j}} \ar[rr]^{f} && W\ar[d]^{\phi}\\
\Delta[m]\ar[rr]^{g} \ar@{.>}[rru]^{\hat{g}} && V &&&
\Delta[m]\ar[rr]^{g} \ar@{.>}[rru]^{\hat{g}} && Z\\
}
\vspace{-4mm}
$$
\xymatrix @R=2pt{&(a)&&&&&&(b)}
$$
\noindent for \w{m=1,2} and \w{0\leq j\leq m} (where
\w{\Lambda^{j}[m]\subseteq\partial\Delta[m]} consists of all but the $j$-th
face of \vsm\w[,]{\Delta[m]} and \w{i_{j}\col \Lambda^{j}[m]\hra\Delta[m]} is
the inclusion).

\noindent\textbf{Case I:} \ \ When \w{m=1} in \wwref{eqlift}(a), the map
\w{f\col \Lambda^{j}[1]\to U} \wb{j=0,1} corresponds to a $0$-simplex
\w{\tilde{\sigma}\in U_{0}} \wwh that is, an $s$-simplex \w{\sigma\in X_{s}}
(since \w{U_{0}=X_{s}} and  \w[),]{\Lambda^{j}[1]\cong\Delta[0]}
and the map \w{g\col \Delta[1]\to V} corresponds to a $1$-simplex
\w{\tilde{\tau}\in V_{1}} \wwh that is, an $s$-simplex \w[.]{\tau\in X_{s}}

Commutativity of the solid square in \wwref{eqlift}(a) \wh that is,
\w{\psi\circ f=g\circ i_{j}} \wwh means that
\w[,]{d_{j}^{V_{1}}(\tilde{\tau})=\psi(\tilde{\sigma})} i.e.,
\begin{myeq}[\label{eqoneverta}]
d_{k+j}^{X_{s}}\,\tau~=~d_{i}^{X_{s}}\,\sigma~.
\end{myeq}

A lift \w{\hat{g}\col \Delta[1]\to U} corresponds to a $1$-simplex
\w{\tilde{\omega}\in U_{1}} \wwh that is, an \wwb{s+1}simplex
\w[,]{\omega\in X_{s+1}} and commutativity of the two triangles in
\wwref{eqlift}(a) translates into the two conditions
\w{d_{j}^{U_{1}}(\tilde{\omega})=\tilde{\sigma}} and
\w[,]{\psi(\tilde{\omega})=\tilde{\tau}} that is:
\begin{myeq}[\label{eqonevertb}]
d_{k+1+j}^{X_{s+1}}\,\omega~=~\sigma\hspace*{8mm}\text{and}\hspace*{5mm}
d_{i}^{X_{s}}\,\omega~=~\tau~.
\end{myeq}
Combining \wref{eqoneverta} and \wref{eqonevertb} yields the simplicial
identity:
\begin{myeq}[\label{eqonevertc}]
d_{i}\,d_{k+1+j}\,\omega~=~d_{k+j}\,d_{i}\,\omega~,
\end{myeq}
\noindent since \w[.]{i<k}

The two $s$-simplices $\sigma$ and $\tau$ satisfying \wref{eqoneverta}
define a map from the following pushout $P$ in \w[:]{\sS}
$$
\xymatrix@R=25pt{
\ar @{} [drr]|<<<<<<<<<<<<<<<<{\framebox{\scriptsize{PO}}}
\Delta[s-1] \ar[d]_{\eta_{k+j}} \ar[rr]^{\eta_{i}} &&
\Delta[s] \ar[d] \ar@/^2pc/[ddr]^{\sigma} \\
\Delta[s] \ar[rr] \ar@/_2pc/[drrr]_{\tau} && P \ar@{.>}[rd]^{(\sigma,\tau)} & \\
&&& X
}
$$
\noindent Since $P$ is a union of two $s$-simplices along a common face, it
is a contractible subspace of \w[,]{\Delta[s+1]} so \w{P\hra\Delta[s+1]} is
an acyclic cofibration in \w[.]{\sS} Because $X$ is fibrant, a lift
\w{\omega\col \Delta[s+1]\to X} for \w{(\sigma,\tau)} \wwh and thus
\w{\hat{g}\col \Delta[m]\to U} \wwh always exists\vsm.

\noindent\textbf{Case II:} \ \ When \w{m=2} in \wwref{eqlift}(a), the map
\w{f\col \Lambda^{j}[2]\to U} \wb{j=0,1,2} corresponds to a pair of
$1$-simplices \w{\tilde{\alpha},\tilde{\beta}\in U_{1}} with
\w[,]{d_{p}\tilde{\alpha}= d_{q}\tilde{\beta}} where
\begin{myeq}[\label{eqpq}]
(p,q)=\begin{cases}
(1,1) & \text{if}\ j=0\\
(0,1) & \text{if}\ j=1\\
(0,0) & \text{if}\ j=2~.
\end{cases}
\end{myeq}
\noindent This means that we have \w{\alpha,\beta\in X_{s+1}=U_{1}} with
\begin{myeq}[\label{eqtwovertz}]
d_{k+1+p}^{X_{s+1}}\,\alpha~=~d_{k+1+q}^{X_{s}}\,\beta~.
\end{myeq}

The map \w{g\col \Delta[2]\to V} corresponds to \w[,]{\sigma\in X_{s+1}=V_{2}} and
the map \w{\hat{g}\col \Delta[2]\to U} corresponds to \w[.]{\omega\in X_{s+2}}

Commutativity of the solid square in \wwref{eqlift}(a) means that
\begin{myeq}[\label{eqtwoverta}]
d_{k+p}^{X_{s+1}}\,\sigma~=~d_{i}^{X_{s}}\,\alpha\hspace*{8mm}\text{and}\hspace*{5mm}
d_{k+q}^{X_{s+1}}\,\sigma~=~d_{i}^{X_{s}}\,\beta~.
\end{myeq}
\noindent Commutativity of the upper triangle in \wwref{eqlift}(a) means:
\begin{myeq}[\label{eqtwovertb}]
d_{k+1+p}^{X_{s+2}}\,\omega~=~\alpha\hspace*{8mm}\text{and}\hspace*{5mm}
d_{k+1+q}^{X_{s+2}}\,\omega~=~\beta~,
\end{myeq}
\noindent and commutativity of the lower triangle in \wwref{eqlift}(a) means:
\begin{myeq}[\label{eqtwovertc}]
d_{i}^{X_{s+1}}\,\omega~=~\sigma~.
\end{myeq}

Combining \wref[,]{eqtwoverta} \wref[,]{eqtwovertb} and \wref{eqtwovertc}
yields the two simplicial identities:
\begin{myeq}[\label{eqtwovertd}]
d_{i}\,d_{k+1+p}\,\omega~=~d_{k+p}\,d_{i}\,\omega\hspace*{8mm}
\text{and}\hspace*{5mm}
d_{i}\,d_{k+1+q}\,\omega~=~d_{k+q}\,d_{i}\,\omega~.
\end{myeq}
\noindent since \w[.]{i<k} The existence of $\omega$ follows as above\vsm .

The analogous cases for \wwref{eqlift}(b) are obtained from these by applying
the inversion \w{I^{\ast}} of \S \ref{rinvolution}.
\end{proof}

\end{document}